\documentclass[english,12pt,oneside]{article}
\usepackage[T1]{fontenc}
\usepackage[margin=1in]{geometry}

\usepackage{ZRXmacrosYifan}
\usepackage{authblk} 
\usepackage[toc,page]{appendix}

\usepackage[authoryear]{natbib}
\usepackage{authblk}

\usepackage[final]{showlabels}

\numberwithin{figure}{section}
\numberwithin{equation}{section}

\usepackage{tikz} 
\usepackage{pgf}
\usetikzlibrary{arrows,automata}

\theoremstyle{plain}
      \newtheorem{thm}{Theorem}[section] 
      \newtheorem{cor}{Corollary}[thm] 
      \newtheorem{lem}[thm]{Lemma}
      \newtheorem{prop}[thm]{Proposition} 
      \newtheorem{claim}{Claim}[section]

\theoremstyle{definition}
      \newtheorem{defn}[thm]{Definition}    
      \newtheorem{eg}[thm]{Example}
      
\theoremstyle{remark}
      \newtheorem{rem}{Remark}[thm]



\begin{document}
\title{Semi-$G$-normal: a Hybrid between Normal and $G$-normal (Full Version)}
\author[1,*]{Yifan Li}
\author[1]{Reg Kulperger}
\author[1]{Hao Yu}
\affil[1]{School of Mathematics and Statistics, University of Western Ontario, London, Canada}
\affil[*]{E-mail: yli2763@uwo.ca}

             

\date{(Detailed research work for conference and open discussions)} 



\setcounter{Maxaffil}{0}
\renewcommand\Affilfont{\itshape\small}




\maketitle

\begin{abstract}
	
The $G$-expectation framework is a generalization of the classical probabilistic system motivated by Knightian uncertainty, where the $G$-normal plays a central role. However, 
from a statistical perspective, $G$-normal distributions look quite different from the classical normal ones. For instance, its uncertainty is characterized by a set of distributions which covers not only classical normal with different variances, but additional distributions typically having non-zero skewness. The $G$-moments of $G$-normals are defined by a class of fully nonlinear PDEs called $G$-heat equations. To understand $G$-normal in a probabilistic and stochastic way that is more friendly to statisticians and practitioners, we introduce a substructure called semi-$G$-normal, which behaves like a hybrid  between normal and $G$-normal: it has variance uncertainty but zero-skewness. We will show that the non-zero skewness arises when we impose the $G$-version sequential independence on the semi-$G$-normal. More importantly, we provide a series of \emph{representations} of random vectors with semi-$G$-normal marginals under various types of \emph{independence}. Each of these representations under a typical \emph{order of independence} is closely related to a class of \emph{state-space volatility models} with a common \emph{graphical} structure.  
In short, semi-$G$-normal gives a (conceptual) transition from classical normal to $G$-normal, allowing us a better understanding of the distributional uncertainty of $G$-normal and the sequential independence. 
%
%






\end{abstract}




\newpage

\tableofcontents

\newpage

\section{Introduction}


The $G$-expectation framework is a new generalization of the classical probabilistic system, which is aimed at dealing with random phenomena in dynamic situations where it is hard to precisely determine a unique probabilistic model. These situations are also closely related to the long-existing concern on model uncertainty in statistical practice. For instance, \cite{chatfield1995model} gives an overview of this concern itself. 
However, how to better connect the idea of this framework with general data practice is still a developing and challenging area that requires researchers and practitioners from different backgrounds to collaborate and reflect on different degrees of uncertainty possibly brought by complicated nature of the data but also from the modeling procedure itself. To give some examples (rather than a complete list), several recent attempts have been made by  \cite{pei2021worst,peng2020improving,peng2020hypothesis,xu2019nonlinear,Li2018Stat} and \cite{jin2016optimal} (which has been published as \cite{jin2021optimal}).
A fundamental and unavoidable problem will be how to better understand the $G$-version distributions and independence from a statistical perspective, which also requires long-term efforts of learning, thinking and exploration. 
This research work can be treated as a detailed systematic report of our exploration on this basic point in the past three years to a broad community. This community includes not only experts in the area of nonlinear expectations (such as the $G$-expectation) but also researchers and practitioners from other related fields who may not be familiar with the theory of the $G$-expectation framework ($G$-framework) but are interested in the interplay between their areas and the $G$-framework,
which requires them to properly understand the meanings and potentials of $G$-version distributions and independence. 
One vision of this report is to 
explore and understand the role of statistical methods incorporating $G$-version distributions or processes (with its own independence)
in general data practice as well as the differences and connections with the existing classical methods. More importantly, we intend to show how we can broaden our horizon of questions we are able to consider by introducing the notions (such as the distributions and independence) in the $G$-framework (this goal has been partially indicated in \ref{subsec:inference-ssvm}).  
This report is also used to discuss with the broad community
to initiate an in-depth discussion on this subject. 
Considering the length and scope of this work, we decide to divide our core discussions into two stages. 
The first stage (\emph{this} paper) can be treated as a \emph{theoretical preparation} for the second stage (forming a companion of this paper) which provides a series of \emph{statistical data experiments} based on the theoretical results in this paper. 



The main objective of this paper is to provide a better interpretation and understanding of the $G$-normal distribution and the $G$-version independence designed for researchers and practitioners from various background who are familiar with classical probability and statistics.
 We will achieve this goal by introducing a new substructure called the semi-$G$-normal distribution, which behaves like a hybrid connecting normal and $G$-normal: it is a typical object with distributional uncertainty that preserves many properties of classical normal but is also closely related to the $G$-normal distribution. 







In any probabilistic framework, if there exists a ``normal'' distribution (or an equivalent distributional object),
it should play a fundamental role in the system. How to understand
and deal with the normal distribution is crucial for the development
of this framework. The $G$-normal distribution, as its classical analogue, plays a central and fundamental role in the development of the $G$-expectation framework. 

\subsection{Introduction to the $G$-expectation framework}

First we give general readers a short introduction to the $G$-expectation framework. 
The classical probabilistic system is good at describing the randomness under a
\emph{single} probability rule or model $\lprob_\theta$ (which could be sophisticated in its form).
However, in practice, there are phenomena where it is hard to precisely determine an unique $\lprob_\theta$ to describe the randomness. In this case,
we cannot ignore the uncertainty in the probability rule itself. 
This kind of uncertainty is often called \emph{Knightian uncertainty} in economy (\cite{knight1921risk}) or \emph{epistemic uncertainty} in statistics (\cite{der2009aleatory}). It is also commonly called \emph{model uncertainty} if it refers to the uncertainty in the probabilistic model. A standard example of Knightian uncertainty comes from the Ellsberg paradox proposed by \cite{ellsberg1961risk} showing the violation of the classical expected utitlity theory based on a linear expectation. 
In this case, we essentially need to work with a \emph{set} $\myset{P}=\{\lprob_\theta,\theta\in\Theta\}$ of probability measures.
In order to quantify the extreme cases under $\myset{P}$, we need to work on a sublinear expectation $\expt$ defined as: 
\begin{equation}
\label{eq:expt-represent}
	\expt[\cdot]\coloneqq\sup_{\lprob\in \myset{P}}\lexpt_{\lprob}[\cdot].
\end{equation}
This sublinear expectation defined as \ref{eq:expt-represent} first appeared as the upper prevision in \cite{huber2004robust}. We also call \ref{eq:expt-represent} as a \emph{representation} of $\expt$. Coherent risk measures proposed by \cite{artzner1999coherent} can be also represented in this form and more details can be found in \cite{follmer2011stochastic}. The notion of Choquet expectation (\cite{choquet1954theory}) is another special type of sublinear expectation which is foundation of a new theory of expected utility by \cite{schmeidler1989subjective} to resolve the Ellsberg paradox in static situation. For dynamic situation, the utility theory can be developed by the sublinear version of $g$-expectation proposed by \cite{chen2002ambiguity}. In principle, $g$-expectation can only deal with those dynamic situations where we can find a reference measure $\measure{Q}$ to dominate $\myset{P}$.
Nonetheless, this situation is ideal for technical convenience but also quite restrictive compared with reality: it means all the probabilities in $\myset{P}$ agree on the same null events. 
For instance, in the context of financial modeling, when there is (Knightian) uncertainty or ambiguity in the \emph{volatility} process $\sigma_t$, the set $\myset{P}$ may not necessarily have a reference measure (\cite{epstein2013ambiguous}). How should we deal with a possibily non-dominated $\myset{P}$ in dynamic situation? It took the community many years to realize that it is necessary to jump out of the classical probability system and start from scratch to construct a new generalization of probability framework, which was established by \cite{peng2004filtration,peng2007g,peng2008multi} and further developed by the academic community led by him, called the $G$-expectation framework.



Since its establishment in 2000s, the $G$-expectation framework is gradually developed into a new generalization of the classical one with its own notion of \emph{independence}
and \emph{distributions}, as well as the associated stochastic calculus. 
The spirit of considering $\myset{P}$ to characterize the Knightian uncertainty is embedded into this framework from its initial setup. 
A distribution under the $G$-expectation can be represented by a family of classical distributions - it provides a convenient way to depict the distributional uncertainty requiring a infinitely dimensional family of distributions which usually may not have an explicit parametric form.
More details about this framework can be found in \cite{denis2011function,peng2017theory,peng2019nonlinear}.


The $G$-normal distribution $\GN(0,\varI)$ is the analogue of normal $\CN(0,\sigma^2)$ in this framework. As indicated by its notation, it is a typical object with variance uncertainty. In theory, it plays a central role in the context of the central limit theorem (\cite{peng2019law}): it is the asymptotic distribution of the normalized sum of a series of \emph{independent} random variables with zero mean and variance uncertainty. 
It has a Stein-type characterization provided by \cite{hu2017stein}. \cite{fang2019limit} provides a insightful discrete approximation and continuous-time form representation of the $G$-normal distribution. 
In pracitice, $G$-normal has also shown its potentials in the study of risk measure such as the Value at Risk induced by $G$-normal ($G$-VaR) rigorously constructed in \cite{peng2020improving} and further developed in recent \cite{peng2020autoregressive}, where the $G$-VaR has mostly outperformed the benchmark methods in terms of the violation rate and predictive performance. 



\subsection{Potential misunderstandings on the $G$-normal and independence}

Since the notion of distribution and independence in the $G$-expectation framework are different from the classical setup, there are several potential misunderstandings on the interpretation of $G$-normal and independence. 
The sources of these misunderstandings can be summarized into the following four aspects (where we have also provided clarification if applicable):
\begin{enumerate}
	\item[A1](The uncertainty set of the $G$-normal) 
	The $G$-expectation of the $G$-normal is defined by the (viscosity) solution a fully nonlinear PDE (the $G$-heat equation), which usually does not have an explicit form unless in some special cases (\cite{hu2009explicit}). In fact, following the spirit of Knightian uncertainty, a better interpretation of the $G$-version distribution is a family of classical distributions characterizing the distributional uncertainty. Nonetheless, for a general reader, if not careful, the notation of $G$-normal $\GN(0,\varI)$ could lead to a misconception
	that it is associated with the family $\{\CN(0,\sigma^2),\sigma\in\sdInt\}$. Although this impression may still hold in special situations as shown in \ref{thm:conn-G}, it is not rigorous in general. 
	Actually, the uncertainty set of $\GN(0,\varI)$ is much larger than this and one evidence is that the $G$-normal distribution has \emph{third-moment uncertainty} (all its odd moments have uncertainty), but all the distributions in the family $\{\CN(0,\sigma^2),\sigma\in\sdInt\}$ are symmetric implying zero third moments. 
	It means that the uncertainty set of the $G$-normal contains those classical elements that has non-zero third moment. It seems like a strange property for a ``normal'' distribution in a probabilistic system (especially when we note that $X\eqdistn -X$ if $X\sim \GN(0,\varI)$). An explicit form of the uncertainty set of $G$-normal is given by \cite{denis2011function}.
	\item[A2](The missing connection between univariate and multivariate $G$-normal) The joint random vector formed by $n$ independent $G$-normal distributed random variables does not follow a multivariate $G$-normal (even under any invertible linear transformation of the original vector.) More study of the counter-intuitive properties of $G$-normal can be found in \cite{bayraktar2015normal}.
	\item[A3](The asymmetry of independence) The independence in this framework is \emph{asymmetric}: $X$ is independent of $Y$ does not necessarily mean $Y$ is independent of $X$. This is why this independence is also called \emph{sequential independence}, which is different from the classical one. 
	One interpretation of this asymmetry in the relation ``$Y$ is independent of $X$'' is from the temporal order: if $Y$ is realized at a time point \emph{after} $X$, the roles between $X$ and $Y$ would be asymmetric (in terms of the possible dependence structure). Another interpretation is from the distributional uncertainty: any realization of $X=x$ has no effect on the uncertainty set of $Y$. Both of the interpretations are valid if one understands the detailed theory of this framework. However, for general audience, both of these are still vague even become quite confusing if one combines them together in a naive way (such as ``if $Y$ happens after $X$, any realization of $Y$ should have no effect on $X$, then we automatically have one way of independence.''). 
	So far we do have a simple example that the independence is \emph{indeed} asymmetric (Example 1.3.15 in \cite{peng2019nonlinear}), but it is not clear \emph{why} the independence is asymmetric in this example. To be specific, how does the distributional uncertainty of the joint $(X,Y)$ (or the \emph{representation} of its sublinear expectation) change if we switch the order of the independence? Such a representation (even in a special case) will be beneficial for general audience to better understand the sequential independence in the sense that they can explicitly see how the order of the independence changes the underlying distributional uncertainty.
	\item[A4](The lack of caution before the data analysis) Suppose one intends to use the $G$-normal distribution to describe the distributional uncertainty from a dataset (either artificial or realistic one). Without enough caution, the misinterpretations of the independence and distribution in this framework mentioned above may further bring confusion or even mistakes to the data analyzing procedure. 
	\end{enumerate} 
	



The objectives of this paper all serve for this central problem: from a statistical perspective, how to better understand the $G$-normal distribution
and the $G$-version independence. The answer to this question will also lead to a better interpretation and understanding of $G$-normal distribution for general audience and practitioners who are familiar with classical probability and statistics. 
We will work towards this central problem from the following four basic questions where each question ``Q[k]'' is corresponding to one of the aspect ``A[k]'' mentioned above,
\begin{enumerate}
	\item[Q1] How does the third-moment uncertainty of $G$-normal arise? Is this possible to use the linear expectations of classical normal to approach the sublinear expectation of $G$-normals (without involving the underlying PDEs)? 
	\item[Q2] How should we appropriately connect the univariate objects and multivariate objects in this framework? Since it is hard to start from univariate $G$-normals to get multivariate $G$-normal, is this possible for us to make a retreat at the starting point, that is, to connect univariate classical normals with a multivariate $G$-normal? 
	\item[Q3] How can we understand the asymmetry of the independence in this framework in terms of representations? 
	\item[Q4] What kinds of data sequence are related to the volatility uncertainty covered by $G$-normal and what are not? 
\end{enumerate} 

The interpretation of $G$-normal and sequential independence will also be important to theoretically investigate the reliability and robustness of risk measure derived from $G$-version distributions such as the current $G$-VaR in the literature. 


\subsection{Our main tool and results in this paper}

Our main tool here is a substructure called the semi-$G$-normal distribution (\ref{subsec:semi-G-normal-intro}),
which behaves like a close relative to both classical models (such as a normal mixture model) 
and also $G$-version objects (such as a $G$-normal). We will also study the various kinds of independence associated with the semi-$G$-normal distributions (\ref{subsec:indep-semignorm}). 


The notion semi-$G$-normal was first proposed in \cite{li2018iterative}, which has used it to design an iterative approximation to the sublinear expectation of $G$-normal and the solution to the $G$-heat equation. Later on this substructure was further developed in the master thesis by \cite{Li2018Stat} where the independence structures have been proposed there to better perform the pseudo simulation in this context. 

This paper gives a more rigorous and systematic construction of these structures and focus more on the distributional and probabilistic aspects of them to show the hybrid roles of semi-$G$-normal between classical normal and $G$-normal. 
To be specific, we will show that there exists a middle stage of independence sitting between the classical (symmetric) independence and the $G$-version (asymmetric) independence. It is called \emph{semi-sequential independence}, which allows the connection between univariate and multivariate object (\ref{thm:n-semi-seqind-sym}), and it is a \emph{symmetric} relation between two semi-$G$-normal objects (\ref{prop:semi-seqind-sym}). 

Moreover, we will provide a series of representations in the form similar to \ref{eq:expt-represent} associated with the semi-$G$-normal distributions and also the random vector with semi-$G$-normal marginal under various kinds of independence. Interestingly, by changing the \emph{order of the independence}, we are equivalently modifying the \emph{graphical structure} in the representation of the sublinear expectation of the joint vector. 
This idea will be shown in \ref{subsec:represent-semignorm-indep} and further studied in \ref{subsec:fine-structure}. These representations provide a more straightforward view on the order of independence in this framework, because we can see how the family of distributions is changed due to the switching of order. Under this view, we can provide a statistical interpretation of the asymmetry of sequential independence between two semi-$G$-normal objects (\ref{subsec:interpret-asymmetry-indep}).
 
Throughout this paper, we will frequently mention the representations of the distributions in the $G$-framework.
Theses representation results are crucial here, because the right hand side of representation is simply a family of classical models and its envelope will be exactly the sublinear expectation of $G$-version objects. Through an intuitive representation, a person who is familiar with classical probability and statistics is able to understand the uncertainty described by the $G$-version objects through the representation.

This remaining content of this paper is organized as follows. \ref{sec:G-frame-setup} will give a basic setup for the $G$-expectation framework for readers to check the rigorous definitions of each concept. \ref{sec:main-re} presents our main results by putting readers in a context of a classical state-space volatility models. This kind of story setup is specially helpful in the discussions of representation associated with semi-$G$-normal in \ref{subsec:represent-semignorm-indep}.
After we go through these representation results associated with this substructure, readers will find that we have already provided the answers to the four questions during the procedure. These answers will be given and elaborated in \ref{sec:hybrid-semignorm}. Finally, \ref{sec:conclude-and-ext} will summarize the whole paper and also provide more possible extensions as future developments. 
The proofs of our theoretical results will be put into \ref{sec:proofs} unless a proof is beneficial to the current discussion or is relatively short to be included in the main content.

\section{Basic settings of the $G$-expectation framework}
\label{sec:G-frame-setup}

This section gives a detailed description of the basic setup (the sublinear expectation space) of the $G$-expectation framework for general audience by starting from a set of probability measures (more rigorous treatments can be found in Chapter 6 in the book by \cite{peng2019nonlinear}). 
Another equivalent way is to start from a space of random variables and a sublinear operator (more details can be found in Chapter 1 and 2 in \cite{peng2019nonlinear}). 



For readers who may not be familiar with this setup, the following reading order is recommended as a candidate one: 
\begin{enumerate}
	\item Take a glance at the initial setup and the meaning of notations in this section (especially the notation for independence which is in \ref{defn:G-indep}); 
	\item Read through our main results (\ref{sec:main-re}) which describe the $G$-version distributions mostly using the representations in terms of classical objects;
	\item Come back to this section to check the detailed definitions (such as the connection between $G$-expectation with the solutions to a class of fully nonlinear partial differential equations).
\end{enumerate}

\subsection{Distributions, independence and limiting results}

Let $\myset{P}=\{\lprob_{\theta},\theta\in\Theta\}$ denote a set
of probability measures on a measurable space $(\Omega,\sigmafield{F})$.
Let $\lexpt_{\theta}$ denote the linear expectation under $\lprob_{\theta}$. 
Consider the following spaces: 
\begin{itemize}
\item $L^{0}(\Omega)$: the space of all $\sigmafield{F}$-measurable real-valued
functions (or the family of random variables $X:\Omega \to \numset{R}$); 
\item $\myset{H}^{*}\coloneqq\{X\in L^{0}(\Omega):\lexpt_{\theta}[X]\text{ exists for each }\theta\in\Theta\}$;
\item $\myset{H}_{p}\coloneqq\{X\in L^{0}(\Omega):\sup_{\theta\in\Theta}\lexpt_{\theta}[\abs{X}^{p}]<\infty\}$
(for $p>0$);
\item $\myset{N}_p\coloneqq \{X\in L^{0}(\Omega):\sup_{\theta\in\Theta}\lexpt_{\theta}[\abs{X}^{p}]=0\}$ (for $p>0$);
\item $\myset{N}\coloneqq \{X\in L^{0}(\Omega):\lprob_\theta(X=0)=1 \text{ for each }\theta\in\Theta\}$.
\end{itemize}
Note that for any $1\leq p\leq q<\infty$,
\[
\myset{H}_{q}\subset\myset{H}_{p}\subset\myset{H}^{*}\subset L^{0}(\Omega).
\]
We also have, for any $p>0$,
\[
\myset{N} = \myset{N}_p.
\]


\begin{defn}
(The upper expectation associated with $\myset{P}$) For any $X\in\myset{H}^{*}$,
we define a functional $\expt:\myset{H}^{*}\to[-\infty,\infty]$ associated
with the family $\myset{P}$ as 
\[
\expt[X]=\expt^{\myset{P}}[X]\coloneqq\sup_{\theta\in\Theta}\lexpt_{\theta}[X],
\]
where $[-\infty,\infty]$ is the extended real line. We also follow
the convention that, if $\lexpt_{\theta}[X]$ exists but is $\infty$
for some $\theta$, the supreme is taken as $\infty$. 
\end{defn}



\begin{defn}[The upper and lower probability]
\label{defn:upprob-lowprob}
For any $A\in \sigmafield{F}$, let 
\[
\upprob(A)\coloneqq\sup_{\lprob\in\myset{P}}\lprob(A),\text{ and }\lowprob(A)\coloneqq\inf_{\lprob\in\myset{P}}\lprob(A).
\]
The set function $\lowprob$ and $\upprob$ are respectively called
the lower and upper probabilities associated with $\myset{P}$. 
\end{defn}

\begin{prop}
\label{prop:H-star-space}The space $\myset{H}^{*}$ satisfies: 
\begin{enumerate}
\item[(1)] $c\in\myset{H}^{*}$ for any constant $c\in\numset{R}$; 
\item[(2)] If $X\in\myset{H}^{*}$, then $\abs{X}\in\myset{H}^{*}$;
\item[(3)] If $A\in\myset{F}$, then $\indicator_{A}\in\myset{H}^{*}$; 
\item[(4)] If $X\in L^{0}(\Omega)$ satisfying $\lprob_{\theta}(X\geq0)=1$ for
any $\theta\in\Theta$, then $X\in\myset{H}^{*}$;
\item[(5)] If $X\in L^{0}(\Omega)$ satisfying $\lprob_{\theta}(X\leq0)=1$ for
any $\theta\in\Theta$, then $X\in\myset{H}^{*}$;
\item[(6)] For any $X\in L^{0}(\Omega)$, $\abs{X}^{k}\in\myset{H}^{*}$ for
$k>0$. 
\end{enumerate}
\end{prop}

\begin{proof}
It is easy to check the first three properties. The logic of (4) comes from the fact that, for each $\lprob\in\myset{P}$, if we
have $X\geq0$, $\lprob$-almost surely, we must have $\lexpt_{\lprob}[X]$
exists. Similar logic can be applied to (5). The property (6) is
a direct result of (4). 
\end{proof}

\begin{rem}
	However, $\myset{H}^{*}$ is not necessarily a linear space. For instance,
let $\myset{P}=\{Q\}$ and $X$ is a Cauchy distributed random variable
under $Q$. We have $X^{+}$ and $X^{-}$ belong to $\myset{H}^{*}$,
but $X=X^{+}-X^{-}\notin\myset{H}^{*}$. 
\end{rem}

By \ref{prop:H-star-space}, for any $X\in L^{0}(\Omega)$, $\expt[\abs{X}^{p}]$
is well-defined for any $p>0$. Then we can write $\myset{H}_{p}$
as
\[
\myset{H}_{p}=\{X\in L^{0}(\Omega):\expt[\abs{X}^{p}]<\infty\}.
\]
%


We will mainly focus on the space $\myset{H}_1$. 

\begin{prop}
\label{prop:H-space}The space $\myset{H}_1$ is a linear space satisfying: 
\begin{enumerate}
\item[(1)] $c\in\myset{H}_1$ for any constant $c\in\numset{R}$; 
\item[(2)] If $X\in\myset{H}_1$, then $cX\in\myset{H}_1$ for any constant $c\in\numset{R}$; 
\item[(3)] If $X,Y\in\myset{H}_1$, then $X+Y\in\myset{H}_1$;
\item[(4)] If $X\in\myset{H}_1$, then $\abs{X}\in\myset{H}_1$;
\item[(5)] If $A\in\sigmafield{F}$, then $\indicator_{A}\in\myset{H}_1$;
\item[(6)] If $X\in\myset{H}_1$, then $\varphi(X)\in\myset{H}_1$ for any bounded
Borel measurable function $\varphi$. 
\end{enumerate}
\end{prop}

\begin{proof}
The properties here can be checked by definition of $\myset{H}_1$. For instance, (3) comes from the inequality: $\lexpt_{\theta}[\abs{X+Y}]\leq\lexpt_{\theta}[\abs{X}]+\lexpt_\theta[\abs{Y}]$ for any $\theta\in \Theta$.
\end{proof}

%
Then we can check that $\expt$ becomes a \emph{sublinear} operator
on the linear space $\myset{H}_1$. In other words, $\expt:\myset{H}_1 \to \numset{R}$ satisfies: for any $X,Y\in \myset{H}_1$,
\begin{enumerate}
\item (Monotonicity) For any $X\geq Y$, $\expt[X]\geq\expt[Y]$; 
\item (Constant preserving) For any $c\in\numset{R}$, $\expt[c]=c$;
\item (Sub-additivity) $\expt[X+Y]\leq\expt[X]+\expt[Y]$; 
\item (Positive homogeneity) For any $\lambda\geq0$, $\expt[\lambda X]=\lambda\expt[X]$; 
\end{enumerate}
Then we call $\expt$ a sublinear expectation and $(\Omega,\myset{H}_1,\expt)$ a \emph{sublinear expectation space}.

Furthermore, note that $\myset{N}=\{X\in L^{0}(\Omega):\expt[\abs{X}]=0\}$ is a linear subspace of $\myset{H}_1$. We can treat $\myset{N}$ as the null space and define the quotient space $\myset{H}_1/\myset{N}$. For any $\{X\} \in \myset{H}_1/\myset{N}$ with representative $X$, we can define $\expt[\{X\}]\coloneqq \expt[X]$, which is still a sublinear expectation. 
We can check that $\expt$ induces a Banach norm $\norm{X}_1\coloneqq \expt[\abs{X}]$ on $\myset{H}_1/\myset{N}$. Let $\hat{\myset{H}}_1$ denote the completion of $\myset{H}_1/\myset{N}$ under $\norm{\cdot}_1$. 
Since we can check that $\myset{H}_1/\myset{N}$ itself is a Banach space, it is equal to its completion $\hat{\myset{H}}_1$ (Proposition 14 in \cite{denis2011function}). 
Let $\myset{H}\coloneqq \hat{\myset{H}}_1$, then we can check that $(\Omega,\myset{H},\expt)$ still forms a sublinear expectation space. 



Rigorously speaking, we also require additional conditions on $\myset{P}$ such as the
weak compactness so that we have regularity on $\expt$ (Theorem 12 in in \cite{denis2011function}). Meanwhile,
there exists such a weakly compact family $\myset{P}$ so that the typical $G$-version distributions (maximal and $G$-normal
distribution) exist in the space $\exptSpace$. More details can be found in Section 2.3 and Section 6.2 in \cite{peng2019nonlinear}. 
The $G$-expectation is defined after we construct the Brownian motion (the $G$-Brownian motion) in this context but throughout this paper, we will only touch the $G$-version distributions and independence so the expectation $\expt$ so far is still a special kind of sublinear one, which we will still call it as the $G$-expectation to stress its typical properties allowing the existence of the $G$-version distributions. Throughout this section, without further notice, we will stay in $\exptSpace$. 



Let $\myset{H}^d\coloneqq\{(X_1,X_2,\dotsc,X_d),X_i \in \myset{H},i=1,2,\dotsc,d\}$. 
For any $X \in \myset{H}^d$, 
we will frequently mention a transformation $\varphi(X)$ of $X$ for a function $\varphi:\numset{R}^d \to \numset{R}$.
Consider the following spaces of functions: 
\begin{itemize}
	\item $\fspacedef(\numset{R}^d)$: the linear space of all bounded and Lipchistz functions;
	\item $\fspace(\numset{R}^d)$: the linear space of functions satisfying the locally Lipchistz property which means 
	\[
	\abs{\varphi(x)-\varphi(y)}\leq C_\varphi (1+\abs{x}^k+\abs{y}^k)\abs{x-y},
	\]
	for $x,y\in\numset{R}^d$, some positive integer $k$ and $C_\varphi>0$ depending on $\varphi$.
\end{itemize}
We will simply write $\varphi\in \fspacedef$ or $\varphi\in \fspace$ if the dimension of the domain of $\varphi$ is clear in the context by checking the dimension of random objects.

Note that $\myset{H}$ satisfies: for any $\varphi\in \fspacedef$, $\varphi(X)\in \myset{H}$ if $X \in \myset{H}^d$. However, this property does not necessarily hold for any $\varphi\in\fspace$. 
Therefore, when we discuss the definition of distributions and independence in this framework, we will use $\varphi \in \fspacedef$.
Later on, we will mention that this space can be extended to any $\varphi \in \fspace$ for a special family of distributions and under some additional conditions. 

\begin{defn}[Distributions] 
\label{defn:G-distn}
There several notions related to the $G$-version distributions: 
\begin{enumerate}
    \item We call $X$ and $Y$ are \emph{identically distributed}, denoted by $X\eqdistn Y$, if for any $\varphi\in \fspacedef$,
\[\expt[\varphi(X)]=\expt[\varphi(Y)].\]
  \item A sequence $\{X_n\}_{n=1}^\infty$ \emph{converges in distribution} to $X$, denoted as $X_n \convergeto{\dist} X$, if for any $\varphi\in \fspacedef$,
\[\lim_{n\to\infty}\expt[\varphi(X_n)]=\expt[\varphi(X)].\]
	
\end{enumerate}

\end{defn}


\begin{defn}[Independence]
\label{defn:G-indep}
	A random variable $Y$ \emph{is (sequentially) independent from} $X$, denoted by $X\seqind Y$, if for any $\varphi\in \fspacedef$,
	\[
	\expt[\varphi(X,Y)]=\expt[\expt[\varphi(x,Y)]_{x=X}].
	\]
\end{defn}


\begin{rem}(Intuition of this independence)
Since both $X$ and $Y$ are treated as random object with potential distributional uncertainty, 
	this independence is essentially talking about the relation between the distributional uncertainty of $X$ and $Y$. If we put our discussion into a context of \emph{sequential data} (where the order of the data matters), this kind of independence often arises in scenarios where $X$ is realized \emph{before} $Y$ and any realization of $X$ has no effect on the distributional uncertainty of $Y$.
\end{rem}

\begin{rem}(Asymmetry of this independence)
One important fact regarding this independence is that it is \emph{asymmetric}: $X\seqind Y$ ($Y$ is independent from $X$) does not necessarily mean $Y\seqind X$ ($X$ is independent from $Y$), which will be illustrated by \ref{eg:indep-ex-org}. 
This is the reason we also call it a \emph{sequential} independence and we use the notation $\seqind$ to indicate the sequential order of the independence between two random objects.	
\end{rem}

\begin{rem}(Connection with the classical independence)
\label{rem:seq-indep-and-cl-indep}
	Note that this sequential independence becomes classical independence (which is symmetric) once $X$ and $Y$ have \emph{certain} classical distribution. In other words, they can be put under a common classical probability space. In this case, $\expt$ reduces to a linear expectation $\lexpt_\lprob$. To give readers a better understanding, without loss of generality, suppose $(X,Y)$ have a classical joint continuous distribution with density function $f_{X,Y}$ and the marginal densities are $f_X$ and $f_Y$, we have, for any applicable $\varphi$,
	\[
\begin{array}{ccc}
\lexpt_\lprob[\varphi(X,Y)] & = & \int\varphi(x,y)f_{X,Y}(x,y)\diff x\diff y\\
\verteq\\
\lexpt_\lprob[\lexpt_\lprob[\varphi(x,Y)]_{x=X}] & = & \int_{x}\int_{y}\varphi(x,y)f_{X}(x)f_{Y}(y)\diff y\diff x
\end{array}.
	\]
	Therefore, we have
	$f_{X,Y}=f_X f_Y$, which means $X$ and $Y$ are (classically) independent.
\end{rem}

%
%


\begin{eg}[Example 1.3.15 in \cite{peng2019nonlinear}]
\label{eg:indep-ex-org}
	Consider two identically distributed $X,Y\in \myset{H}$ with $\expt[-X]=\expt[X]=0$ and $\sdr^2=\expt[X^2]>-\expt[-X^2]=\sdl^2$. Also assume $\expt[\abs{X}]>0$ such that $\expt[X^+]=\frac{1}{2}\expt[\abs{X}+X]=\frac{1}{2}\expt[\abs{X}]>0$. Then we have 
	\[
	\expt[XY^2]=\begin{cases}
		(\sdr^2-\sdl^2)\expt[X^+] & \text{if } X \seqind Y \\
		0 & \text{if } Y\seqind X
	\end{cases}.
	\]
\end{eg}

We will further study the interpretation of the independence (especially its asymmetric property) 
in \ref{subsec:interpret-asymmetry-indep} 
by giving a detailed version of \ref{eg:indep-ex-org} with representation theorems. 

 
Next we give the notion of independence extended to a sequence of random variables. 

\begin{defn}(Independence of Sequence)
\label{defn:indep-seq}
For a sequence $\{X_i\}_{i=1}^n$ of random variables,
they are \emph{(sequentially) independent} if 
\[
(X_1,X_2,\dotsc,X_i)\seqind X_{i+1},
\]
for $i=1,2,\dotsc,n-1$. For notational convenience, the sequential independence of $\{X_i\}_{i=1}^n$ is denoted as 
\begin{equation}
\label{eq:seqiid}
X_1\seqind X_2 \seqind \cdots \seqind X_n.	
\end{equation}
This sequence $\{X_i\}_{i=1}^n$ is further \emph{identically and independently distributed} if they are sequentially independent and $X_{i+1}\eqdistn X_i$ for $i=1,2,\dotsc,n-1$. This property is called (nonlinearly) i.i.d. in short.
\end{defn}

\begin{rem}
	Note that the independence \ref{eq:seqiid} is stronger than the pairwise relation $X_k\seqind X_{k+1}$ with $k=1,2,\dotsc,n-1$.
\end{rem}



Now we introduce two fundamental $G$-version distributions: maximal and $G$-normal distributions. The former one can be treated as an analogue of ``constant'' in classical sense. The latter one is a generalization of classical normal. 
We call $\bar X$ an independent copy of $X$ if $\bar X \eqdistn X$ and $X\seqind \bar X$. 

We first introduce the $G$-distribution which is the joint vector
of these two fundamental distributions.

Let $\numset{S}(d)$ denote the collection of all $d\times d$ symmetric
matrices.

\begin{prop}
\label{prop:existence-G-distn}
Let $G:\numset{R}^{d}\times\numset{S}(d)\to\numset{R}$ be a function
satisfying: for each $p,\bar{p}\in\numset{R}^{d}$ and $A,\bar{A}\in\numset{S}(d)$, 
\begin{equation}
\begin{cases}
G(p+\bar{p},A+\bar{A})\leq G(p,A)+G(\bar{p},\bar{A}),\\
G(\lambda p,\lambda A)=\lambda G(p,A)\text{ for any }\lambda\geq0,\\
G(p,A)\leq G(p,\bar{A})\text{ if }A\leq\bar{A}.
\end{cases}\label{eq:G-condition}
\end{equation}
Then there exists a pair $(X,\eta)$ on some sublinear expectation
space $(\Omega,\myset{H},\expt)$ such that 
\begin{equation}
G(p,A)=\expt[\frac{1}{2}\langle AX,X\rangle+\langle p,\eta\rangle],\label{eq:G-expt-distn}
\end{equation}
and for any $a,b\geq0$,
\begin{equation}
(aX+b\bar{X},a^{2}\eta+b^{2}\bar{\eta})\eqdistn(\sqrt{a^{2}+b^{2}}X,(a^{2}+b^{2})\eta),\label{eq:G-distn-eqdistn}
\end{equation}
where $(\bar{X},\bar{\eta})$ is an independent copy of $(X,\eta)$. 

\end{prop}

\begin{rem}
	The relation \ref{eq:G-distn-eqdistn} is equivalent to 
	$
	(X+\bar{X},\eta+\bar{\eta})\eqdistn(\sqrt{2}X,2\eta).
	$
\end{rem}

The proof of \ref{prop:existence-G-distn} is available at Section 2.3 in \cite{peng2019nonlinear}. Then we have the notion of $G$-distribution associated with a function $G$. 

\begin{defn}
($G$-distribution) A pair $(X,\eta)$ satisfying \ref{eq:G-distn-eqdistn}
is called $G$-distributed associated with a function $G$ in terms
of \ref{eq:G-expt-distn}. 
\end{defn}

The sublinear expectation of the random vector $(X,\eta)$ above can be characterized by the solution to a parabolic partial differential equation. 

\begin{prop}
Consider a $G$-distributed random vector $(X,\eta)$ associated with
a function $G$. For any $\varphi\in\fspacedef(\numset{R}^{d}\times\numset{R})$,
let 
\[
u(t,x,y)\coloneqq\expt[\varphi(x+\sqrt{t}X,y+t\eta)],\;(t,x,y)\in[0,\infty)\times\numset{R}^{d}\times\numset{R}^{d}.
\]
Then $u$ is the unique (viscosity) solution to the following parabolic
partial differential equation (PDE): 
\[
\partial_{t}u-G(D_{y}u,D_{x}^{2}u)=0,
\]
with initial condition $u|_{t=0}=\varphi$, where $D_{x}^{2}u\coloneqq(\partial_{x_{i}x_{j}}^{2}u)_{i,j=1}^{d}$
and $D_{y}u\coloneqq(\partial_{y_{i}}u)_{i=1}^{d}$. This PDE is called
a $G$-equation. 
\end{prop}

\begin{rem}
	Readers may turn to \cite{crandall1992user} for more details on the notion of viscosity solutions. In this paper, we do not require readers' knowledge on the viscosity solution. Moreover, it can be treated as a classical one when the function $G$ satisfies the strong elliplicity condition.  
\end{rem}




Next we provide a useful established property of the $G$-distributed random vector $(X,\eta)$. Suppose $\abs{\eta},\abs{X}^2\in \myset{H}$ and the following uniform integrability conditions are statisfied (proposed by \cite{zhang2016rosenthal}): 
\begin{equation}
	\label{eq:cond-unif-int-1}
	\lim_{\lambda \to \infty}\expt[(\abs{\eta}-\lambda)^+]=0,
\end{equation}
and 
\begin{equation}
	\label{eq:cond-unif-int-2}
	\lim_{\lambda \to \infty}\expt[(\abs{X}^2-\lambda)^+]=0.
\end{equation}
Then for any $\varphi\in\fspace$ (which is larger than $\fspacedef$), we still have $\varphi(\eta,X) \in \myset{H}$ (which is a Banach space). (This result is provided in Section 2.5 in \cite{peng2019nonlinear}.)
Therefore, in the following context, when we talk about $\varphi(\eta,X)$ for a $G$-distributed random vector $(\eta,X)$, we can take $\varphi \in \fspace$. 





If we pay attention to each marginal part in \ref{eq:G-distn-eqdistn}, we can see that $X$ is similar to a classical normal distribution while $\eta$ behaves like a constant (we do not consider Cauchy distribution here because we assume the existence of expectation). It turns out $X$ follows a $G$-normal distribution and $\eta$ follows a maximal distribution. 



\begin{defn}[Maximal distribution]
A $d$-dimensional random vector $\eta$ follows 
	a \emph{maximal distribution} if, for any independent copy $\bar \eta$, we have 
	\[
	\eta +\bar{\eta}\eqdistn 2\eta.
	\]
	Another equivalent and specific definition is that $\eta$ follows the maximal distribution $\Maximal(\Gamma)$ if there exists a bounded, closed and convex subset $\Gamma\subset\numset{R}^d$ such that, for any $\varphi\in\fspace$,
\[\expt[\varphi(\eta)]=\max_{y\in\Gamma} \varphi(y). \]
\end{defn}




\begin{defn}[$G$-normal distribution]
	A $d$-dimensional random vector $X$ follows a \emph{$G$-normal distribution} if, for any independent copy $\bar X$, we have 
	\[
	X + \bar X \eqdistn \sqrt{2} X.
	\]
	When $d=1$, we have $X\sim\GN(0,\varI)$ ($0\leq \sdl \leq \sdr$) with \emph{variance-uncertainty}: $\sdl^2\coloneqq-\expt[-X^2]$ and $\sdr^2\coloneqq\expt[X^2]$.
\end{defn} 



%
%



%

\begin{prop}[$G$-normal distribution characterized by the $G$-heat Equation] 
\label{$G$-normal-heat-eq}
A random vector $X$ follows the $d$-dimensional $G$-normal distribution, if and only if
	$v(t,x)\coloneqq \expt[\varphi(x+\sqrt{t}X)]$
	is the solution to the $G$-heat equation defined on $(t,x)\in [0,1]\times\numset{R}^d$:
	\begin{equation}
	\label{eq:G-heat-eq}
		v_t-G(D_x^2 v)=0,\,v|_{t=0}=\varphi,
	\end{equation}
	where $G(\mymat{A})\coloneqq\frac 12\expt[\langle \mymat{A}X,X \rangle]:\numset{S}_d \to \numset{R}$, which is a sublinear function characterizing the distribution of $X$. For $d=1$, we have $G(a)=\frac 12 (\sdr^2 a^+ -\sdl^2 a^-)$ and when $\sdl^2>0$, \ref{eq:G-heat-eq} is also called \emph{the Black-Scholes-Barenblatt equation} with volatility uncertainty. 
	
\end{prop}


\begin{rem}
\label{rem:GN-degenerate-to-CN}
For $d=1$, when $\sdl=\sdr=\sigma$, the $G$-normal distribution $X\sim\GN(0,\varI)$ can be treated as a classical normal $\CN(0,\sigma^2)$ because the $G$-heat equation is reduced to a classical one. 
\end{rem}

\begin{rem} (Covariance uncertainty) 
	We can use the function $G(\mymat{A})\coloneqq\frac 12\expt[\langle \mymat{A}X,X \rangle]$ to characterize the definition of $G$-normal distribution. In fact, $G(\mymat{A})$ can be further expressed as 
	\[G(\mymat{A})=\frac 12 \sup_{\dvar\in\dvarset} \trace[\mymat{A}\dvar]\] where $\dvarset=\{\mymat{B}\mymat{B}^T:\mymat{B}\in \numset{S}_d\}$ is a collection of non-negative definite symmetric matrices which can be treated as the uncertainty set of the covariance matrices. In this sense, we can write $X\sim \GN(\myvec{0},\dvarset)$.
\end{rem}

\begin{prop}
	\label{prop:convex-concave-case}
	Consider $X\sim \GN(0,\varI)$ and a classical distributed random variable $\stdrv \sim \CN(0,1)$.
	For any $\varphi \in \fspace(\numset{R})$,  we have 
	\[
	\expt[\varphi(X)] = \begin{cases}
		\lexpt[\varphi(\sdr \stdrv)] & \text{if }\varphi \text{ is convex}\\
		\lexpt[\varphi(\sdl \stdrv)] & \text{if }\varphi \text{ is concave}
	\end{cases}.
	\]	
\end{prop}




\begin{thm}[Law of Large Numbers] 
\label{$G$-LLN}
	Consider a sequence of i.i.d. $\{Z_i\}_{i=1}^\infty$ satisfying 
	\begin{equation}
		\label{eq:lln-cond}
		\lim_{\lambda \to \infty} \expt[(\abs{Z_1}-\lambda)^+]=0.
	\end{equation}
	Then for any continuous functions $\varphi$ satisfying the linear growth condition $\abs{\varphi(x)}\leq C(1+\abs{x})$,
	we have
	\[
	\lim_{n\to \infty} \expt[\varphi(\frac{1}{n} \sum_{i=1}^n Z_i)] = \max_{v\in \Gamma} \varphi(v),
	\]
	where $\Gamma$ is the bounded, closed and convex subset decided by 
	\[
	 \max_{\eta \in \Gamma}\langle p, Z_1 \rangle  =\expt[\langle p, Z_1 \rangle],\;p\in\numset{R}^d.
	 \]
	For $d=1$, let
	$\meanl \coloneqq -\expt[-Z_1]$ and $\meanr \coloneqq \expt[Z_1]$. Then $\frac{1}{n} \sum_{i=1}^n Z_i \convergeto{\dist} \Maximal\meanInt$, that is, we have, 
	\[
	\lim_{n\to \infty} \expt[\varphi(\frac{1}{n} \sum_{i=1}^n Z_i)] = \expt[\varphi(\Maximal\meanInt)]=\max_{\meanl\leq v\leq \meanr} \varphi(v).
	\]
\end{thm} 


\begin{thm}[Central Limit Theorem]
\label{$G$-CLT}
	Consider a sequence of i.i.d. $\{X_i\}_{i=1}^\infty$ satisfying mean-certainty $\expt[X_1]=-\expt[-X_1]=\myvec 0$ and 
\begin{equation}
		\label{eq:lln-cond}
		\lim_{\lambda \to \infty} \expt[(\abs{X_1}^2-\lambda)^+]=0.
	\end{equation}
	Then for any continuous functions $\varphi$ satisfying the linear growth condition $\abs{\varphi(x)}\leq C(1+\abs{x})$, 
	\[
	\lim_{n\to\infty} \expt[\varphi(\frac{1}{\sqrt n}\sum_{i=1}^{n} X_i)] = \expt[\varphi(X)],
	 \]
	 where $X$ is a $G$-normally distributed random variable characterized by the sublinear function $G$ defined as 
	 \[
	 G(A)\coloneqq \expt[\frac{1}{2} \langle AX_1,X_1\rangle],\; A\in \numset{S}_d.
	 \]
	For $d=1$, let $\sdl^2\coloneqq-\expt[-X_1^2]$ and $\sdr^2\coloneqq\expt[X_1^2]$. Then we have $\frac{1}{\sqrt n}\sum_{i=1}^{n} X_i$ converges in distribution to $X\sim \GN(0,\varI)$ . 
\end{thm}

\begin{prop}
\label{prop:extend-function-space}
	Consider a sequence $\{Y_n\}_{n=1}^\infty$ and $Y_n$ satisfying 
	\[
	\sup_n \expt[\abs{Y_n}^p] + \expt[\abs{Y}^p] <\infty,
	\]
	for any $p\geq 1$. 
	If the convergence $\lim_{n\to\infty} \expt[\varphi(Y_n)]=\expt[\varphi(Y)]$ holds for any $\varphi \in \fspacedef$, then it also holds for $\varphi\in \fspace$. 
\end{prop}

\begin{rem}
	\ref{prop:extend-function-space} is a direct result of Lemma 2.4.12 in \cite{peng2019nonlinear}. It useful when we need to extend the function space for $\varphi$ to discuss the convergence in distribution. 
\end{rem}


\subsection{Basic results on independence of sequence}
\label{subsec:basic-re-indep-seq}

We prepare several basic results on sequential independence between random vectors in the $G$-framework:
\begin{itemize}
	\item \ref{ind-subseq} gives a general result showing the sequential independence between two random vectors implies the independence between their sub-vectors. 
	\item \ref{ind-subseq-2} shows the sequential independence of a sequence implies the independence of the sub-sequence. 
	\item \ref{prop:indep-sub-seq} shows under the sequential independence of a sequence, any two non-overlapping subvector has the sequential independence (as long as keeping the original order.)
\end{itemize}
These results are useful for the discussions in \ref{subsec:indep-semignorm}.
We provide the proofs for the convenience of general readers and help them better understand how to deal with the sequential independence $\seqind$.


\begin{prop}
\label{ind-subseq}
For any subsequences $\{i_p\}_{p=1}^k$ and $\{j_q\}_{q=1}^l$ satisfying $1\leq i_{1}<i_{2}<\dotsc<i_{k}\leq n$ and $1\leq j_{1}<j_{2}<\dotsc<j_{l}\leq m$,  we have the general result that 
\[
(X_{1},X_{2},\dotsc,X_{n})\seqind(Y_{1},Y_{2},\dotsc,Y_{m}) \implies (X_{i_{1}},X_{i_{2}},\dotsc,X_{i_{k}})\seqind(Y_{j_{1}},Y_{j_{2}},\dotsc,Y_{j_{l}}). 
 \]   
\end{prop}



\begin{proof}
For any applicable test function $\varphi\in \fspace(\R^{k+l})$, 
	define another function $\psi \in \fspace(\R^{n+m})$ on a larger space by
\[
\psi(x_{1},x_{2},\dotsc,x_{n},y_{1},y_{2},\dotsc,y_{m})\coloneqq\varphi(x_{i_{1}},x_{i_{2}},\dotsc,x_{i_{k}},y_{j_{1}},y_{j_{2}},\dotsc,y_{j_{l}})
,\]
then
\begin{align*}
& \hphantom{=} \expt[\varphi(X_{i_{1}},X_{i_{2}},\dotsc,X_{i_{k}},Y_{j_{1}},Y_{j_{2}},\dotsc,Y_{j_{l}})]\\
 & =\expt[\psi(X_{1},X_{2},\dotsc,X_{n},Y_{1},Y_{2},\dotsc,Y_{m})]\\
 & =\expt[\expt[\psi(x_{1},x_{2},\dotsc,x_{n},Y_{1},Y_{2},\dotsc,Y_{m})]_{x_{i}=X_{i},i=1,\dotsc n}]\\
 & =\expt[\expt[\varphi(x_{i_{1}},x_{i_{2}},\dotsc,x_{i_{k}},Y_{j_{1}},Y_{j_{2}},\dotsc,Y_{j_{l}})]_{x_{i_p}=X_{i_p},p=1,\dotsc,k}].\qedhere
\end{align*}

\end{proof}


%

\begin{prop} 
\label{ind-subseq-2}
For any subsequence $\{i_p\}_{p=1}^k$ satisfying $1\leq i_{1}<i_{2}<\dotsc<i_{k}\leq n$, we have the result that 
	\[
	X_{1}\seqind X_{2}\seqind\dotsc\seqind X_{n} \implies X_{i_{1}} \seqind X_{i_{2}} \seqind \dotsc \seqind X_{i_{k}}.
	 \]
\end{prop}

\begin{proof}
	It is equivalent to prove 
	$
	(X_{i_1},X_{i_2},\dotsc,X_{i_{j-1}})\seqind X_{i_j}
	 $
	for any $j=2,\dotsc,k$. 
	For any $j = 2,\dotsc,k$, by the definition of independence of the full sequence $\{X_i\}_{i=1}^n$, we have
	\[
	(X_1,X_2,\dotsc,X_{i_{j-1}},\dotsc, X_{i_j-1}) \seqind X_{i_j}.
	 \]  
	From \ref{ind-subseq}, we directly have the sequential independence for the subvectors: 
	\[
	(X_{i_1},X_{i_2},\dotsc,X_{i_{j-1}})\seqind X_{i_j}.\qedhere
	 \]

\end{proof}

The following \ref{lem:three-obj-seqind} and \ref{lem:three-obj-seqind-2} will be useful in our later discussion, where the dimension of the three objects $X,Y,Z$ could be arbitrary finite number. 

\begin{lem}
\label{lem:three-obj-seqind}If $X\seqind Y\seqind Z$, then $X\seqind(Y,Z)$. 
\end{lem}

\begin{proof}
Let 
$
H(x,y)\coloneqq\expt[\varphi(x,y,Z)].
$
Then we can check
\begin{align*}
\expt[\expt[\varphi(x,Y,Z)]_{x=X}] 
& \overset{(1)}{=}\expt[\expt[\expt[\varphi(x,y,Z)]_{y=Y}]_{x=X}]\\
 & =\expt[\expt[H(x,Y)]_{x=X}]\\
 & \overset{(2)}{=}\expt[H(X,Y)]\\
 & =\expt[\expt[\varphi(x,y,Z)]_{x=X,y=Y}]\\
 & \overset{(3)}{=}\expt[\varphi(X,Y,Z)],
\end{align*}
where (1) is due to $Y\seqind Z$, (2) comes from $X\seqind Y$ and (3) comes from $(X,Y)\seqind Z$.
\end{proof}

\begin{lem}
	\label{lem:three-obj-seqind-2}
	If $X\seqind (Y,Z)$ and $Y \seqind Z$, we have $(X,Y)\seqind Z$.
\end{lem}

\begin{proof}
Let 
$
H(x,y)\coloneqq\expt[\varphi(x,y,Z)].
$
Then
\begin{align*}
\expt[\expt[\varphi(x,y,Z)]_{x=X,y=Y}] & =\expt[H(X,Y)]\\
 & \overset{(1)}{=}\expt[\expt[H(x,Y)]_{x=X}]\\
 & =\expt[\expt[\expt[\varphi(x,y,Z)]_{y=Y}]_{x=X}]\\
 & \overset{(2)}{=} \expt[\expt[\varphi(x,Y,Z)]_{x=X}]\\
 & \overset{(3)}{=}\expt[\varphi(X,Y,Z)],
\end{align*}
where (1) comes from $X\seqind Y$, (2) comes from $Y\seqind Z$ and (3) comes from $X\seqind (Y,Z)$. 
\end{proof}

\begin{prop}
\label{prop:indep-sub-seq}If $X_{1}\seqind X_{2}\seqind\cdots\seqind X_{n}$,
for any increasing subsequence $\{i_{j}\}_{j=1}^{k}\subset\{1,2,\dotsc,n\}$,
we have 
\[
(X_{i_{1}},\dotsc,X_{i_{l}})\seqind(X_{i_{l+1}},\dotsc,X_{i_{k}}).
\]
\end{prop}

\begin{proof}
Let $Y_{j}\coloneqq X_{i_{j}}.$ Then we have $Y_{1}\seqind Y_{2}\seqind\cdots\seqind Y_{k}$.
Our goal is to show for any $l=1,2,\dotsc,k-1$, 
\begin{equation}
(Y_{1},\dotsc,Y_{l})\seqind(Y_{l+1},\dotsc,Y_{k}).\label{eq:goal-indep-sub-seq}
\end{equation}
Then we can proceed by math induction. Let $m=k-l$. The result \ref{eq:goal-indep-sub-seq}
holds when $m=1$ because we directly have
\[
(Y_{1},\dotsc,Y_{k-1})\seqind Y_{k},
\]
by the definition of $Y_{1}\seqind Y_{2}\seqind\cdots\seqind Y_{k}$.
Suppose \ref{eq:goal-indep-sub-seq} holds for $m=j$. We need to
show the case with $m=j+1$: 
\begin{equation}
(Y_{1},\dotsc,Y_{k-j-1})\seqind(Y_{k-j},Y_{k-j+1}\dotsc,Y_{k}).\label{eq:goal-math-induction}
\end{equation}
Let 
\begin{align*}
A_{1} & \coloneqq(Y_{1},\dotsc,Y_{k-j-1}),\\
A_{2} & \coloneqq Y_{k-j},\\
A_{3} & \coloneqq(Y_{k-j+1},\dotsc,Y_{k}).
\end{align*}
Then we have $A_{1}\seqind A_{2}$ by the definition of $Y_{1}\seqind Y_{2}\seqind\cdots\seqind Y_{k-j}.$
We also have $(A_{1},A_{2})\seqind A_{3}$ by the result for $m=j$.
Then we can follow the same logic of \ref{lem:three-obj-seqind} to
show 
\[
A_{1}\seqind(A_{2},A_{3}),
\]
which is exactly \ref{eq:goal-math-induction}. The proof is finished by math induction.
\end{proof}

\begin{prop}
\label{prop:four-obj-seqind-gen}
The following two statements are equivalent: 
\begin{enumerate}
\item[(1)] $X_{1}\seqind X_{2}\seqind X_{3}\seqind X_{4}$, 
\item[(2)] $(X_{1},X_{2})\seqind(X_{3},X_{4})$, $X_{1}\seqind X_{2}$ and $X_{3}\seqind X_{4}$. 
\end{enumerate}
\end{prop}

\begin{proof}
Since (1) implies (2), so we only need to show the other direction.
By the definition of (1), we simply need to check: 
\[
(X_{1},X_{2},X_{3})\seqind X_{4}.
\]
This is a direct consequence of \ref{lem:three-obj-seqind-2} by letting $X^{*}\coloneqq(X_{1},X_{2})$,
$Y^{*}\coloneqq X_{3}$ and $Z^{*}\coloneqq X_{4}$. 
\end{proof}

\begin{prop}
\label{prop:linear-prop-expt} For any $X,Y\in \myset{H}$, we have $\expt[X+Y]=\expt[X]+\expt[Y]$
as long as either one of the following conditions holds: 
\begin{enumerate}
\item $\expt[X]=-\expt[-X]$;
\item $\expt[Y]=-\expt[-Y]$; 
\item $Y$ is independent from $X$: $X\seqind Y$;
\item $X$ is independent from $Y$: $Y\seqind X$.
\end{enumerate}
\end{prop}

\begin{proof}
	We only need to show Condition 1 and 3. Under Condition 1, 
	we have 
	\[
	\expt[X+Y]\leq \expt[X]+\expt[Y] = \expt[Y]-\expt[-X] \leq \expt[Y-(-X)]=\expt[X+Y].
	\]
	Under condition 3, we have 
	\[
	\expt[X+Y]=\expt[\expt[x+Y]_{x=X}] = \expt[X+\expt[Y]]=\expt[X]+\expt[Y].\qedhere
	\]
\end{proof}


\ref{thm:mutual-indep} is an important result that shows the asymmetry of independence between two random objects prevails in this framework except when their distributions are maximal or classical ones.

\begin{thm}[\cite{hu2014independence}]
\label{thm:mutual-indep} For two non-constant random varibles $X,Y\in\myset{H}$,
if $X$ and $Y$ are mutually independent ($X\seqind Y$ and $Y\seqind X$),
then they belong to either of the following two cases: 
\begin{enumerate}
\item The distributions of $X$ and $Y$ are classical (no distributional
uncertainty); 
\item Both $X$ and $Y$ are maximally distributed. 
\end{enumerate}
\end{thm}

We can also easily obtain the following result.
\begin{prop}
\label{prop:one-indep-to-mutual-indep}For two non-constant random
varibles $X,Y\in\myset{H}$, if they belong to either of the two
cases in \ref{thm:mutual-indep}, then we have $X\seqind Y$ implies
$Y\seqind X$. 
\end{prop}

\begin{proof}
	When $X,Y$ are classically distributed, the results can be derived from \ref{rem:seq-indep-and-cl-indep}. When they are maximally distributed, this result has been studied in Example 14 in \cite{hu2014independence} and it has been generalized to \ref{thm:maximal-multi-relation} whose proof is in \ref{pf:subsec-property-maximal}.
	We sketch the proof here to show the intuition for general readers. Suppose $X\sim \Maximal(K_1)$ and $Y\sim \Maximal(K_2)$ where $K_1$ and $K_2$ are two bounded, closed and convex sets. If $X\seqind Y$, 
for any $\varphi\in\fspace(\numset{R}^{2})$, we can work on the expectation of $(X,Y)$ to show the other direction of independence, 
\begin{align*}
 \expt[\varphi(X,Y)] & =\expt[\expt[\varphi(x,Y)]_{x=X}]
  =\expt[(\max_{y\in K_{2}}\varphi(x,y))_{x=X}]\\
 & =\max_{x\in K_{1}}\max_{y\in K_{2}}\varphi(x,y)
  =\max_{(x,y)\in K_{1}\times K_{2}}\varphi(x,y)\\
 & =\max_{y\in K_{2}}\max_{x\in K_{1}}\varphi(x,y)
  =\expt[\expt[\varphi(X,y)]_{y=Y}],
\end{align*}
where we have used the fact that $
\varphi^*_{y}(x)\coloneqq\max_{y\in K_2}\varphi(x,y)\in\fspace(\numset{R})
$ if $\varphi\in\fspace(\numset{R}^{2})$ (to apply the representation), which is validated by \ref{claim:maximal-varphi-k} in the proof of \ref{thm:maximal-multi-relation}. Hence, we have $Y\seqind X$. 
\end{proof}

\section{Our main results: semi-$G$-normal and its representations}
\label{sec:main-re}

This section serves for two objectives. On the one hand, we will introduce
a new substructure called the semi-$G$-normal distribution and explain
its hybrid property and intermediate role sitting between the classical
normal and $G$-normal distribution. On the other hand, this section
is also designed to give general readers a gentle trip towards the
$G$-normal distribution by starting from our old friend, the classical
normal distribution.

Although most of the theoretical results presented in this section are in the sublinear expectation space $\exptSpace$ by default unless indicated in the context, we will introduce most of the subsections by starting from a discussion on the distributional uncertainty of a random object in a classical state-space volatility model, whose context will be set up in \ref{subsec:setup-story}.

Without further notice, these are the notations we are going to consistently use in this paper: 
\begin{itemize}
    \item $\numset{N}_{+}$: the set of all positive integers.
	\item $\myvec{x}_{(n)}\coloneqq(x_{1},x_{2},\dotsc,x_{n})$ and $\myvec{x}_{(n)}*\myvec{y}_{(n)}\coloneqq(x_{1}y_{1},x_{2}y_{2},\dotsc,x_{n}y_{n})$.
	\item Let $\idtymat_d$ denote a $d\times d$ identity matrix.
	\item In $\lexptSpace$, let $\lexpt_\lprob$ denote the linear expectation with respect to $\lprob$ and we may write it as $\lexpt$ if the underlying $\lprob$ is clear from the context. 
	\item Random variables in $\exptSpace$: $V\sim \Maximal\sdInt$, $\stdrv \sim \GN(0,[1,1]$, $W\coloneqq V\stdrv$, $W^G\sim \GN(0,\varI)$.
	\item Random variables in $\lexptSpace$: $\sigma:\Omega \to \sdInt$, $\stdrv \sim \CN(0,1)$, $Y\coloneqq \sigma \stdrv$. Note that we can treat $\stdrv$ as a random object in both sublinear and classical system  due to \ref{rem:GN-degenerate-to-CN}.
\end{itemize}

The reason we use two different sets of random variables in two spaces is mainly for simplicity of our discussion, which will be further explained in \ref{rem:prefered-form-represent}.

 






\subsection{Setup of a story in a classical filtered probability space}
\label{subsec:setup-story}

In $\lexptSpace$, 
consider $(\epsilon_{t})_{t\in \numset{N}_+}$ as a sequence 
of classically i.i.d. random variables satisfying $\lexpt_\lprob[|\epsilon_{1}|^{k}]<\infty$
for $k\in\numset{N}_{+}$. Let $(\sigma_{t})_{t\in \numset{N}_+}$ be a sequence of bounded random variables which can be treated as \emph{states} (or \emph{volatility regimes}) with state space $S_\sigma \subset \sdInt$. 
Let $Y_{t}=\sigma_{t}\epsilon_{t},t\in\numset{N}_{+}$ denote the \emph{observation} sequence. 
(It seems like a zero-delay setup, but this is not essential in our current scope of discussion.) Consider a representative $Y=\sigma\epsilon$ where $(\sigma,\epsilon)\coloneqq(\sigma_1,\epsilon_1)$. 

For simiplicty of discussion, at each time point $t$,
we assume that $\epsilon_{t}$ follows $\CN(0,1)$ and $\sigma_t$ and $\epsilon_t$ are classically independent, denoted as $\sigma_t \independent \stdrv_t$. 
Consider the following (discrete-time) filtrations: 
\begin{align*}
\sigmafield{G}_{t} & \coloneqq\sigma(\sigma_{s},s\leq t)\vee\sigmafield{N},\\
\sigmafield{Y}_{t} & \coloneqq\sigma(Y_s,s\leq t)\vee\sigmafield{N},\\
\sigmafield{F}_{t} & \coloneqq\sigma((\sigma_{s},Y_{s}),s\leq t)\vee\sigmafield{N}.
\end{align*}
where $\sigmafield{N}$ is the collection of $\lprob$-null sets used to complete each of the generated $\sigma$-field mentioned above. Note that $\sigmafield{F}_{t}$ is the same as $\sigma((\epsilon_{s},\sigma_{s}),s\leq t)\vee\sigmafield{N}$. Let $\numset{F}\coloneqq \{ \sigmafield{F}_t \}_{t\in \numset{N}_+}$.
In a classical filtered probability space $(\Omega,\sigmafield{F},\numset{F},\lprob)$, we will start the following subsections by 
putting ourseleves, as a group of  data analysts, in a context of dealing with uncertainty on the distributions of one state variable $\sigma$, one observation variable $Y=\sigma\epsilon$ and a sequence of observation variables $(Y_1,Y_2,\dotsc,Y_n)$ for $n\in \numset{N}_+$.




\subsection{Preparation: properties of maximal distribution}
\label{subsec:property-maximal}
Suppose we have uncertainty on the distribution of the state variable $\sigma$ (and
it is realistic because $\sigma$ is not directly observable in practice)
due to lack of prior knowledge. Another possible situation is different member in our group has different belief on the behavior of $\sigma$ or different preference on the choice of the model - the distribution of $\sigma$ could be a degenerate, discrete, absolutely continuous or even arbitrary one with support on $\sdInt$. 
In order to quantify this kind of model uncertainty 
 for a given transformation $\varphi$ (as a test function), we usually need to involve the maximum expected value of $\sigma$:
\begin{equation}
\sup_{\sigma\in\mysetrv{A}\sdInt}\lexpt[\varphi(\sigma)],\label{eq:extrem-expt}
\end{equation}
where $\mysetrv{A}\sdInt$ can be chosen depending on the available
prior information. Possible choices of $\mysetrv{A}\sdInt$ include,
\begin{itemize}
\item $\mysetrv{D}\sdInt$: the space of all classically distributed
random variables with support on $\sdInt$.
\item $\mysetrvdisc\sdInt\coloneqq\{\sigma\in\mysetrv{D}\sdInt:\text{ discretely distributed}\}$.
\item $\mysetrvcont\sdInt\coloneqq\{\sigma\in\mysetrv{D}\sdInt:\text{ absolutely continuously distributed}\}$.
\item $\mysetrvdeg\sdInt\coloneqq\{\sigma\in\mysetrv{D}\sdInt:\lprob(\sigma=v)=1,v\in\sdInt\}$, which is the family of all random variables following degenerate (or Dirac) distribution with mass point at $v\in \sdInt$.\end{itemize}


We are going to show that \ref{eq:extrem-expt} will all be the same
as the sublinear expectation of maximal distribution in the $G$-framework. 


\begin{defn} (Univariate Maximal Distribution)
\label{def:maximal} In sublinear expectation space $\exptSpace$, a random variable $V$ follows maximal distribution $\Maximal\sdInt$ with $\sdl\leq\sdr$ if,
for any $\varphi\in\fspace(\numset{R}),$ 
\[
\expt[\varphi(V)]=\max_{v\in\sdInt}\varphi(v).
\]
\end{defn}
\begin{rem}
	We can take maximum because we are working on a continuous $\varphi$ on a compact set $\sdInt$. 
\end{rem}

The reason that we use the notation $\sdInt$ (which is like an interval for standard deviation) is for the convenience of our later discussion. 




\begin{prop}[Representations of Univariate Maximal Distribution] 
\label{prop:rep-uni-maximal}
Consider $V\sim\Maximal\sdInt$, then for any $\varphi\in\fspace(\numset{R})$, we have 
$\lexpt[\abs{\varphi(V)}]<\infty$ and 
\begin{align}
\expt[\varphi(V)] & =\max_{\sigma\in\mysetrvdeg\sdInt}\lexpt[\varphi(\sigma)]\label{eq:represent-maximal-1}\\
& = \max_{\sigma\in\mysetrvdisc\sdInt}\lexpt[\varphi(\sigma)]\label{eq:represent-maximal-disc} \\
 & =\sup_{\sigma\in\mysetrvcont\sdInt}\lexpt[\varphi(\sigma)]\label{eq:represent-maximal-2}\\
 & =\max_{\sigma\in\mysetrv{D}\sdInt}\lexpt[\varphi(\sigma)].\label{eq:represent-maximal-3}
\end{align}
\end{prop}

\begin{rem}
Note that $\mysetrvdisc\cup\mysetrvcont\subset\myset{D}$. The probability
laws associated with $\mysetrvcont$ are equivalent, but $\mysetrvdisc$
and $\myset{D}$ do not have this property. 
\end{rem}


\begin{rem}
\label{rem:prefered-form-represent}
In \ref{prop:rep-uni-maximal}, we write the representation in the
form of 
\begin{equation}
\expt[\varphi(V)]=\max_{\sigma\in\myset{A}}\lexpt[\varphi(\sigma)],\label{eq:represent-form-1}
\end{equation}
where $\myset{A}$ denote a family of random variables in $\lexptSpace$.
Equivalently, if we treat $V$ as a random variable for both sides
(which requires more careful preliminary setup and we will not touch at this stage, more details can be found in Chapter 6 of \cite{peng2019nonlinear}), 
\ref{eq:represent-form-1} becomes
\begin{equation}
\expt[\varphi(V)]=\max_{\lprob_{V}\in\lprob\circ\mysetrv{A}^{-1}}\lexpt_{\lprob}[\varphi(V)],\label{eq:represent-form-2}
\end{equation}
where $\lprob_V$ is the distribution of $V$ and $\lprob\circ\mysetrv{A}^{-1}\coloneqq\{\lprob\circ\sigma^{-1},\sigma\in\mysetrv{A}\}$
becomes a family of distributions. Throughout this
paper, we prefer to use the form \ref{eq:represent-form-1} for simplicity
of notations and minimization of technical setup, but readers can
always informally view \ref{eq:represent-form-1} as a equivalent
form of \ref{eq:represent-form-2}. In this way, we can better see
the distributional uncertainty of $V$. 
\end{rem}

\begin{rem}
Meanwhile, \ref{prop:rep-uni-maximal}
provides four ways to represent the distributional uncertainty of $V$.
In practice, practitioners may choose the representation they need
depending on the available prior knowledge or their belief on the random
phenomenon. 
\end{rem}

\begin{defn}
(Multivariate Maximal Distribution) 
\label{defn:multi-maximal}
In sublinear expectation space $\exptSpace$, a random vector
$\myvec{V}:\Omega \to \numset{R}^d$ follows a (multivariate) maximal
distribution $\Maximal(\myset{V})$, if there exists a compact
and convex subset $\myset{V}\subset\numset{R}^{d}$ satisfying: for
any $\varphi\in\fspace(\numset{R}^{d})$, 
\[
\expt[\varphi(\myvec{V})]=\max_{\myvec{\sigma}\in\myset{V}}\varphi(\myvec{\sigma}).
\]
\end{defn}

One can also easily extend the representation in \ref{prop:rep-uni-maximal} to a multivariate case (\ref{prop:rep-multi-maximal}) by considering $\mysetrv{D}(\myset{V})$ which is the space of all classically distributed random variables with support on $\myset{V}$ and also considering its subspaces $\mysetrvdeg(\myset{V})$, $\mysetrvcont(\myset{V})$ and $\mysetrvdisc(\myset{V})$ as well. 

\begin{prop}[Representations of multivariate maximal distribution]
\label{prop:rep-multi-maximal} For $\myvec{V}\sim\Maximal(\myset{V})$, we have for any $\varphi\in\fspace(\numset{R}^d)$,
\begin{equation}
\label{eq:represent-multi-maximal}
	\expt[\varphi(\myvec{V})]=\sup_{\sigma\in\myset{A}(\dsdset)}\lexpt[\varphi(\myvec{\sigma})],
\end{equation}
where $\mysetrv{A}$ can be chosen from $\mysetrvall$
and $\sup$ can be changed to $\max$ except when $\myset{A}=\mysetrvcont$.
\end{prop}
\begin{proof}
	It can be extended from the proof of \ref{prop:rep-uni-maximal}.
\end{proof}

Next we provide a property for multivariate maximal distribution under transformations.
\begin{prop}
\label{cor:Properties-of-Maximal}
Suppose $\myvec{V}\sim \Maximal(\myset{V})$. Then for any locally Lipschitz function $\psi:(\numset{R}^{d},\norm{\cdot})\to(\numset{R}^{k},\norm{\cdot})$, we have
\[
\myvec{S}\coloneqq\psi(\gdsd)=\psi(V_{1},V_{2},\dotsc,V_{d})\sim\Maximal(\myset{S}),
\]
where $\myset{S} \coloneqq \psi(\myset{V})=\{\psi(\myvec{\sigma}):\myvec{\sigma}\in\myset{V}\}$. 
\end{prop}

\begin{rem}
\ref{cor:Properties-of-Maximal} is generalized version of Proposition 25 and Remark 26 in \cite{jin2021optimal}. It shows that the transformation of a maximal is still a maximal distribution, with support equal to the range of the function.
\end{rem}

Next we give a connection between univariate and multivariate maximal distribution. 

\begin{prop}
\label{thm:maximal-multi-relation}(The relation between multivariate
and the univariate maximal distribution) 
Consider a sequence of maximally distributed random variables
$V_{i}\sim\Maximal[\sdl_{i},\sdr_{i}]$ with $\sdl_{i}\leq\sdr_{i}$, $i=1,2,\dotsc,d$, then the following three statements are equivalent: 
\begin{enumerate}
	\item[(1)] $\{V_{i}\}_{i=1}^d$ are sequentially independent $V_{1}\seqind V_{2}\seqind\cdots\seqind V_{d}$, 
	\item[(2)]  $V_{i_{1}}\seqind V_{i_{2}}\seqind\cdots\seqind V_{i_{d}}$ for any permutation $(i_{1},i_{2},\dotsc,i_{d})$ of $(1,2,\dotsc,d)$, 
	\item[(3)] $\myvec{V} \coloneqq (V_{1},V_{2},\dotsc,V_{d})\sim\Maximal(\prod_{i=1}^{d}[\sdl_{i},\sdr_{i}])$, where the operation $\prod_{i=1}^{d}$ is the Cartesian product. 
\end{enumerate}
\end{prop}

\begin{rem}
	\ref{thm:maximal-multi-relation} shows that the sequential independence between maximal distribution can be arbitrarily switched without changing its joint distribution, which is a maximal distribution supporting on a $d$-dimensional rectangle. Reversely speaking, if a random vector follows a maximal distribution concentrating this rectangle shape, it implies the sequential independence among its components. 
\end{rem}

%



As a special case of \ref{thm:maximal-multi-relation}, for two maximal distributed random variables $V_i,i=1,2$, $V_{1}\seqind V_{2}$
implies that $V_{2}\seqind V_{1}$. 
In fact, \ref{thm:mutual-indep} given by \cite{hu2014independence} shows for two non-constant, non-classical distributed random objects, this kind of mutual independence only holds for maximal distributions. The asymmetry of sequential independence prevails among the distributions in the $G$-expectation framework. 


\subsection{Preparation: setup of a product space (a newly added part)}


We start from a set $\myset{Q}$ of probability measures and a single probability
measure $P$, where $P$ does not have to be in $\myset{Q}$. Let
$\expt_{1}[\cdot]\coloneqq\sup_{Q\in\myset{Q}}\lexpt_{Q}[\cdot]$
and $\expt_{2}[\cdot]\coloneqq\lexpt_{P}[\cdot]$. Then we have 
the associated sublinear expectation spaces $(\Omega_{i},\myset{H}_{(i)},\expt_{i})$
with $i=1,2$. Note that $\expt_{2}$, as a linear operator, can be treated as a degenerate sublinear expectation. 
We may simply write the linear expectation $\lexpt_{P}$ as $\lexpt$
if the probability measure is clear from the context. Since $\expt_{2}$ is a linear expectation, the distributions under $(\Omega_{2},\myset{H}_{(2)},\expt_{2})$ can be treated as classical ones for which we assume they contain common classical distributions (such as classical normal).
We also assume $\myset{Q}$ is designed such that $G$-distribution exists in $(\Omega_{1},\myset{H}_{(1)},\expt_{1})$. 
Then we can combine these two spaces into a product space $(\Omega_{1}\times\Omega_{2},\myset{H}_{(1)}\otimes\myset{H}_{(2)},\expt_{1}\otimes\expt_{2})$. It is also forms a sublinear expectation space. More details on this notion of product space can be found in \cite{peng2019nonlinear} (Section 1.3). 

For readers' convenience, here we provide a brief description of this product space. 
\begin{enumerate}
\item The space is $\myset{H}_{(1)}\otimes\myset{H}_{(2)}$ defined as
\begin{align*}
\myset{H}_{(1)}\otimes\myset{H}_{(2)}= & \{X(\omega_{1},\omega_{2})=f(K(\omega_{1}),\eta(\omega_{2})),(\omega_{1},\omega_{2})\in\Omega_{1}\times\Omega_{2},\\
 & K\in\myset{H}_{(1)}^{m},\eta\in\myset{H}_{(2)}^{n},f\in\fspace(\numset{R}^{m+n}), m,n \in \numset{N}_+ \}.
\end{align*}
\item For $X(\omega_{1},\omega_{2})=f(K(\omega_{1}),\eta(\omega_{2}))\in\myset{H}_{(1)}\otimes\myset{H}_{(2)}$,
we have defined
\begin{align*}
\expt[X]=\expt_{1}\otimes\expt_{2}[X] & \coloneqq\expt_{1}[\expt_{2}[f(k,\eta)]_{k=K}]\\
 & =\sup_{Q\in\myset{Q}}\lexpt_{Q}[\lexpt_{P}[f(k,\eta)]_{k=K}]\\
 & =\sup_{Q\in\myset{Q}}\int\int f(k,y)P_{\eta}(\diff y)Q_{K}(\diff k)\\
 & =\sup_{\lprob\in\myset{P}}\lexpt_{\lprob}[X],
\end{align*}
where $\myset{P}\coloneqq\{Q\otimes P,Q\in\myset{Q}\}$ where $Q\otimes P$
is the product measure of $P$ and $Q$. Note that $\expt_{2}\otimes\expt_{1}\neq\expt_{1}\otimes\expt_{2}$
due to the sublinearity of $\expt_{1}$. 
\end{enumerate}
\begin{prop}
\label{prop:projection-K-eta}For a random variable $K$ on $(\Omega_{1},\myset{H}_{(1)},\expt_{1})$
and $\eta$ on $(\Omega_{2},\myset{H}_{(2)},\expt_{2})$, by letting
$\bar{K}(\omega_{1},\omega_{2})\coloneqq K(\omega_{1})$ and $\bar{\eta}(\omega_{1},\omega_{2})\coloneqq\eta(\omega_{2})$,
we have the following results: 
\begin{enumerate}
\item $\bar{K},\bar{\eta}\in\myset{H}_{(1)}\otimes\myset{H}_{(2)}$, 
\item $X(\omega)=f(\bar{K}(\omega),\bar{\eta}(\omega)),$ for $\omega\in\Omega_{1}\times\Omega_{2}$,
\item For any $\varphi\in\fspacedef$, $\expt[\varphi(\bar{K})]=\expt_{1}[\varphi(K)],$
\item For any $\varphi\in\fspacedef$, $\expt[\varphi(\bar{\eta})]=\lexpt_{P}[\varphi(\eta)],$
\item $\bar{K}\seqind\bar{\eta}.$
\end{enumerate}
\end{prop}

\begin{proof}
Item 1 and 2 are obvious to see. 
For Item 3, we have
\begin{align*}
\expt[\varphi(\bar{K})] & =\expt_{1}[\expt_{2}[\varphi(f_{1}(k,\eta))]_{k=K}]\\
 & =\expt_{1}[\expt_{2}[\varphi(k)]_{k=K}]=\expt_{1}[\varphi(K)].
\end{align*}
Similarly, we can show Item 4. For Item 5, our goal is to show
\[
\expt[\varphi(\bar{K},\bar{\eta})]=\expt[\expt[\varphi(k,\bar{\eta})]_{k=\bar{K}}].
\]
We can see the equation above holds from the following step: with $H(k)\coloneqq \lexpt_{P}[\varphi(k,\bar{\eta})]$,
\begin{align*}
\text{RHS} & =\expt[\lexpt_{P}[\varphi(k,\bar{\eta})]_{k=\bar{K}}]\\
 & =\expt[H(\bar{K})]=\expt_{1}[H(K)]\\
 & =\expt_{1}[\lexpt_{P}[\varphi(k,\bar{\eta})]_{k=K}]\\
 & =\expt[\varphi(K,\eta)]=\expt[\varphi(\bar{K},\bar{\eta})]. \qedhere
\end{align*}
\end{proof}

\begin{rem}
\label{rem:prob-eta} In the following context, without further notice,
we will not distinguish $\bar{K}$ (or $\bar{\eta}$) with $K$ (or
$\eta$). Moreover, we can see that, by letting $\eta(\omega_{1},\omega_{2})\coloneqq\eta(\omega_{2})$
so that $\eta\in\myset{H}_{(1)}\otimes\myset{H}_{(2)}$, from Item 3,
we have 
\[
\expt[\varphi(\eta)]=\sup_{\lprob\in\myset{P}}\lexpt_{\lprob}[\varphi(\eta)]=\lexpt_{P}[\varphi(\eta)],
\]
where $\lprob$ is any product measure $Q\otimes P$ where $Q\in\myset{Q}$.
By making $\varphi$ into $-\varphi$, we can show that for any $\lprob\in\myset{P}$,
\[
\lexpt_{P}[\varphi(\eta)]=\inf_{\lprob\in\myset{P}}\lexpt_{\lprob}[\varphi(\eta)]\leq\lexpt_{\lprob}[\varphi(\eta)]\leq\sup_{\lprob\in\myset{P}}\lexpt_{\lprob}[\varphi(\eta)]=\lexpt_{P}[\varphi(\eta)],
\]
or simply 
$
\lexpt_{\lprob}[\varphi(\eta)]=\lexpt_{P}[\varphi(\eta)].
$
It means that the probability law of $\eta$ is always $P_\eta$ under each product measure $\lprob\in \myset{P}$. 
\end{rem}

Let $\bar{\mysetrv{H}}_{s}$ denote a subspace of the product space mentioned above:
\begin{align*}
\bar{\mysetrv{H}}_{s}\coloneqq  \{X\in\myset{H}_{(1)}\otimes\myset{H}_{(2)}: 
& X(\omega_{1},\omega_{2})=f(K(\omega_{1}),\eta(\omega_{2})),K\in\myset{H}_{(1)}^{m}\sim\Maximal(\Theta), \\ 
& f\in \fspace(\numset{R}^{m+n}), \Theta \subset \numset{R}^m, m,n \in \numset{N}_+ \}.
\end{align*}
For any $X\in\bar{\mysetrv{H}}_{s}$, by the representation of maximal distribution, we have 
\[
\expt[X]=\sup_{\theta\in\Theta}\lexpt_{P}[f(\theta,\eta)].
\]


\subsection{Univariate semi-$G$-normal distribution}
\label{subsec:semi-G-normal-intro}
Recall the story setup in \ref{subsec:setup-story}. Note that $Y=\sigma\epsilon$ can be treated as a normal mixture with scaling latent variable $\sigma:\Omega\to \sdInt$. For simplicity of discussion, we have assumed $\sigma\independent \stdrv$ and $\stdrv\sim \CN(0,1)$.
Suppose we are further faced with the uncertainty on the distribution of $Y=\sigma\epsilon$ due to uncertain $\sigma$ part. 
Then the maximum expected value under this distributional uncertainty is
\begin{equation}
\label{eq:extrem-expt-semiGN}
	\sup_{\sigma\in\mysetrv{A}\sdInt}\lexpt[\varphi(\sigma\epsilon)],
\end{equation}
where the choice of $\mysetrv{A}\sdInt$ is the same as \ref{subsec:property-maximal}. 
It turns out, in either of the choices, \ref{eq:extrem-expt-semiGN} can be expressed as the sublinear expectation of a semi-$G$-normal $W\sim\semiGN(0,\varI)$ (\ref{defn:uni-semignorm}). 

To begin with, note that $\GN(0,[1,1])$ can be treated as the same as the classical distribution $\CN(0,1)$ due to \ref{rem:GN-degenerate-to-CN}. Therefore, we can also say $\epsilon \sim \GN(0,[1,1])$ in the sublinear expectation space. In the following context, we will not distinguish between $\CN(0,1)$ and $\GN(0,[1,1])$. Similarly, a standard multivariate normal $\CN(\myvec{0},\mymat{I}_d)$ can be treated as both a classical distribution and also a degenerate version of a multivariate $G$-normal.





\begin{defn}[Univariate semi-$G$-normal distribution]
\label{defn:uni-semignorm}
For any
$W\in\bar{\mysetrv{H}}_{s}$, we call $W$ follows a semi-$G$-normal
distribution $\semiGN(0,\varI)$ if there exist $V\in\bar{\myset{H}}_{s}\sim\Maximal\sdInt$
and $\stdrv\in\bar{\myset{H}}_{s}\sim\CN(0,1)$ with $V\seqind\stdrv$,
such that 
\begin{equation}
\label{eq:semignorm-decompose}
	W=V\epsilon,
\end{equation}
where the direction of independence cannot be reversed. It is denoted as $W\sim \semiGN(0,\varI)$. 
\end{defn}


\begin{rem}(Existence of Semi-$G$-normal distribution) 
\label{rem:existence-semignorm}
Since there exist $V'\in\myset{H}_{(1)}\sim\Maximal\sdInt$
and $\stdrv'\in\myset{H}_{(2)}\sim N(0,1)$, let $V(\omega_{1},\omega_{2})\coloneqq V'(\omega_{1})$
and $\stdrv(\omega_{1},\omega_{2})\coloneqq\stdrv'(\omega_{2})$.
Consider $f(x,y)=xy$, then \ref{prop:projection-K-eta} ensures that $W=f(V,\stdrv)=f(V',\stdrv')\in\bar{\mysetrv{H}}_s$ satisfies the properties required by \ref{defn:uni-semignorm}. 
\end{rem}

	

\begin{rem}(Why we cannot reverse the direction of independence)
There are two reasons: 
\begin{enumerate}
\item The sublinear expectation will essentially change if we do so: the resulting
distribution will be different.  For instance, if we assume $\epsilon\seqind V$
and let $\tilde{W}\coloneqq\epsilon V$, we have 
\begin{align*}
\expt[\tilde{W}] & =\expt[\expt[xV]_{x=\epsilon}]
  =\lexpt[\epsilon^{+}\sdr-\epsilon^{-}\sdl]\\
 & =\lexpt[\sdr\epsilon+(\sdr-\sdl)\epsilon^{-}]
  =(\sdr-\sdl)\lexpt[\epsilon^{-}] \\
 & =\frac{1}{2}(\sdr-\sdl)\lexpt[\abs{\epsilon}]>0,
\end{align*}
and similarly, 
\[
-\expt[-\tilde{W}]=-\frac{1}{2}(\sdr-\sdl)\lexpt[\abs{\stdrv}]<0.
\]
where $\stdrv^{+}=\max\{\stdrv,0\}$ and $\stdrv^{-}=\max\{-\stdrv,0\}$. We can see that $\tilde{W}$ and $W$ already exhibit their difference in the first moment: $W$ has certain mean zero but $\title{W}$ has mean-uncertainty. 
\item We can never mutual independence in this case because $V$ is maximal and $\epsilon$
is classical then it does not belong to the cases in \ref{thm:mutual-indep} proved by \cite{hu2014independence}.
\end{enumerate}
\end{rem}

As we further proceed in this paper, we will see that the property of $W$ is closely related to the random vector $(V,\stdrv)$ in its decomposition \ref{eq:semignorm-decompose} (such as the results in \ref{subsec:indep-semignorm}). 
The following \ref{prop:unique-decompose} guarantees the uniqueness of such decomposition. 


\begin{prop}[The uniqueness of decomposition] \label{prop:unique-decompose}
In \ref{defn:uni-semignorm}, if there exist two pairs $(V_1,\stdrv_1)$ and $(V_2,\stdrv_2)$ satisfying the required properties and 
\[
W=V_1\stdrv_1=V_2\stdrv_2,
\]
we must have $V_1=V_2$ and $\stdrv_1=\stdrv_2$. 
\end{prop}

%


\begin{thm}[Representations of univariate semi-$G$-normal]
\label{thm:represent-uni-semignorm} Consider two classically distributed random variables $\sigma:\Omega \to \sdInt$ and $\stdrv\sim \CN(0,1)$ satisfying $\sigma\independent \stdrv$. 
For any $\varphi \in \fspace(\numset{R})$, we have $\expt[\abs{\varphi(W)}] < \infty$ and 
\begin{align}
\expt[\varphi(W)] & =\max_{\sigma \in \mysetrvdeg\sdInt} \lexpt[\varphi(\sigma\stdrv)] = \max_{\sigma \in \sdInt} \lexpt[\varphi(\sigma\stdrv)] \label{eq:represent-semignorm-1}\\
& =\max_{\sigma\in\mysetrvdisc\sdInt}\lexpt[\varphi(\sigma\stdrv)] \label{eq:represent-semignorm-disc}\\
& =\sup_{\sigma\in\mysetrvcont\sdInt}\lexpt[\varphi(\sigma\stdrv)]\label{eq:represent-semignorm-2}\\
& =\max_{\sigma\in\mysetrv{D}\sdInt}\lexpt[\varphi(\sigma\stdrv)],\label{eq:represent-semignorm-3}
\end{align}
where $\mysetrvall$ are the same as the ones in \ref{prop:rep-uni-maximal}.
\end{thm}

The proof of \ref{thm:represent-uni-semignorm} is closely related to the representation of maximal distribution. First we need to prepare the following lemma. 

\begin{lem}
\label{lem:varphi-stdrv}
	For any fixed $v\in\sdInt$, let $\varphi_\stdrv(v)\coloneqq \expt[\varphi(v\epsilon)]$ with $\stdrv\sim \GN(0,[1,1])$. Then we have $\varphi_\stdrv \in \fspace(\numset{R})$. 
\end{lem}
\begin{proof}[Proof of \ref{lem:varphi-stdrv}]
	Note that $\stdrv \eqdistn \GN(0,[1,1]) \eqdistn \CN(0,1)$ as mentioned in \ref{rem:GN-degenerate-to-CN}. Then $\varphi_\stdrv(v)\coloneqq\expt[\varphi(v\stdrv)] = \lexpt[\varphi(v\stdrv)]$. Next we can show $\varphi_\stdrv \in \fspace(\numset{R})$ by definition: 
	\begin{align*}
|\varphi_{\stdrv}(x)-\varphi_{\stdrv}(y)| & =|\lexpt_\lprob[\varphi(x\stdrv)-\varphi(y\stdrv)]|\\
 & \leq\lexpt_\lprob[C_{\varphi}(1+\abs{x\stdrv}^{k}+\abs{y\stdrv}^{k})\abs{\stdrv}\cdot\abs{x-y}]\\
 & =C_{\varphi}(\lexpt_\lprob[\abs{\stdrv}]+\lexpt_\lprob[\abs{\stdrv}^{k+1}]\abs{x}^{k}+\lexpt_\lprob[\abs{\stdrv}^{k+1}]\abs{y}^{k})\abs{x-y}\\
 & \leq C(1+\abs{x}^{k}+\abs{y}^{k})\abs{x-y},
\end{align*}
where $C=C_\varphi\max\{\lexpt_\lprob[\abs{\stdrv}],\lexpt_\lprob[\abs{\stdrv}^{k+1}]\}$. 
\end{proof}

\begin{proof}[Proof of \ref{thm:represent-uni-semignorm}]
	 Under the sequential independence $V\seqind \epsilon$, for any $\varphi \in \fspace(\numset{R})$, we have
\[
\expt[\varphi(W)] =\expt[\varphi(V\stdrv)]
 =\expt[\expt[\varphi(v\stdrv)]_{v=V}]=\expt[\varphi_\stdrv(V)].
\]
First we have $\varphi_\stdrv \in \fspace(\numset{R})$ by \ref{lem:varphi-stdrv}.
Then we can use \ref{def:maximal} to show the finiteness of $\expt[\abs{\varphi(W)}]$ due to the continuity of $\varphi_\stdrv$:
$
\expt[\abs{\varphi(W)}] = \max_{v\in \sdInt} \abs{\varphi_\stdrv(v)} < \infty.
$
   Next we check each representation in \ref{thm:represent-uni-semignorm} by applying the associated representation of maximal distribution in \ref{prop:rep-uni-maximal}. For instance, we can show \ref{eq:represent-semignorm-2} based on \ref{eq:represent-maximal-2}: 
   \begin{align*}
   \expt[\varphi(W)]&=\expt[\varphi_\stdrv(V)] = \sup_{\sigma \in\mysetrvcont\sdInt} \lexpt[\varphi_\stdrv(\sigma)] \\ 
   & = \sup_{\sigma\in\mysetrvcont\sdInt} \lexpt[\lexpt[\varphi(v\stdrv)]_{v=\sigma}] = \sup_{\sigma\in\mysetrvcont\sdInt} \lexpt[\varphi(\sigma\stdrv)],
   \end{align*}
where we use the fact that $\sigma \independent \stdrv$ and \ref{rem:seq-indep-and-cl-indep}. 
\end{proof}

\begin{rem}
\ref{thm:represent-uni-semignorm} means that there are several ways to interpret the distributional uncertainty of semi-$G$-normal: 
\begin{itemize}
	\item \ref{eq:represent-semignorm-1} shows it can be described as a collection of $\CN(0,\sigma^2)$ with $\sigma\in\sdInt$ (which gives a direct way to compute this sublinear expectation); 
	\item \ref{eq:represent-semignorm-2,eq:represent-semignorm-disc,eq:represent-semignorm-3} show it can be described as a collection of classical normal mixture distribution with (discretely, absolutely continuously or arbitrarily) distributed scale parameter ranging in $\sdInt$.
	\end{itemize}	
\end{rem}

\begin{rem}
Let $F_\sigma$ denote the cumulative distribution function  of $\sigma$ under $\lprob$ and $F_{\myset{A}\sdInt}$ represent the family of $F_\sigma$ with $\sigma\in \myset{A}\sdInt$. Then we can apply the classical Fubini theorem in the evaluation of $\lexpt[\varphi(\sigma\stdrv)]$ in \ref{thm:represent-uni-semignorm} to get a more explicit form of representation:
\[
	\expt[\varphi(W)]= \sup_{F_\sigma \in F_{\myset{A}\sdInt}} \int_{\sdl}^{\sdr}\lexpt[\varphi(v \stdrv)] F_\sigma (\diff v),
\]
where $\myset{A}$ can be chosen from $\mysetrvall$.
\end{rem}

\begin{rem}[Why is it called a ``semi'' one?]
\label{rem:comp-semiGN-GN}
The essential reason is that the uncertainty set of distributions associated with the semi-$G$-normal is \emph{smaller} than the one of $G$-normal. 
Let $W^G\sim \GN(0,\varI)$ and $W\sim \semiGN(0,\varI)$. 
In fact, we have the following existing result: for any $\varphi\in\fspace(\numset{R})$
\begin{equation}
\label{eq:comp-GN-semiGN}
\expt[\varphi(W^G)]\geq\max_{v\in\sdInt}\lexpt[\varphi(v\stdrv)]=\expt[\varphi(W)],	
\end{equation}
which can be proved by applying the comparison theorem of parabolic partial differential equations (in \cite{crandall1992user}) to the associated $G$-heat and classical heat equations with initial condition $\varphi$ (the inequality become a strict one when $\varphi$ is neither convex nor concave). For readers' convenience, the result \ref{eq:comp-GN-semiGN} is included in Section 2.5 in \cite{peng2019nonlinear}. 
Meanwhile, we have the representation of $\expt$ from a set $\myset{P}$ of probability measures,
\[
\expt[\varphi(X)] = \sup_{\measure{Q}\in \myset{P}} \lexpt_\measure{Q}[\varphi(X)] =\sup_{\measure{Q}_X \in \myset{P}_X} \lexpt_\measure{Q}[\varphi(X)], 
\]
where $\myset{P}_X \coloneqq \{\measure{Q}\circ X^{-1},\measure{Q} \in \myset{P}\}$
 characterizes the distributional uncertainty of $X$. Hence, \ref{eq:comp-GN-semiGN} tells us $\myset{P}_{W} \subset \myset{P}_{W^G}$. A more explicit discussion of this distinction will be provided in \ref{rem:comp-semiGN-GN-2}. 
\end{rem}




\begin{rem}[The distribution of $\stdrv$]
	In principle, the distribution of $\stdrv$ can be changed to any other types of classical distribution with finite moment generating function and all the related results like representations will also hold. We choose standard normal because we are working on an intermediate structure between normal and $G$-normal. Another reason comes from the following \ref{thm:conn-G}. 
\end{rem}

\begin{prop}[A special connection between semi-$G$-normal and $G$-normal distribution] \label{thm:conn-G}
Let $\GNrv\sim \GN(0,\varI)$ and $W\sim \semiGN(0,\varI)$. 
For $\varphi\in\fspace(\numset{R})$, 
when $\varphi$ is convex or concave, we have 
\[
\expt[\varphi(\GNrv)]=\expt[\varphi(\SGNrv)]=\begin{cases}
\lexpt_\lprob[\varphi(\CN(0,\sdr^{2}))] & \varphi\text{ is convex}\\
\lexpt_\lprob[\varphi(\CN(0,\sdl^{2}))] & \varphi\text{ is concave}.
\end{cases}.
\]
\end{prop}



\subsection{Multivariate semi-$G$-normal distribution}
\label{subsec:multi-semignorm}

The definition of semi-$G$-normal distribution can be naturally extended
to multi-dimensional situation. Intuitively speaking, the multivariate semi-$G$-normal distribution can be treated as an analogue of the classical multivariate normal distribution which can be written as:
\begin{equation}
\label{eq:cl-multi-norm-transform}
	\CN(\myvec 0,\dvar ) = \dvar^{1/2}N(\myvec 0,\idtymat_{d}),
\end{equation}
 where $\idtymat_{d}$ is a $d\times d$ identity matrix and $\dvar$ is the covariance matrix. 

Let $\numset{S}_{d}^{+}$ denote the family of real-valued symmetric positive semi-definite $d\times d$ matrices.
Consider a bounded, closed and convex subset $\dvarset\subset\numset{S}^+_d$.  
For any element $\dvar \in \dvarset$, it has a non-negative symmetric square root denoted as $\dsd$. Let $\dsdset\coloneqq \dvarset^{1/2}$ which is the set of $\dvar^{1/2}$ with $\dvar \in \dvarset$. Then we can treat $\dvar$ as the covariance matrix of a classical multivariate normal distribution due to \ref{eq:cl-multi-norm-transform} and $\dvarset$ as a collection of covariance matrices. Note that $\dsdset$ is still a bounded, closed and convex set. Then a matrix-valued maximal distribution $\Maximal(\dsdset)$ can be directly extended from \ref{defn:multi-maximal}.




\begin{defn}[Multivariate Semi-$G$-normal distribution]
\label{defn:multi-semignorm}
Let a bounded, closed and convex subset $\dvarset\subset\numset{S}^+_d$
be the uncertainty set of covariance matrices and $\dsdset \coloneqq \dvarset^{1/2}$.
In a sublinear expectation space, a $d$-dimensional random vector $\myvec{W}$ follows a (multivariate) \emph{semi-$G$-normal distribution}, denoted by $\myvec{W}\sim \semiGN(\myvec{0},\dvarset)$,  if there exists a (degenerate) $G$-normal distributed $d$-dimensional random vector
\[
\myvec{\stdrv}\sim N(\myvec 0,\idtymat_d):\Omega\rightarrow\R^{d},
\]
and a $d\times d$-dimensional maximally distributed random matrix 
\[
\mymat{V} \sim \Maximal(\dsdset):\Omega\rightarrow\R^{d\times d},
\]
as well as $\myvec{\stdrv}$ is independent from $\mymat{V}$, expressed as $\mymat{V}\seqind \myvec{\stdrv}$, such that 
\[
\myvec{W} =\mymat{V}\myvec{\epsilon},
\]
where the direction of independence here cannot be reversed.
\end{defn}

\begin{rem}
  The existence of multivariate semi-$G$-normal distribution comes from the same logic as \ref{rem:existence-semignorm} (by using the existence of the $G$-distribution in a multivariate setup). 
\end{rem}

\begin{rem}
	Note that $\mymat{V}$ in \ref{defn:multi-semignorm} is a random matrix. The relation $\mymat{V}\seqind \myvec{\stdrv}$ is defined by a multivariate version of \ref{defn:G-indep}. 
\end{rem}



Similar to the discussions in \ref{subsec:property-maximal}, we can extend the notion of semi-$G$-normal distribution and its representation to multivariate siuation. 



\begin{thm}(Representation of multivariate semi-$G$-normal distribution)
	\label{thm:represent-multi-semignorm} 
	Consider the random vector $\myvec{W}$ in \ref{defn:multi-semignorm}. For any $\varphi\in\fspace(\numset{R}^d)$, we have $\expt[\abs{\varphi(\myvec{W})}]<\infty$ and
\begin{equation}
	\label{eq:represent-multi-semignorm}
	\expt[\varphi(\myvec{W})] = \sup_{\cdsd \in \mysetrv{A}(\dsdset)} \lexpt[\varphi(\cdsd\dstdrv)],
\end{equation}	
where $\myset{A}$ can be chosen from $\{\mysetrv{D},\mysetrvdisc,\mysetrvcont,\mysetrvdeg\}$
and $\sup$ can be changed to $\max$ except when $\myset{A}=\mysetrvcont$.
\end{thm}

\begin{proof}[Proof of \ref{thm:represent-multi-semignorm}]
The logic of this proof is exactly the same as the one of \ref{thm:represent-uni-semignorm} where we apply the representation of the maximal distribution $\Maximal(\dsdset)$ which can be easily checked that it has the same form as \ref{prop:rep-multi-maximal}. 
\end{proof}

\begin{rem}
\ref{thm:represent-multi-semignorm} means that there are several ways to interpret the distributional uncertainty of multivariate semi-$G$-normal $\semiGN(\myvec{0},\dvarset)$: 
\begin{itemize}
	\item it can be described as a collection of $\CN(0,\cdvar)$ with constant covariance matrix $\cdvar\in\cdvarset$;
	\item it can be described as a collection of classical \emph{multvariate normal mixture distributions} with (discretely, absolutely continuously, arbitrarily) distributed random covariance matrices (as a latent scaling variable) ranging in $\dvarset$.
	\end{itemize}	
\end{rem}


By using \ref{thm:represent-multi-semignorm}, we can conveniently study the \emph{covariance uncertainty} between the marginals of $\myvec{W}$. 
First, we can define the the upper and lower covariance between each marginal
of $\myvec{W}=(W_{1},W_{2},\dotsc,W_{d})$ as (note that $W_i$ has certain mean zero)
\[
\myupper{\gamma}(i,j)\coloneqq\expt[W_{i}W_{j}],
\]
and 
\[
\mylower{\gamma}(i,j)\coloneqq-\expt[-W_{i}W_{j}].
\]
Then these two quantities turn out to be closely related to $\dvarset$ illustrated as follows. 


\begin{prop}
[Upper and lower covariance between semi-$G$-normal marginals]
\label{prop:cov-semi-G-norm}
For
each $\cdvar\in\dvarset$, let $\Sigma_{ij}$ denote the $(i,j)$-th
entry of $\cdvar$, and 
\[
[\mylower{\Sigma}_{ij},\myupper{\Sigma}_{ij}]\coloneqq[\min_{\Sigma\in\dvarset}\Sigma_{ij},\max_{\Sigma\in\dvarset}\Sigma_{ij}].
\]
Then we have 
\[
\mylower{\gamma}(i,j)=\mylower{\Sigma}_{ij}\text{ and }\myupper{\gamma}(i,j)=\myupper{\Sigma}_{ij}.
\]
Specially speaking, we have 
\[
\sdr_{i}^{2}\coloneqq\expt[W_{i}^{2}]=\myupper{\cdvar}_{ii} \text{ and }\sdl_{i}^{2}\coloneqq-\expt[-W_{i}^{2}]=\mylower{\cdvar}_{ii}.
\]
\end{prop}

\begin{proof}
For each $(i,j)\in\{1,2,\dotsc,d\}^2$, let $f_{ij}(\myvec{W})=W_{i}W_{j}$. Then it is obvious that $f_{ij}\in\fspace(\numset{R}^{d})$.
For each $\Sigma$, let 
\[
(Y_{1},Y_{2},\dotsc,Y_{d})\coloneqq\cdsd\dstdrv.
\]
Then by applying \ref{thm:represent-multi-semignorm},
\begin{align*}
\myupper{\gamma}(i,j)=\expt[W_{i}W_{j}] & =\max_{\Sigma^{1/2}\in\dsdset}\lexpt[f_{ij}(\cdsd\dstdrv)]\\
 & =\max_{\Sigma^{1/2}\in\dsdset}\lexpt[Y_{i}Y_{j}]\\
 & =\max_{\Sigma^{1/2}\in\dsdset}\Sigma_{ij}=\myupper{\Sigma}_{ij}.
\end{align*}
Similarly we can show $\mylower{\gamma}(i,j)=\mylower{\Sigma}_{ij}$.
\end{proof}


\subsection{Three types of independence related to semi-$G$-normal distribution}
\label{subsec:indep-semignorm}

Besides the existing $G$-version independence (also called sequential independence) in \ref{defn:indep-seq}, 
this substructure of semi-$G$-normal distribution also provides the possibility to study finer structures of independence in this framework, and interestingly, we will show in \ref{subsec:represent-semignorm-indep} each type of independence is related to a family of state-space volatility models. 

We will introduce three types of independence regarding semi-$G$-normal distributions. Readers may recall the notation $\seqind$ for the independence of sequence (\ref{defn:indep-seq}).
Throughout this section, we assume $W_i = V_i\stdrv_i \eqdistn \semiGN(0,\varI)$ for $i=1,2,\dotsc,n$, which is a sequence of semi-$G$-normally distributed random variables, then accordingly $V_i\eqdistn \Maximal\sdInt$ and $\stdrv_i \eqdistn \GN(0,[1,1])$. Let 
\[
\varI \idtymat_n \coloneqq \left\{ \dvar =\left(\begin{array}{ccc}
\sigma^2_1\\
 & \ddots\\
 &  & \sigma^2_n
\end{array}\right):\sigma^2_{i}\in [\underline\sigma^2, \overline\sigma^2],i=1,2,\dotsc,n
\right\}.
\]
The identity of variance intervals is not essential and the results in this section can be easily generalized to the case $W_i\sim \semiGN(0,[\sdl_i^2,\sdr_i^2]), i=1,2,\dotsc,n$.

\begin{defn}
\label{semi-GN-ind-n}
For a sequence of semi-$G$-normal distributed random variables $\{W_i\}_{i=1}^n(=\{V_i\stdrv_i\}_{i=1}^n)$, we have three types of independence: 
	\begin{enumerate}
\item $\{W_i\}_{i=1}^n$ are \emph{semi-sequentially independent}  (denoted as $ W_1\semiseqind W_2 \semiseqind \dotsc \semiseqind W_n$) if : 
\begin{equation}
\label{semi-seq-ind-n}
	V_1 \seqind V_2 \seqind \dotsc \seqind V_n \seqind \stdrv_1\seqind \stdrv_2\seqind \dotsc \seqind \stdrv_n;
\end{equation}
\item $\{W_i\}_{i=1}^n$ are \emph{sequentially independent} (denoted as $W_1 \seqind W_2 \seqind \dotsc \seqind W_n$) if: 
\begin{equation}
	\label{seq-ind-n}
	V_1\stdrv_1\seqind V_2\stdrv_2 \seqind\dotsc\seqind V_n\stdrv_n;
\end{equation}

\item $\{W_i\}_{i=1}^n$ are \emph{fully-sequentially independent} (denoted as $W_1 \fullseqind W_2 \fullseqind \dotsc \fullseqind W_n$) if:
\begin{equation}
	\label{full-seq-ind-n}
	V_1\seqind \stdrv_1 \seqind V_2 \seqind \stdrv_2 \seqind \dotsc \seqind V_n \seqind \stdrv_n.
\end{equation}
	\end{enumerate} 
	\end{defn}
	
\begin{rem}[Compatibility with the definition of semi-$G$-normal]
\label{rem:compatible-semignorm}
 The requirement of independence to form the semi-$G$-normal distribution is simply $V_i\seqind \epsilon_i$, which is guaranteed by all the three types of independence by \ref{ind-subseq}. Furthermore, for two semi-$G$-normal object $W=V\stdrv$ and $\bar{W}=\bar{V}\bar{\stdrv}$, we can see that $W\fullseqind \bar W$ implies 
\[
(V,\stdrv)\seqind(\bar{V},\bar{\stdrv}),
\]
which further indicates 
$
W\seqind \bar{W}.
$
However, $W \fullseqind \bar W$ (or $W \seqind \bar W$) does not imply $W \semiseqind \bar W$ since the latter actually reverses the order of independence between $\stdrv$ and $\bar{V}$ in the former. 
\end{rem}

\begin{rem}(Existence of these types of independence)
It comes from the same logic used in \ref{rem:existence-semignorm} due to the existence of $n$ sequentially independent $G$-distributed random vectors.
\end{rem}

\begin{thm}
\label{thm:equiv-fullseqind}
The fully-seqential independence of $\{W_{i}\}_{i=1}^{n}$ can be equivalently defined as: 
\begin{enumerate}
\item[(F1)] The pair $(V_{i},\stdrv_{i})$ are sequentially independent: $(V_{1},\stdrv_{1})\seqind(V_{2},\stdrv_{2})\seqind\cdots\seqind(V_{n},\stdrv_{n}).$
\item[(F2)] The elements within each pair $(V_{i},\stdrv_{i})$ satisfy $V_{i}\seqind\stdrv_{i}$
with $i=1,2,\dotsc,n$.
\end{enumerate}
\end{thm}

\begin{rem}
 We add the condition (F2) only to stress the intrinsic requirement on independence from the definition of semi-$G$-normal. The main requirement of fully-sequential independence is (F1). It is also the reason why $\fullseqind$ is stronger than $\seqind$ because the latter only involves the product $V\stdrv$ but the former is about the joint vector $(V,\stdrv)$.
\end{rem}

The fully-sequential independence is a stronger version of sequential independence and it does not exhibit much difference with sequential independence in our current scope of discussion (which will be illustrated by \ref{thm:represent-n-seqind-semignorm}).

Hence, the key new type of independence here is the \emph{semi-sequential independence}, which is different from the sequential independence and also leads to different joint distribution of $(W_1,W_2,\dotsc,W_n)$. 
We will study the properties and behaviours of semi-$G$-normal under semi-sequential independence. 
Under such kind of independence, some of the intuitive properties we have in classical situation are preserved. First of all, it is actually a \emph{symmetric} independence among objects with distributional uncertainty (\ref{prop:semi-seqind-sym}). This symmetry makes it different from the sequential independence although $\semiseqind$ is defined through $\seqind$. Moreover, the joint vector of $n$ semi-sequentially independent semi-$G$-normal follows a multivariate semi-$G$-normal.
It actually provides a view on how to connect univariate and multivariate objects (under distributional uncertainty), which is a non-trivial task for $G$-normal distribution. It further provides a path to start from univariate classical normal to approach a multivariate $G$-normal (by using the multivariate semi-$G$-normal as a middle stage). This idea will be further illustrated in \ref{subsec:connect-uni-multi}.




We call it as ``semi-sequential'' independence because the only ``sequential'' requirement in the independence is $(V_1,V_2,\dotsc,V_n) \seqind (\stdrv_1,\stdrv_2,\dotsc,\stdrv_n)$ but the squential order within each vector is inessential in the sense that it can be arbitrarily switched. \ref{thm:equiv-semiseqind} elaborates this point by giving an equivalent definition. 


\begin{thm}
\label{thm:equiv-semiseqind}
The semi-sequential independence of $\{W_{i}\}_{i=1}^{n}$ can be equivalently defined as:
\begin{enumerate}
\item[(S1)] The $\epsilon$ part is independent from $V$ part: $(V_{1},V_{2},\dotsc,V_{n})\seqind(\stdrv_{1},\stdrv_{2},\dotsc,\stdrv_{n})$, 
\item[(S2)] The elements in $V$ part are sequentially independent: $\nseqind{V}{n}$, 
\item[(S3)] The elements in $\stdrv$ part are classically independent. 
\end{enumerate}
\end{thm}
\begin{rem}
	The order of independence within $V$ part in (S2) is inessential in the sense that it can be arbitrarily switched by \ref{thm:maximal-multi-relation}. Meanwhile, the order in $\stdrv$ part can also be switched due to the classical independence. Hence, this equivalence definition of semi-sequential independence indicates some intrinsic \emph{symmetry} of this relation coming from the only two categories of distributions (maximal and classical) that allow mutual independence. This point will be elaborated in the discussion of \ref{prop:semi-seqind-sym} and further formalized in \ref{thm:n-semi-seqind-sym}.
	 \end{rem}

%


To show the idea of the symmetry of $\semiseqind$, we start from a simple case with $n=2$ and include a short proof for readers to grasp the intuition. The validation of other results in this section is given in \ref{pf:subsec:indep-semignorm}. 


\begin{prop}[Symmetry in semi-sequential independence]
\label{prop:semi-seqind-sym}
The following statements are equivalent: 
\begin{enumerate}
	\item[(1)]  $W_{1}\semiseqind W_{2}$, 
	\item[(2)]   $W_{2}\semiseqind W_{1}$, 
	\item[(3)]  $(W_1,W_2)\sim \semiGN(\myvec{0},\varI \idtymat_2)$.
\end{enumerate}
\end{prop}

The proof of \ref{prop:semi-seqind-sym} relies on the following \ref{lem:four-obj-symmetry}, which is a direct consequence of \ref{prop:four-obj-seqind-gen} but we still include a separate proof from scratch to show the idea.

\begin{lem}
\label{lem:four-obj-symmetry}
The following two statements are equivalent: 
\begin{enumerate}
\item[(1)] $V_{1}\seqind V_{2}\seqind\stdrv_{1}\seqind\stdrv_{2}$, 
\item[(2)] $(V_{1},V_{2})\seqind(\stdrv_{1},\stdrv_{2})$, $V_{1}\seqind V_{2}$,
$\stdrv_{1}\seqind\stdrv_{2}$. 
\end{enumerate}
\end{lem}


\begin{proof}[Proof of \ref{lem:four-obj-symmetry}]
	 We can directly see $(1)\implies(2)$ because the independence of a sequence implies the independence among the non-overlapping subvectors as long as keeping the original order (\ref{prop:indep-sub-seq}). 
	 
	 $(2)\implies(1)$. The relation $(1)$ is equivalent to, 
\begin{enumerate}
\item $V_{1}\seqind V_{2}$,
\item $(V_{1},V_{2})\seqind\stdrv_{1}$, 
\item $(V_{1},V_{2},\stdrv_{1})\seqind\stdrv_{2}$.
\end{enumerate}
The first two are directly implied by (2). 
For a fixed scalar vector $(v_{1},v_{2},e_{1})$, let
$
H(v_{1},v_{2},e_{1})\coloneqq\expt[\varphi(v_{1},v_{2},e_{1},e_{2})].
$
Then the third one is equivalent to 
\[
\expt[\varphi(V_{1},V_{2},\stdrv_{1},\stdrv_{2})]=\expt[H(V_{1},V_{2},\stdrv_{1})].
\]
In fact, since $(V_{1},V_{2})\seqind(\stdrv_{1},\stdrv_{2})$, we have
\begin{align*}
\expt[\varphi(V_{1},V_{2},\stdrv_{1},\stdrv_{2})] & =\expt[\expt[\varphi(v_{1},v_{2},\stdrv_{1},\stdrv_{2})]_{v_{i}=V_{i},i=1,2}]\\
 & \overset{(a)}{=}\expt[\expt[\expt[\varphi(v_{1},v_{2},e_{1},\stdrv_{2})]_{e_{1}=\stdrv_{1}}]_{v_{i}=V_{i},i=1,2}]\\
 & =\expt[\expt[H(v_{1},v_{2},\stdrv_{1})]_{v_{i}=V_{i},i=1,2}]\\
 & \overset{(b)}{=}\expt[H(V_{1},V_{2},\stdrv_{1})],
\end{align*}
where $(a)$ comes from the independence $\stdrv_{1}\seqind\stdrv_{2}$
and $(b)$ is due to the relation $(V_{1},V_{2})\seqind\stdrv_{1}$. 
\end{proof}

\begin{proof}[Proof of \ref{prop:semi-seqind-sym}]
The equivalence of the three statements will be proved by this logic: $(1)\implies (2) \implies (3) \implies (1)$. 

$(1)\implies(2)$. By \ref{lem:four-obj-symmetry}, (1) indicates
\begin{equation}
\label{eq:four-obj-rel-1}
(V_{1},V_{2})\seqind(\stdrv_{1},\stdrv_{2}),V_{1}\seqind V_{2},\stdrv_{1}\seqind\stdrv_{2}.	
\end{equation}
In \ref{eq:four-obj-rel-1}, the roles in $V$ part are symmetric and so are $\stdrv$ part (due to \ref{thm:mutual-indep}). Then \ref{eq:four-obj-rel-1} is equivalent to 
\begin{equation}
\label{eq:four-obj-rel-2}
(V_{2},V_{1})\seqind(\stdrv_{2},\stdrv_{1}),V_{2}\seqind V_{1},\stdrv_{2}\seqind\stdrv_{1},	
\end{equation}
which in turn implies $W_2\semiseqind W_1$ by \ref{lem:four-obj-symmetry}. 

$(2)\implies(3)$ Let $\myvec{W}\coloneqq(W_{1},W_{2})$. Then 
\[
\myvec{W}=(V_{1}\stdrv_{1},V_{2}\stdrv_{2})=\mymat{V}\myvec{\stdrv},
\]
where $\mymat{V}=\diag(V_1,V_2)$ and $\myvec{\stdrv}=(\stdrv_1,\stdrv_2)$. 
Under the independence $(2)$, we have \ref{eq:four-obj-rel-2} by \ref{lem:four-obj-symmetry}, which further implies 
$
\mymat{V}\seqind \myvec{\stdrv}.
$
We also have $(V_1,V_2)\sim \Maximal(\sdInt^2)$ from $V_2\seqind V_1$ (\ref{thm:maximal-multi-relation}), then $\mymat{V} \sim \Maximal(\sdInt \idtymat_2)$. 
	Meanwhile, $\stdrv_2\seqind \stdrv_1$ means they are actually classically independent with the joint distribution 
	$
	\myvec{\stdrv} \sim N(\myvec 0,\idtymat_n ^2)
	$ because the distribution of $\stdrv_i$ is classical. 
	Therefore, by \ref{defn:multi-semignorm}, 
	$
	\myvec W = \mymat{V} \myvec{\stdrv} \sim \semiGN(\myvec{0}, \varI \idtymat_2).$

$(3)\implies(1).$ First, from the definition of $\myvec{W}=(W_{1},W_{2})\sim\semiGN(\myvec{0}, \varI \idtymat_2)$,
there exist 
$
\gdsd=\diag(V_{1},V_{2})\sim\Maximal(\sdInt\idtymat_2),
$
and 
$
\dstdrv=(\stdrv_{1},\stdrv_{2})\sim\CN(0,\idtymat_{2}),
$
with independence 
\begin{equation}
\gdsd\seqind\dstdrv,\label{eq:mat-V-seqind-ep}
\end{equation}
such that 
$
\myvec{W}=\mymat{V}\myvec{\stdrv}.
$
In other words, 
$
(W_{1},W_{2})=(V_{1}\stdrv_{1},V_{2}\stdrv_{2}).
$
We directly have $\stdrv_{1}$ and $\stdrv_{2}$ are classically
independent from their joint distribution. Next we study the independence
between $V_{i}$ part. Similarly, we can the joint distribution $(V_1,V_2)\sim \Maximal(\sdInt^2)$ from the distribution of $\mymat{V}$:
\begin{align*}
\expt[\varphi(V_{1},V_{2})] & =\expt[\varphi((1,1)\mymat{V})]\\
 & =\max_{B\in \sdInt \idtymat_2}\varphi((1,1)B)\\
 & =\max_{(v_{1},v_{2})\in \sdInt^2}\varphi(v_{1},v_{2}).
\end{align*}
By \ref{thm:maximal-multi-relation}, we have $V_{1}\seqind V_{2}$ (also vice versa). Note that \ref{eq:mat-V-seqind-ep}
implies 
$
(V_{1},V_{2})\seqind(\stdrv_{1},\stdrv_{2}).
$
Hence we have $W_{1}\semiseqind W_{2}$ by \ref{lem:four-obj-symmetry}.
\end{proof}



\begin{prop}[Zero sublinear covariance implies semi-sequential independence] 
\label{prop:zero-corr-semiseqind}
If $(W_1,W_2)$ follows a bivariate semi-$G$-normal and they have certain zero covariance:
\[\expt[W_1W_2] = -\expt[-W_1W_2]=0,\]
	then we have $W_1\semiseqind W_2$ (and vice versa). 
\end{prop}

\begin{proof}
	It is a direct result of \ref{prop:semi-seqind-sym} and \ref{prop:cov-semi-G-norm}. 
\end{proof}

\ref{prop:zero-corr-semiseqind} and \ref{prop:semi-seqind-sym} seem like natural results for a ``normal'' object in multivariate case, but this is the first time we establish such kind of connections within the $G$-expectation framework, because the $G$-normal distribution does not have such properties in multivariate case. For instance, for $\myvec{X}=(X_1,X_2)$ with $X_i\sim \GN(0,\varI)$. On the one hand, given the independence $X_1\seqind X_2$, $\myvec{X}$ does \emph{not} follow a bivariate $G$-normal, neither does $\mymat{A}\myvec{X}$ under any invertible transformation $\mymat{A}$. One the other hand, if $\myvec{X}$ follows a bivariate $G$-normal $\GN(\myvec{0},\varI \idtymat_2)$, we do \emph{not} have $X_1\seqind X_2$ or $X_2\seqind X_1$. These kinds of strange properties bring barriers to the understanding of $G$-normal in multivariate situations, especially on the connection univariate and multivariate objects. More details of this concern can be found in \cite{bayraktar2015normal}. Fortunately, the substructure of semi-$G$-normal provides some insights to reveal this connection. 



\ref{prop:semi-seqind-sym} can be extended to $n$ random variables. 

\begin{thm}
\label{thm:n-semi-seqind-sym}
	The following three statements are equivalent: 
	\begin{itemize}
		\item[(1)] $\nsemiseqind{W}{n}$,
		\item[(2)] $W_{k_{1}}\semiseqind W_{k_{2}}\semiseqind\cdots\semiseqind W_{k_{n}}$ for any permutation $\{k_j\}_{j=1}^n$ of $\{1,2,\dotsc,n\}$,
		\item[(3)] $(W_{1},W_{2},\dotsc,W_{n})\sim\semiGN(\myvec{0},\varI\mymat{I}_{n})$.
	\end{itemize}
\end{thm}


\begin{rem}
\ref{thm:n-semi-seqind-sym} shows the semi-$G$-normal under the the semi-sequential independence has \emph{symmetry} and \emph{compatiability with the multivariate case}.
	The underlying reason is that it takes advantage the (only) two families of distributions that allow both properties: classical normal and maximal distribution. For classical normal, we know that a bivariate normal with diagonal covariance matrix is equivalent to the (symmetric) independence between components. For maximal, the results are provided in \ref{thm:maximal-multi-relation}.  
\end{rem}



We end this session by showing the stability of semi-$G$-normal distribution under semi-sequential independence, which indicates that more analogous generalizations of results on classical normal can be discussed here.


\begin{prop}
\label{semi-GN-eqdistn}
For any $\bar W$ satisfying $\bar W \eqdistn W$ and $W \semiseqind \bar W$, we have 
\[
W+\bar W \eqdistn \sqrt 2 W.
  \]
\end{prop}

\begin{proof}
	With $W=V\stdrv$ and $\bar W =\bar{V} \bar{\stdrv}$, semi-sequential independence means: 
	\[
	V \seqind \bar{V} \seqind \stdrv \seqind \bar{\stdrv}.
	\] 
	For any $\varphi\in \fspace(\numset{R})$, first recall that $\varphi_\stdrv(v)\coloneqq \expt[\varphi(v\stdrv)]$ is in $\fspace(\numset{R})$ (\ref{lem:varphi-stdrv}). On the one hand,
	\begin{align*}
\expt[\varphi(V\stdrv+\bar{V}\bar{\stdrv})] & = \expt[\expt[\varphi(v\stdrv+\bar{v}\bar{\stdrv})]_{v=V,\,\bar{v}=\bar{V}}]
  = \expt\left[\varphi_\stdrv\left(\sqrt{v^{2}+\bar{v}^{2}}\right)\right]\\
 & = \expt\left[\expt[\varphi_\stdrv(\sqrt{v+\bar{V}})]_{v=V}\right]
  = \max_{x\in\sdInt }\max_{y\in\sdInt }\varphi_\stdrv(\sqrt{x^{2}+y^{2}}),\end{align*}
where we use the fact that $V\seqind \bar{V}$. 
On the other hand, 
\[
\expt[\varphi(\sqrt{2}W)] = \max_{x\in \sdInt} \lexpt[\varphi(\sqrt{2}x\stdrv)] = \max_{x\in \sdInt}\varphi_\stdrv(\sqrt{2}x).
\] 
Since 
\[\{\sqrt{x^{2}+y^{2}};\,(x,y)\in\sdInt^2 \}=[\sqrt{2}\sdl,\,\sqrt{2}\sdr ]=\{\sqrt{2}x;\,x\in\sdInt \},\]
we have $\expt[\varphi(W+\bar{W})]=\expt[\varphi(\sqrt{2}W)]$ for all $\varphi \in \fspace(\numset{R})$.
\end{proof}

We will further investigate the connection and distinction between these types of independence in \ref{subsec:represent-semignorm-indep} by studying their representations. 


\subsection{Representations under three types of independence}
\label{subsec:represent-semignorm-indep}

Let us come back to the story setup in \ref{subsec:setup-story} to introduce our results to general audience. 
Suppose we intend to study the dynamic of the whole observation process (which is the observable data sequence)
\[(Y_{1},Y_{2},\dots,Y_{n})=(\sigma_{1}\epsilon_{1},\sigma_{2}\epsilon_{2},\dotsc,\sigma_{n}\epsilon_{n}).\] 
Depending on the background information or knowledge (or the lackness
of reliable knowledge on the data pattern and underlying dynamic),
we may still have uncertainty on the distribution or dynamic
of $\myvec{Y}$. 
Especially in the early stage of data analysis, it is
 usually required to specify a model structure and search for the optimal one in a family of them. 
However, at this stage, how to select or distinguish the \emph{family of models} is an important and non-trivial task in statistial modeling. Suppose we assume that the underlying $\myvec{\sigma}$ process belong to a family $\mysetrv{A}_{n}\sdInt$, but some patterns of the data sequence, which could be generally quantified by $\lexpt[\varphi(\sigma_{1}\epsilon_{1},\sigma_{2}\epsilon_{2},\dotsc,\sigma_{n}\epsilon_{n})]$ for a test function $\varphi$, seems to exceed even the extreme cases in $\mysetrv{A}_{n}\sdInt$. In this case, we may tend to reject the hypothesis that $\sigma \in \mysetrv{A}_{n}\sdInt$. In this situation, we usually need to work with the maximum expected value under the uncertainty $\myvec{\sigma}\in\mysetrv{A}_{n}\sdInt$: 
\begin{equation}
\label{eq:max-lexpt-An}
	\sup_{\myvec{\sigma}\in\mysetrv{A}_{n}\sdInt}\lexpt[\varphi(\sigma_{1}\epsilon_{1},\sigma_{2}\epsilon_{2},\dotsc,\sigma_{n}\epsilon_{n})].
\end{equation}

However, in principle, $\mysetrv{A}_{n}\sdInt$ might be an infinite dimensional family of non-parametric (or semi-parametric) dynamics (due to lack of information on the undelying dynamic). In our current context of discussions, the possible choices of $\mysetrv{A}_{n}\sdInt$ include:
\begin{itemize}
	\item $\mysetrv{S}_{n}\sdInt\coloneqq \{\myvec{\sigma}: \sigma_{t}\text{ is }\sigmafield{G}_{t}\text{-measurable},
  \myvec{\sigma}_{(n)}\independent\myvec{\epsilon}_{(n)}\}$. As illustrated by \ref{fig:diagram-S}, it includes independent mixture models and a typical class of hidden Markov models \emph{without} feedback process.(In \ref{fig:diagram-S}, we omit the edge from $\sigma_1$ to $\sigma_3$ only for graphical simplicity.)  
  \item 
$\mysetrv{L}_{n}\sdInt \coloneqq \{\myvec{\sigma}:  \sigma_{t} \text{ is }\sigmafield{Y}_{t-1}\text{-measurable}\}$. As illustrated by \ref{fig:diagram-L}, it includes those state space models that the future state variable only depends on the historical observations. 
  \item $\mysetrv{L}^*_{n}\sdInt\coloneqq \bigl\{ \myvec{\sigma}: \sigma_{t} \text{ is }\sigmafield{F}_{t}\text{-measurable},
 (\sigma_{t}\independent\epsilon_{t})|\sigmafield{F}_{t-1}\}.$ As illustrated by \ref{fig:diagram-L-bar}, it contains a class of hidden Markov models with feedback process: the future state variable depends on both the previous states and observations. In \ref{fig:diagram-L,fig:diagram-L-bar}, the dashed arrows mean these are possible feedback effects. 
 \end{itemize}
 
 Note that 
\[
\mysetrv{S}_{n}\sdInt\cup\mysetrv{L}_{n}\sdInt\subset\mysetrv{L}^*_{n}\sdInt.
\]
This includes two aspects: 
\begin{itemize}
\item $\mysetrv{S}_{n}\sdInt\subset\mysetrv{L}^*_{n}\sdInt$
due to the fact that $\sigmafield{G}_{t}\subset\sigmafield{F}_{t}$
and $\myvec{\sigma}_{(n)}\independent\myvec{\epsilon}_{(n)}$;
\item $\mysetrv{L}_{n}\sdInt\subset\mysetrv{L}^*_{n}\sdInt$ because for any $\myvec{\sigma}\in \mysetrv{L}_{n}\sdInt$, given $\sigmafield{F}_{t-1}\supset \sigmafield{Y}_{t-1}$, $\sigma_t$ can be treated as a constant thus we must have $\sigma_t\independent \epsilon_t|\myset{F}_{t-1}$.

\end{itemize}

\begin{rem}
The condition $\sigma_t\independent \epsilon_t |\sigmafield{F}_{t-1}$ in $\mysetrv{L}^*_{n}\sdInt$ is equivalent to 
\[
\eta_t \independent \stdrv_t|\sigmafield{F}_{t-1},
\]
where $\eta_t \coloneqq \sigma_t -\lexpt[\sigma_t|\sigmafield{F}_{t-1}]$ is a sequence of $\numset{F}$-martingale increments.
\end{rem}
 
 In traditional statistical modeling, how to deal with the quantity \ref{eq:max-lexpt-An} is essentially a difficult task when $\mysetrv{A}_{n}\sdInt$ is highly unspecified and only contains some vague conditions on the possible design of edges (such as the additional edges in \ref{fig:diagram-L-bar} compared with \ref{fig:diagram-S}). 
 
 In this section, we will show that \ref{eq:max-lexpt-An} can be related to the $G$-expectation of a random vector with semi-$G$-normal marginals and different choice of $\mysetrv{A}_{n}\sdInt$ is corresponding to a type of independence associated with semi-$G$-normal. After transforming \ref{eq:Gnorm-cl-comp} into a $G$-expectation, it becomes more convenient to evaluate the $G$-expectation and this evaluation procedure also gives us a guidance on what should be the ``skeleton'' part to consider the extreme scenario when dealing with different forms of $\mysetrv{A}_{n}\sdInt$. 


\begin{figure}[h]
\centering
	\begin{tikzpicture}[
            > = stealth, 
            shorten > = 1pt, 
            auto,
            node distance = 2cm, 
            semithick 
        ]

        \tikzstyle{every state}=[
            draw = black,
            thick,
            fill = white,
            minimum size = 4mm
        ]
        
        \node[state](s1){$\sigma_1$};
        \node[state](o1)[below of =s1]{$Y_1$};
        \path[->] (s1) edge node {$\cdot \epsilon_1$} (o1);
        \node[state](s2)[right of =s1]{$\sigma_2$};
        \node[state](o2)[below of =s2]{$Y_2$};
        \path[->] (s2) edge node {$\cdot \epsilon_2$} (o2);
        \path[->] (s1) edge (s2);
        
        \node[state](s3)[right of =s2]{$\sigma_3$};
        \node[state](o3)[below of =s3]{$Y_3$};
        \path[->] (s3) edge node {$\cdot \epsilon_3$} (o3);
        \path[->] (s2) edge (s3);

   \end{tikzpicture}
   \caption{Diagram for $\mysetrv{S}_3\sdInt$}
   \label{fig:diagram-S}
\end{figure}

\begin{figure}[h]
\centering
	\begin{tikzpicture}[
            > = stealth, 
            shorten > = 1pt, 
            auto,
            node distance = 2cm, 
            semithick 
        ]

        \tikzstyle{every state}=[
            draw = black,
            thick,
            fill = white,
            minimum size = 4mm
        ]
        
        \node[state](s1){$\sigma_1$};
        \node[state](o1)[below of =s1]{$Y_1$};
        \path[->] (s1) edge node {$\cdot \epsilon_1$} (o1);
        \node[state](s2)[right of =s1]{$\sigma_2$};
        \node[state](o2)[below of =s2]{$Y_2$};
        \path[->] (s2) edge node {$\cdot \epsilon_2$} (o2);
        \path[dashed,->](o1) edge (s2);
        
        \node[state](s3)[right of =s2]{$\sigma_3$};
        \node[state](o3)[below of =s3]{$Y_3$};
        \path[->] (s3) edge node {$\cdot \epsilon_3$} (o3);
        \path[dashed,->](o2) edge (s3);
        \path[dashed,->](o1) edge (s3);

   \end{tikzpicture}
   \caption{Diagram for $\mysetrv{L}_3\sdInt$}
   \label{fig:diagram-L}
\end{figure}
   

\begin{figure}[h]
	\centering
\begin{tikzpicture}[
            > = stealth, 
            shorten > = 1pt, 
            auto,
            node distance = 2cm, 
            semithick 
        ]

        \tikzstyle{every state}=[
            draw = black,
            thick,
            fill = white,
            minimum size = 4mm
        ]
        
        \node[state](s1){$\sigma_1$};
        \node[state](o1)[below of =s1]{$Y_1$};
        \path[->] (s1) edge node {$\cdot \epsilon_1$} (o1);
        \node[state](s2)[right of =s1]{$\sigma_2$};
        \node[state](o2)[below of =s2]{$Y_2$};
        \path[->] (s2) edge node {$\cdot \epsilon_2$} (o2);
        \path[->] (s1) edge (s2);
        \path[dashed,->](o1) edge (s2);
        
        \node[state](s3)[right of =s2]{$\sigma_3$};
        \node[state](o3)[below of =s3]{$Y_3$};
        \path[->] (s3) edge node {$\cdot \epsilon_3$} (o3);
        \path[->] (s2) edge (s3);
        \path[dashed,->](o2) edge (s3);
        \path[dashed,->](o1) edge (s3);

   \end{tikzpicture}
   \caption{Diagram for $\mysetrv{L}^*_3\sdInt$}
   \label{fig:diagram-L-bar}
   \end{figure}

Our main result can be summarized as follows.

\begin{thm}
(Representations of $n$ semi-$G$-normal random variables under various types of independence)
\label{thm:represent-n-seqind-semignorm}
Consider $W_{i}\sim\semiGN(0,\varI),i=1,2,\dotsc,n$ and any $\varphi\in\fspace(\numset{R}^{n})$,

\begin{itemize}
\item Under semi-sequential independence:
\begin{equation}
\label{eq:cond-semiseqind}
	W_{1}\semiseqind W_{2}\semiseqind\cdots\semiseqind W_{n},
\end{equation}

we have $\expt[\abs{\varphi(\myvec{W})}]<\infty$, and
\begin{equation}
	\label{eq:represent-n-semiseqind}
	\expt[\varphi(W_1,W_2,\dotsc,W_n)]=\max_{\myvec{\sigma}\in\mysetrv{S}_{n}\sdInt}\lexpt_\lprob[\varphi(\sigma_1\stdrv_1,\sigma_2\stdrv_2,\dotsc,\sigma_n\stdrv_n)].
\end{equation}
\item Under sequential independence:
\begin{equation}
\label{eq:cond-seqind}
	W_{1}\seqind W_{2}\seqind\cdots\seqind W_{n},
\end{equation}
or fully-sequential independence: 
\begin{equation}
\label{eq:cond-fullseqind}
W_{1}\fullseqind W_{2}\fullseqind\cdots\fullseqind W_{n},	
\end{equation}
we have $\expt[\abs{\varphi(\myvec{W})}]<\infty$, and
\begin{align}
	\expt[\varphi(W_1,W_2,\dotsc,W_n)] & =\max_{\myvec{\sigma}\in\mysetrv{L}_{n}\sdInt}\lexpt_\lprob[\varphi(\sigma_1\stdrv_1,\sigma_2\stdrv_2,\dotsc,\sigma_n\stdrv_n)] \label{eq:represent-n-seqind-1}\\
	& =\max_{\myvec{\sigma}\in\mysetrv{L}^*_{n}\sdInt}\lexpt_\lprob[\varphi(\sigma_1\stdrv_1,\sigma_2\stdrv_2,\dotsc,\sigma_n\stdrv_n)] \label{eq:represent-n-seqind-2}.
\end{align}


\end{itemize}

\end{thm}

\begin{proof}[Proof of \ref{thm:represent-n-seqind-semignorm}]
	Turn to \ref{pf:subsec:represent-semignorm-indep}. 
\end{proof}

\begin{rem}
	We can only say $\expt[\varphi(\myvec{W})]=\expt[\varphi(V_i\stdrv_i,i=1,2,\dotsc,n)]$ stays the same under sequential or fully sequential independence. It does not mean these two types of independence are equivalent. Their difference might arise when we consider a more general situation $\expt[\varphi((V_i,\stdrv_i),i=1,2,\dotsc,n)]$, which is out of our current scope of discussion. 
\end{rem}

\begin{rem}
	Here we only consider $W_t$ as univariate semi-$G$-normal which can also be routinely extended to multivariate semi-$G$-normal (defined in \ref{subsec:multi-semignorm}). Then $\sigma_t$ is also requred to be changed to a matrix-valued process. 
\end{rem}

\begin{rem}
\label{rem:vision-represent}
	The vision here is that we can use the $G$-expectation of semi-$G$-version random vector under various types of independence to obtain the envelope associated with different family of model structures. With or without a kind of dependence (such as with or without the feedback), the family of models is usually infinite dimensional because, in principle, the form of the feedback dependence could be any kind of nonlinear function. 
	Nonetheless, \ref{eq:represent-n-semiseqind,eq:represent-n-seqind-1,eq:represent-n-seqind-2} tell us that, instead of going through all possible elements on the right hand side, we can move to the left side of the equation treat it as a sublinear expectation which has a convenient way to evaluate. 
	For instance, under semi-sequential independence, by \ref{thm:n-semi-seqind-semi-G-normal}, $\myvec{W}$ follows a multivariate semi-$G$-normal, then we only need to run through a finite-dimensional subset (as the ``skeleton'' part) to get the extreme scenario,
	\[
	\max_{\myvec{\sigma}\in\mysetrv{S}_{n}\sdInt}\lexpt_\lprob[\varphi(\sigma_1\stdrv_1,\sigma_2\stdrv_2,\dotsc,\sigma_n\stdrv_n)] = \max_{\myvec{\sigma}\in\sdInt^n} \lexpt_\lprob[\varphi(\sigma_1\stdrv_1,\sigma_2\stdrv_2,\dotsc,\sigma_n\stdrv_n)].
	\]
	Under sequential independence, we only need to run through an iterative algorithm to evaluate $\expt[\varphi(W_1,W_2,\dotsc,W_n)]$, which will be explained in \ref{subsec:connect-expt-CN-GN}.
\end{rem}

\begin{cor} 
\label{cor:represent-normalized-sum}
As special cases, 
under semi-sequential independence, 
we have 
\begin{equation}
	\label{eq:represent-sum-semiseqind}
	\expt[\varphi(\frac{1}{\sqrt{n}}\sum_{t=1}^{n}W_{t})]=\max_{\sigma\in\mysetrv{S}_{n}\sdInt}\lexpt_{\lprob}[\varphi(\frac{1}{\sqrt{n}}\sum_{t=1}^{n}\sigma_{t}\epsilon_{t})]
\end{equation}
Under sequential independence or fully sequential independence, we have
\begin{equation}
	\label{eq:represent-sum-seqind}
	\expt[\varphi(\frac{1}{\sqrt{n}}\sum_{t=1}^{n}W_{t})]=\max_{\sigma\in\mysetrv{L}^*_{n}\sdInt}\lexpt_{\lprob}[\varphi(\frac{1}{\sqrt{n}}\sum_{t=1}^{n}\sigma_{t}\epsilon_{t})],
\end{equation}
where $\mysetrv{L}^*_{n}\sdInt$ can be replaced by $\mysetrv{L}_{n}\sdInt$.
\end{cor}

\begin{proof}
	This a direct result of \ref{thm:represent-n-seqind-semignorm}.
\end{proof}

\begin{rem}
\label{rem:represent-sum-semiseqind}
	Under semi-sequential independence, by \ref{semi-GN-eqdistn},
	we have 
	\[
	\expt[\varphi(\frac{1}{\sqrt{n}}\sum_{t=1}^{n}W_{t})] = \expt[\varphi(W_1)],
	\]
	then we have 
	\[
	\max_{\sigma\in\mysetrv{S}_{n}\sdInt}\lexpt_{\lprob}[\varphi(\frac{1}{\sqrt{n}}\sum_{t=1}^{n}\sigma_{t}\epsilon_{t})] = \max_{\sigma\in\sdInt} \lexpt[\varphi(\sigma\stdrv)].
	\]
\end{rem}


\begin{rem} To show consistency with the existing results in the literature, 
	if we choose $\mysetrv{L}^*_{n}\sdInt$ in \ref{cor:Properties-of-Maximal}, (we can also change the
distribution of $\stdrv$ in $W$ to any applicable classical distribution,)
	then we can apply the CLT in the $G$-expectation framework to the
left handside to retrieve a result similar to the one in \cite{rokhlin2015central} (which is obtained by treating it as a discrete-time stochastic control problem):  
\[
\expt[\varphi(\GNrv)]=\lim_{n\to\infty}\expt[\varphi(\frac{1}{\sqrt{n}}\sum_{i=1}^{n}W_{i})]=\lim_{n\to\infty}\sup_{\sigma\in\mysetrv{L}^*_{n}\sdInt}\lexpt[\varphi(\frac{1}{\sqrt{n}}\sum_{i=1}^{n}\sigma_{i}\epsilon_{i})],
\]
where $\GNrv\sim \GN(0,\varI)$. When choosing $\mysetrv{L}_{n}\sdInt$, \ref{cor:represent-normalized-sum} is related to the discussion in Section 4 of \cite{fang2019limit}. It is also related to the formulation in \cite{dolinsky2012weak}, although the latter uses a different approach. 
\end{rem}

\begin{rem}[A more explicit distinction between semi-$G$-normal and $G$-normal]
\label{rem:comp-semiGN-GN-2}
Let us extend our discussion to a continuous-time version of the setup in \ref{subsec:setup-story}. 
	By \cite{denis2011function}, the distributional uncertainty of $\GN(0,\varI)$ can be explictily written as 
 \[
 \expt[\varphi(W^G)]= \sup_{\sigma \in \myset{L}\sdInt} \lexpt_\lprob[\varphi(\int_{0}^{1}\sigma_{s}dB^P_{s})],
 \]
 where $B_t^P$ is a classical Brownian motion (induced by $\stdrv_t$) under $(\Omega,\sigmafield{F}, \mathbb{F}, \lprob)$ and $\myset{L}\sdInt$ is the collection of all $\sigmafield{F}_t$-measurable processes valuing in $\sdInt$. Meanwhile, by considering the continuous-time version of \ref{rem:represent-sum-semiseqind}, the distributional uncertainty of $\semiGN(0,\varI)$ can be expressed as 
 \[
 \expt[\varphi(W)] = \sup_{\sigma \in \myset{S}\sdInt} \lexpt_\lprob[\varphi(\int_{0}^{1}\sigma_{s}dB^P_{s})],
 \]
 where $\myset{S}\sdInt$ is the collection of all $\sigmafield{G}_t$-measurable processes valuing in $\sdInt$. Note that $\myset{S}\sdInt\subset \myset{L}\sdInt$ because $\myset{S}\sdInt$ only considers those $\sigma_t$ processes that is independent from $B_t$. 
 This gives another more explicit distinction between semi-$G$-normal and $G$-normal distribution compared with \ref{rem:comp-semiGN-GN}.
\end{rem}

\begin{cor}
\label{cor:convex-case-same}
Under the setup of \ref{thm:represent-n-seqind-semignorm}, when $\varphi\in \fspace(\numset{R}^{n})$ is convex or concave,
\[
\expt[\varphi(W_{1},W_{2},\dotsc,W_{n})],
\]
will be the same under either sequential or semi-sequential independence. Furthermore, in these cases, we have
\begin{equation}
\label{eq:indep-convex-case-same}
	\expt[\varphi(W_1,W_2,\dotsc,W_n)]=\begin{cases}
	\lexpt[\varphi(\sdl\stdrv_1,\sdl\stdrv_2,\dotsc,\sdl\stdrv_n)] & \text{ when }\varphi \text{ is concave}\\
	\lexpt[\varphi(\sdr\stdrv_1,\sdr\stdrv_2,\dotsc,\sdr\stdrv_n)] & \text{ when }\varphi \text{ is convex}
\end{cases}.
\end{equation}
\end{cor}

The following result can be treated as an extension of 
\ref{thm:conn-G}. 

\begin{cor}
\label{cor:convex-case-same-2}
Let $\{W_i^G\}_{i=1}^n$ denote a sequence of nonlinearly i.i.d.
$G$-normally distributed random variables with $W_1^G \sim \GN(0,\varI)$. 
When $\varphi\in \fspace(\numset{R}^{n})$ is convex or concave, we have
\[
\expt[\varphi(W_1^G,W_2^G,\dotsc,W_n^G)] = \expt[\varphi(W_1,W_2,\dotsc,W_n)],
\]
where $W_i\sim \semiGN(0,\varI),i=1,2,\dotsc,n$ and they can be either sequentially or semi-sequentially independent.
\end{cor}



We can also prove that the representations mentioned in this paper also hold for $\varphi(x) = \ind{x\leq y}$ so that we can apply them to consider the upper probability or
\emph{capacity} induced by the sublinear expectation: $\upprob(A) = \expt[\indicator_A]$ (from \ref{defn:upprob-lowprob} and \ref{prop:H-space}). Without loss of generality, we only discuss the univariate case, which can be routinely extended to multivariate situations. 




\begin{defn}(The upper and lower cdf) 
In sublinear expectation space, the \emph{upper cdf} of a random variable $X$ is 
\[
\myupper{F}_{X}(y)\coloneqq\upprob(X\leq y)=\expt[\ind{X\leq y}], 
\]
and the \emph{lower cdf} is 
\[
\mylower{F}_{X}(y)\coloneqq\lowprob(X\leq y)=-\expt[-\ind{X\leq y}]. 
\]
\end{defn}

\begin{thm} (Representations of the upper and lower cdf)
\label{thm:represent-capacity}
Let $X$ denote a random variable in sublinear expectation space and $X^\alpha$ is a random variable in the classical probability space whose distribution is characterized by a latent variable $\alpha$. 
Suppose a representation of the sublinear expectation, 
\begin{equation}
\expt[\varphi(X)]=\sup_{\alpha\in\myset{A}}\lexpt[\varphi(X^{\alpha})],\label{eq:represent-gen}
\end{equation}
holds for any $\varphi\in\fspace(\numset{R})$. 
Then we also have the representations for the upper cdf,
\begin{equation}
\myupper{F}_X(y)=\upprob(X\leq y)=\sup_{\alpha\in\myset{A}}\lprob(X^{\alpha}\leq y),
\label{eq:represent-upcdf}	
\end{equation}
which holds for for any continuity point $y$ of $\myupper{F}_X$.  
In other words, the representation can be extended to functions in the form $\varphi(x)\coloneqq\ind{x\leq y}$. 
Meanwhile, we also have the representation for the lower cdf, 
\begin{equation}
\mylower{F}_X(y)=\lowprob(X\leq y)=\inf_{\alpha\in\myset{A}}\lprob(X^{\alpha}\leq y),
\label{eq:represent-lowcdf}	
\end{equation}
which holds for any continuity point $y$ of $\mylower{F}_X$.
\end{thm}


\begin{proof}[Proof of \ref{thm:represent-capacity}]
It is easy to show $\myupper{F}_{X}(y)$ is a monotone function, then the set of discountinuous
points is at most a countable set. Let $y$ be any continuous point
of $\myupper{F}_{X}(y)$. For any $\epsilon>0$, take $\delta$ small
enough such that, 
\[
\myupper{F}_{X}(y+\delta)-\myupper{F}_{X}(y-\delta)\leq\epsilon.
\]
Take $f$ and $g$ be two bounded continuous functions such that 
\[
f(x)=\begin{cases}
1 & x\leq y-\delta\\
\in[0,1] & y-\delta<x\leq y\\
0 & x>y
\end{cases},
\]
and 
\[
g(x)=\begin{cases}
1 & x\leq y\\
\in[0,1] & y<x\leq y+\delta\\
0 & x>y+\delta
\end{cases}.
\]
Then we have 
\[
\ind{x\leq y-\delta}\leq f(x)\leq\ind{x\leq y}\leq g(x)\leq\ind{x\leq y+\delta}.
\]
We can apply this inequality to $X^{\alpha}$ for any given $\alpha$:
\[
\lexpt[f(X^{\alpha})]\leq\lprob(X^{\alpha}\leq y)\leq\lexpt[g(X^{\alpha})],
\]
then 
\[
\sup_{\alpha\in\myset{A}}\lexpt[f(X^{\alpha})]\leq\sup_{\alpha\in\myset{A}}\lprob(X^{\alpha}\leq y)\leq\sup_{\alpha\in\myset{A}}\lexpt[g(X^{\alpha})].
\]
Note that $f,g\in\fspace(\numset{R})$ we can use the representation
\ref{eq:represent-gen} to get,
\[
\upprob(X\leq y-\delta)\leq\expt[f(X)]\leq\sup_{\alpha\in\myset{A}}\lprob(X^{\alpha}\leq y)\leq\expt[g(X)]\leq\upprob(X\leq y+\delta).
\]
Then 
\[
\myupper{F}_{X}(y)-\epsilon\leq\myupper{F}_{X}(y-\delta)\leq\sup_{\alpha\in\myset{A}}\lprob(X^{\alpha}\leq y)\leq\myupper{F}_{X}(y+\delta)\leq\myupper{F}_{X}(y)+\epsilon.
\]
Since $\epsilon>0$ can be arbitrarily small, we have proved the required result \ref{eq:represent-upcdf} for $\myupper{F}_{X}$. 
To validate the representation \ref{eq:represent-lowcdf} for $\mylower{F}_X$, we simply need to replace $\myupper{F}_{X}$ with $\mylower{F}_X$ and change $\sup$ to $\inf$ accordingly. 
\end{proof}

\begin{rem}(Notes on the continuity of $\upprob$)
	Note that \ref{thm:represent-capacity} does not require the continuity of $\upprob$: $\upprob(A_n)\to \upprob(A)$ if $A_n\to A$. Since one can easily check that $\upprob$ is automatically lower continuous: $\upprob(A_n)\uparrow \upprob(A)$ if $A_n\uparrow A$, the upper continuity ($\upprob(A_n)\downarrow \upprob(A)$ if $A_n\downarrow A$) is what we are really discussing here whenever we say the continuity of $\upprob$. Here we try to avoid the assumption on the upper continuity of $V$ which is a quite strong and restrictive one. Even under the regularity of $\expt$, we can only say the upper continuity holds for closed $A$ (Lemma 7 in \cite{denis2011function}). 
	However, when $y$ is a continuous point of $\myupper{F}_X$, consider any sequence $y_n$ converging to $y$ as $n\to \infty$, we do have $V(A_n)\to V(A)$ for sets $A_n\coloneqq\{X\leq y_n\}$ and $A\coloneqq\{X\leq y\}$; namely, $\upprob$ appears some continuity on this kind of sets. 
\end{rem}


\section{The hybrid roles and applications of semi-$G$-normal distributions}
\label{sec:hybrid-semignorm}

In this section, we will show the hybrid roles of semi-$G$-normal distributions, connecting the intuition between the classical framework and the $G$-expectation framework, by answering the four questions mentioned in the introduction. 

\subsection{How to connect the linear expectations of classical normal with $G$-normal}
\label{subsec:connect-expt-CN-GN}

In principle, it is feasible to understand the expectation of $G$-normal distribution through the structure of $G$-heat equation. Nonetheless, as a generalization of the normal distribution, it will be better if we can understand the $G$-normal distribution in a more distributional sense. 
\emph{Is it possible to understand the $G$-normal distribution from our old friend, the classical normal?} 
It is indeed a natural concern or question but this is essentially not straightforward. 
 Even for people who have partially learned the theory of the $G$-expectation framework, there usually exists several \emph{common thinking gaps} between classical normal and $G$-normal distribution. 

For instance, as mentioned in \ref{rem:comp-semiGN-GN}, for $\varphi\in \fspace(\numset{R})$,
\begin{equation}
\label{eq:Gnorm-cl-comp}
	\expt[\varphi(\GN(0,\varI))]\geq \sup _{\sigma\in\sdInt} \lexpt[\varphi\big(N(0,\sigma^2))],
	\end{equation}
  which indicates that the uncertainty set of $G$-normal distribution is larger than a class of classical normal distributions with $\sigma\in\sdInt$. Especially, \cite{hu2009explicit} shows the strict inequality that when $\varphi(x)=x^3$, we have
$\expt[\big(\GN(0,\varI)\big)^3]>0.$ (It stays positive for any odd moments.)
Let $W^G\eqdistn \GN(0,\varI)$. By checking the $G$-function defined in \ref{$G$-normal-heat-eq}, we have
\begin{equation}
	\label{eqdistn-GN}
	W^G\eqdistn -W^G,
\end{equation} 
which indicates that the $G$-normal distribution should have some ``symmetry''. However, exactly due to this identity in distribution shown in \ref{eqdistn-GN}, we should have $W^G$ and $-W^G$ share the same (sublinear) third moment: 
$
\expt[-(W^G)^3]=\expt[(-W^G)^3)]=\expt[(W^G)^3] >0
,$
which directly implies, 
\[
\expt[\big(\GN(0,\varI)\big)^3]>0=\lexpt[\big(N(0,\sigma^2)\big)^3]>-\expt[-\big(\GN(0,\varI)\big)^3]. 
 \]
It tells us that the degree of symmetry or skewness of $G$-normal distribution is uncertain, which somehow looks like a ``contradiction'' with \Cref{eqdistn-GN} and seems quite counter-intuitive for a ``normal'' distribution. 
 
Based on the above-mentioned statements showing how different the $G$-normal and classical normal are, our motivation comes from the following opposite aspect:
is this possible for us to connect the linear expectation $\lexpt[\varphi\big(N(0,\sigma^2))]$ of classical normal distribution with the sublinear expectation $\expt[\varphi(\GN(0,\varI))]$ of $G$-normal distribution (or use the former to approach the latter one)? 

This section will first give an affirmative answer to this question by providing an iterative algorithm given by our previous work (\cite{li2018iterative}) based on the semi-$G$-normal distribution. Then we are going to extend this iterative algorithm into a general computational procedure to deal with weighted summations in statistical practice.

\begin{thm}[The Iterative Approximation of the $G$-normal Distribution] 
\label{iterate-G-1} 

For any $\varphi\in \fspace(\numset{R})$ and integer $n \geq 1$, consider the series of iteration functions 
$\{\varphi_{i,n}\}_{i=1}^n$ with initial function $\varphi_{0,n}(x)\coloneqq\varphi(x)$ and iterative relation: 
\begin{equation}
\label{eq:iterate-G-1}
	\varphi_{i+1,n}(x)\coloneqq \max_{\sigma\in[\sdl,\sdr]}\lexpt_\lprob[\varphi_{i,n}(N(x,\sigma^2/n))],i=0,1,\dotsc,n-1.
\end{equation}
The final iteration function for a given $n$ is $\varphi_{n,n}$. As $n\to\infty$, we have $\varphi_{n,n}(0)\to\expt[\varphi(\GNrv)]$, where $\GNrv \sim \GN(0,\varI)$.

\end{thm}

\begin{rem}
As opposed to \ref{eq:Gnorm-cl-comp}, the relation \ref{eq:iterate-G-1} shows that, to correclty understand the sublinear expectation of $\GN(0,\varI)$, we need to start from the linear expectation of classical normal and go through an iterative maximization of the function $\varphi$ itself to approach the expectation of $\GN(0,\varI)$. For a fixed $n$, we actually have 
	\[
	\expt[\varphi(\GN(0,\varI))]\approx \max_{\sigma\in \sdInt} \lexpt_\lprob[\varphi_{n-1,n}(\CN(0,\sigma^2/n))].
	\]
\end{rem}

\begin{rem}
\label{rem:improve-iterate}
	From a computational aspect, the normal distribution in \ref{eq:iterate-G-1} can be replaced by other classical distributions with finite moment generating functions because this algorithm is based on the $G$-version central limit theorem (as indicated in \ref{prop:gen-conn-semiGN-GN}). 
	The interval $\sdInt$ can be further simplified to a two-point set $\{\sdl,\sdr\}$
or a three-point set $\{\sdl,\frac{\sdl+\sdr}{2},\sdr\}$ for computational
convenience. 
More theoretical details and numerical aspects (as well as PDE sides) of this iterative algorithm can be found in \cite{li2018iterative}.
This iterative algorithm is also related to the idea of the discrete-time formulation in \cite{dolinsky2012weak}. 
\end{rem}


\begin{rem}
Consider a sequence $\{W_i\}_{i=1}^n$ of nonlinearly i.i.d. semi-$G$-normal random variables with $W_1\sim \semiGN(0,\varI)$. 
	Each iteration function can also be expressed as the sublinear expectation of the semi-$G$-normal distribution (letting $W_0\coloneqq 0$): 
	\[
	\varphi_{i,n}(x)=\expt[\varphi(x+\sum_{j=0}^{i}\frac{W_{n-j}}{\sqrt{n}})]=\expt[\varphi(x+\sum_{j=0}^{i}\frac{W_j}{\sqrt{n}})],
	\]
	for $i=0,1,\dotsc,n$. 
	Moreover, \cite{li2018iterative} further show that the series of iteration functions is an approximation of the whole solution surface of $G$-heat equation on a given time grid. To be specific, consider the $G$-heat equation defined on $[0,\infty) \times\R$:
	\[
	u_t+G(u_{xx})=0,\,u|_{t=1}=\varphi,
	\]
where $G(a)\coloneqq \frac{1}{2}\expt[aX^{2}]=\frac{1}{2}(\sdr^{2}a^{+}-\sdl^{2}a^{-})$ and $\varphi\in \fspace(\numset{R})$.
	For each $p\in (0,1]$, we have
	\[
	|u(1-p,x) - \varphi_{\lfloor np\rfloor,n}(x)|=|\expt[\varphi(x+\sqrt p X)]-\expt[\varphi(x+\sum_{i=0}^{\lfloor np\rfloor}\frac{W_i}{\sqrt{n}})]| = C_\varphi(1+|x|^k) O(\frac{1}{(np)^{\alpha/2}}),
	\]
	where $\alpha \in (0,1)$ depending on $(\sdl,\sdr)$. 
\end{rem}

The basic idea of the iterative algorithm comes from the following result.
In the following context, without futher notice, let $\{W_i\}_{i=1}^\infty$ denote a sequence of nonlinearly i.i.d. semi-$G$-normally distributed random variables with $W_1 \sim \semiGN(0,\varI)$. 

\begin{prop}
	(A general connection between semi-$G$-normal and $G$-normal) 
	\label{prop:gen-conn-semiGN-GN}
	For any $\varphi \in \fspace(\numset{R})$, we have 
	\[
	\lim_{n\to \infty}\expt[\varphi(\frac{1}{\sqrt{n}}\sum_{i=1}^n W_i)] = \expt[\varphi(W^G)],
	\]
	where $W^G\sim \GN(0,\varI)$.
\end{prop}

\begin{rem}
	The iterative algorithm can also be extended to $d$-dimensional cases by extending the dimensions of $\{W_i\}_{i=1}^\infty$ and $W^G$ accordingly.
\end{rem}

\begin{proof}
	This is direct result of the $G$-version central limit theorem (\ref{$G$-CLT-semiGN}). We can extend the space of function $\varphi$ to $ \fspace(\numset{R})$ because the condition in \ref{prop:extend-function-space} is satisfied. 
	Let $S_{n}\coloneqq\frac{1}{\sqrt{n}}\sum_{i=1}^{n}W_{i}$.
	In fact, for any $p\geq 1$, since $f(x_1,x_2,\dotsc,x_n)=\abs{\sum_{i=1}^p x_i}^p$ is a convex function, by \ref{cor:convex-case-same}, with $\stdrv_i\sim \CN(0,1),i=1,2,\dotsc,n$, we have 
	\[
	\expt[\abs{S_n}^{p}] =\lexpt[\abs{\frac{1}{\sqrt{n}}\sum_{i=1}^{n}\sdr\stdrv_{i}}^{p}]
  =\sdr^{p}\lexpt[\abs{\stdrv_{1}}^{p}].
	\]
	Meanwhile, $\expt[\abs{W^G}^p]=\sdr^{p}\lexpt[\abs{\stdrv_{1}}^{p}]$ due to \ref{prop:convex-concave-case}.
	Hence, we have, for any
$p\in\numset{N}_{+}$,
\[
\sup_{n}\expt[\abs{S_{n}}^{p}]+\expt[\abs{W^{G}}^{p}]<\infty.\qedhere
\]
\end{proof}
Then the iterative algorithm (\ref{iterate-G-1}) can be treated as a direct evaluation of $\expt[\varphi(\frac{1}{\sqrt{n}}\sum_{i=1}^n W_i)]$. 
Interestingly, 
the uncertainty set of each $W_i$ is strictly smaller than the $G$-normal distribution by \ref{eq:Gnorm-cl-comp} but their normalized sum is able to approach the $G$-normal. 
Then it leads us to another closely related question: 
how does the uncertainty set of $\expt[\varphi(\frac{1}{\sqrt{n}}\sum_{i=1}^n W_i)]$ exactly aggregate (towards the one of $G$-normal) as $n$ increases? How does the $G$-version \emph{independence} change the uncertainty set associated with the expectation of the joint random vector 
	\[
	\expt[\varphi(W_1,W_2,\dotsc,W_n)]?
	\]
	This question has been answered by the representations shown in \ref{thm:represent-n-seqind-semignorm} and \ref{cor:represent-normalized-sum}.
	
	
We can also extend the idea of iterative algorithm into a procedure that can deal with sublinear expectation under sequential independnce in a broader sense. We call it as a \emph{$G$-EM (Expectation-Maximation) procedure} because it happens to involve expectation and maximization step (but it has no direct relation to the Expectation-Maximiazation algorithm in statistical modeling.)



One of the goals of $G$-EM procedure is to deal
with following object for a any \emph{fixed} $\varphi\in\fspace$:
\begin{equation}
\expt[\varphi(\left\langle \myvec{a},\myvec{W}\right\rangle )]=\expt[\varphi(\sum_{i=1}^{n}a_{i}W_{i})],\label{eq:iterate-weight-sum}.
\end{equation}
where $W_{i}\sim\semiGN(0,\varI),i=1,2,\dotsc,n$ are sequentially independent (the
distribution of $W_i$ could also be generalized to any member in a semi-$G$-family
of distributions which will be defined in \ref{subsec:semi-G-family}) and $\myvec{a}\in\numset{R}^{n}$ is the weight
vector. Without loss of generality, we assume the Eulicdean norm
$\norm{\myvec{a}}=1$ (or $\sum a_{i}^{2}=1$). These kinds of objects are common in data practice (in the context of financial modeling, statistics or actuarial science). We are going to give an example of a simple linear regression problem in \ref{subsec:robust-CI}.



 
%


The iterative algorithm is a special case of this, with $a_{i}=1/\sqrt{n},i=1,2,\dotsc,n$:
\[
\expt[\varphi(\frac{1}{\sqrt{n}}\sum_{i=1}^{n}W_{i})],
\]
which converges to $\expt[\varphi(\GN(0,\varI))]$ as $n\to\infty$.
 
However, in practice, using a asymptotic result may not be
feasible here for the following reasons: 
\begin{enumerate}
	\item Note that $a_{i}$ could be in arbitrary form (ususally
depend on the data or problem itself). Although we do we have results like the weighted central limit theorem proved by \cite{zhang2014weighted}, we may not always a general asysmptotic result for it.
 \item More fundamentally, $n$
could be a small number which still has a gap with the asymptotic
result. In this case, we need to
have a non-asymptoic approximation by 
involving the convergence rates of the central limiting theorem (like the Berry-Essen bound in classical case) which has been studied by \cite{fang2019limit,huang2019monotone,song2020normal,krylov2020shige}.
 \item If $n$ is small compared with the dimension
$d$ of the data, it further requires us to have a non-asymptotic
view of \ref{eq:iterate-weight-sum}. 
\end{enumerate}

Next we explain the details of the $G$-EM procedure to deal with \ref{eq:iterate-weight-sum}.
Again, under the spirit of iterative approximation, \ref{eq:iterate-weight-sum}
can be computed by the following procedure: with $\varphi_{0,n}\coloneqq\varphi$,
for $i=0,1,2,\dotsc,n-1$,
\[
\varphi_{i+1,n}(x)=\expt[\varphi(x+a_{n-i}W_{n-i})]=\max_{\sigma_{n-i}\in\sdInt}\lexpt[\varphi_{i,n}(x+a_{n-i}\sigma_{n-i}\epsilon_{n-i})].
\]
Finally we have 
\[
\expt[\varphi(\sum_{i=1}^{n}a_{i}W_{i})]=\varphi_{n,n}(0).
\]

Then we can store the optimal choice of $\sigma_{i}$ control process
for our later simulation study (then there is no need to run the iterative
algorithm again). Remember, the order we get the optimal $\sigma^{*}$
process is in the \emph{backward }order
\[
(\sigma_{n}^{*},\sigma_{n-1}^{*},\dotsc,\sigma_{1}^{*}),
\]
To follow the original order, we need to reverse it and the optimal
$\sigma^{*}$ process is in the form of 
\begin{align*}
\sigma_{1}^{*} & \in\sdInt\\
\sigma_{k}^{*} & =\sigma_{k}^{*}(\sum_{i=1}^{k-1}a_{i}W_{i}),k=2,\dotsc,n.
\end{align*}
In this way, we have 
\[
\expt[\varphi(\left\langle \myvec{a},\myvec{W}\right\rangle )]=\expt[\varphi(\sum_{i=1}^{n}a_{i}W_{i})]=\lexpt[\varphi(\sum_{i=1}^{n}a_{i}\sigma_{i}^{*}W_{i})],
\]
and the linear expectation can be approximated by a classical
Monte-Carlo simulation.

\subsection{How to connect univariate and multivariate objects}
\label{subsec:connect-uni-multi}
There are two basic properties for classical normal distribution, which brings convenience in the study of multivariate statistics. First, in $\lexptSpace$
for any two independent $X_1$ and $X_2$ both following $\CN(0,1)$, we must have $(X_1,X_2)$ form a bivariate normal. (This result still holds even if they are not independent but linearly correlated.) 
Second, a $\numset{R}^{d}$-valued random vector $\myvec{X}$
follows multivariate normal if and only if the inner product $\left\langle \myvec{a},\myvec{X}\right\rangle $
is normal for any $\myvec{a}\in\numset{R}^{d}$.
However, these two properties no longer hold for $G$-normal distributions. Readers can find the following established result in the book \cite{peng2019nonlinear} (Exercise 2.5.1),
\begin{prop}
\label{prop:ind-Gnorm}
  Suppose $X_{1}\seqind X_{2}$ and $X_{1}\eqdistn X_{2}\eqdistn N(0,\varI)$ with $\sdl < \sdr$, for $\myvec{X}\coloneqq(X_{1},X_{2})$, we have
  \begin{enumerate}
\item $\left\langle \myvec a,\myvec{X}\right\rangle $ is $G$-normal distributed for
any $\myvec{a}\in\mathbb{R}^{2}$; 
\item $\myvec{X}$ does not follow a bivariate $G$-normal distribution. 
\end{enumerate}
\end{prop}

\ref{prop:ind-Gnorm} shows we cannot construct bivariate $G$-normal distribution directly from two independent univariate $G$-normal distributed random variables. It stays unfeasible even considering any invertible linear transformations of the random vector $(X_1,X_2)$ as shown by \cite{bayraktar2015normal}, which study more details about these strange properties of $G$-normal in multidimensional case. 

To further explain the obstacle here, let us first recall that, in \ref{subsec:connect-expt-CN-GN}, we have shown 
how to start from the linear expectation $\lexpt[\varphi(\CN(0,\sigma^2))]$ of classical normal to correctly understand (also compute) the sublinear expectation $\expt[\varphi(\GN(0,\varI))]$ of $G$-normal. 
Suppose our next goal here is to help general audience further understand (or compute) the sublinear expectation $\expt[\varphi(\GN(0,\dvarset))]$ of a multivariate $G$-normal distribution with covariance uncertainty characterized by $\dvarset$ such as $\dvarset \coloneqq \{\diag(\sigma_1^2,\sigma_2^2),\sigma_i\in\sdInt,i=1,2\}$ from $(X_1,X_2)$ with $X_i\sim \GN(0,\varI)$, $i=1,2$ and $X_1\seqind X_2$. 
However, as shown in \ref{prop:ind-Gnorm}, it is difficult to achieve this goal from this path because
$\myvec{X}=(X_1,X_2)$ is not $G$-normal distributed, and neither is $\mymat{A}\myvec{X}^T$ for any invertible $2\times 2$ matrix $\mymat{A}$. 



It turns out the connection between univariate and multivariate object is essentially nontrivial.
 The contribution of this section is to show that this connection can be 
 revealed by introducing an intermediate substructure, called the semi-$G$-normal 
 imposed with \emph{semi-sequential independence}. Typically, \ref{thm:n-semi-seqind-semi-G-normal} shows that a joint vector of semi-sequentially independent univariate semi-$G$-normal follows a multivariate semi-$G$-normal (with a diagonal covariance matrix).

\begin{thm}
\label{thm:n-semi-seqind-semi-G-normal}
	For a sequence of semi-$G$-normal distributed random variables $\{W_i\}_{i=1}^n$, satisfying $W_i \sim \semiGN (0,[\underline\sigma_i^2, \overline\sigma_i^2])$ for $i=1,2,\dotsc,n$, and
\[
W_1\semiseqind W_2 \semiseqind \dotsc \semiseqind W_n,
   \]
we have 
\[
(W_1,W_2,\dotsc,W_n)^{T} \sim \semiGN (\myvec 0,\dvarset)
,\]
where $\dvarset \subset \numset{S}_d^+$ is the uncertainty set of covariance matrices defined as 
\[
\dvarset\coloneqq\left\{ \dvar =\left(\begin{array}{ccc}
\sigma^2_1\\
 & \ddots\\
 &  & \sigma^2_n
\end{array}\right):\sigma^2_{i}\in [\underline\sigma_i^2, \overline\sigma_i^2],i=1,2,\dotsc,n
\right\}.\]

\end{thm}
%

\begin{proof} It is a direct result of \ref{thm:n-semi-seqind-sym} (the non-identical variance interval here is inessential to the proof). 
\end{proof}

Next \ref{prop:multi-semi-G-normal-transform} shows we can do linear transformation on $\myvec{W}$ to get a multivariate semi-$G$-normal with a non-diagnonal covariance matrix. 

\begin{prop}[Multivariate semi-$G$-normal under linear transformation]
\label{prop:multi-semi-G-normal-transform}
	Let $\myvec W_{n\times 1}\sim \semiGN(\myvec 0,\dvarset)$. For any constant matrix $\mymat A \in \numset{R}^{r\times n}$ with $r\leq n$, we have 
	\[
	\mymat A \myvec W \sim \semiGN(\myvec 0, \mymat{A}\dvarset\mymat{A}^T),
	 \]
	where 
	\[
	\mymat{A}\dvarset\mymat{A}^T \coloneqq \left\{ \mymat{A}\dvar\mymat{A}^T: \dvar\in\dvarset \right\}\subset \numset{R}^{r\by r}.
	 \]
\end{prop}

\begin{proof}
	 First of all, note that
	$
	\mymat A_{r\by n} \myvec W_{n\by 1} = \mymat A_{r\by n} \mymat{V}_{n\by n} \myvec{\stdrv}_{n\by 1}
	$
	with $\mymat{V} \sim \Maximal(\dsdset)$. 
	For any $H\in\fspace(\R^{r\by n})$, we have 
	\[
	\expt[H(\mymat A \mymat{V})]=\max_{\dsd\in\dsdset}\lexpt_\lprob[H(\mymat{A} \dsd)]=\max_{\mymat{B} \in \mymat{A} \dsdset}\lexpt_\lprob[H(\mymat B)],
	 \] 
	so $\mymat A \mymat{V} \sim \Maximal(\mymat A \dsdset)$, which can be treated as the 
	 scaling property for the $n\times n$-dimensional maximal distribution. It follows from $\mymat{V}\seqind \myvec{\stdrv}$  that $\mymat{A}\mymat{V}\seqind \myvec{\stdrv}$. Therefore, 
	 \[
	 \mymat A \myvec W \eqdistn \Maximal(\mymat A \dsdset) N(\myvec 0,\idtymat_n^2) \sim \semiGN(\myvec 0, \dvarset^{\prime}),
	  \]
	 where 
	 \begin{align*}
	 	\dvarset^{'} &\coloneqq \left\{\mymat B \mymat B^T: \mymat B \in \mymat A \dsdset \right \} \\
	 	& = \left\{(\mymat A \dsd) (\mymat A \dsd)^T:  \dsd \in \dsdset \right \} \\
	 	& = \left\{\mymat A \dvar \mymat A^T: \dvar \in \dvarset \right \}. 
	 \end{align*}
	 In other words, 
	 $
	\mymat A \myvec W \sim \semiGN(\myvec 0, \mymat{A}\dvarset\mymat{A}^T).
	 $
\end{proof}


 


Then we can use a sequence of $\myvec W\sim \semiGN(\myvec 0, \dvarset)$ to approach the multivariate $G$-normal $\GN(\myvec 0, \dvarset)$ by nonlinear CLT. 

\begin{thm}
\label{$G$-CLT-semiGN}
	Consider a sequence of nonlinearly i.i.d. $\{\myvec{W}_i\}_{i=1}^\infty$ with $W_1\sim\semiGN(0,\dvarset)$.
	 Let $W^G$ be a $G$-normal distributed random vector following $\GN(0,\dvarset)$. Then we have, for any $\varphi \in\fspace$, 
	 \[
	\lim_{n\to\infty} \expt[\varphi(\frac{1}{\sqrt n}\sum_{i=1}^{n} \myvec{W}_i)] = \expt[\varphi(\myvec{W}^G)].
	 \]
	 It means that,
	\[
	\frac{1}{\sqrt n}\sum_{i=1}^{n} \myvec{W}_i \convergeto{\dist} \GN(0,\dvarset).
	\]
\end{thm}

\begin{proof}
This is a multivariate version of \ref{prop:gen-conn-semiGN-GN}. We only need to validate the conditions. First of all, the sequence $\{W_i\}_{i=1}^\infty$ definitely has certain zero mean.
Then, notice that the distribution of $\myvec{W}^G$ is characterized by the function $G(\mymat{A})=\frac 12 \sup_{\dvar\in\dvarset} \trace[\mymat{A}\dvar]$, where $\trace[\cdot]$ means the trace of the matrix. 
We only need to prove that $G(\mymat{A}) = \frac 12 \expt[\langle \mymat{A}W_1,W_1 \rangle]$ for any $\mymat A \in \numset{S}_d$. By the representation of semi-$G$-normal distribution, letting $\myvec{\stdrv} \sim N(0,\dvar)$, we have 
\begin{align*}
	\expt[\langle \mymat{A}W_1,W_1 \rangle] & = \sup_{\dvar\in\dvarset} \lexpt_\lprob[\langle \mymat{A}\myvec{\stdrv},\myvec{\stdrv} \rangle] = \sup_{\dvar\in\dvarset} \lexpt_\lprob[(\mymat{A}\myvec{\stdrv})^T \myvec{\stdrv}]\\ 
	& = \sup_{\dvar\in\dvarset} \lexpt_\lprob[\myvec{\stdrv}^T \mymat{A} \myvec{\stdrv}] = \sup_{\dvar\in\dvarset} \lexpt_\lprob[\trace[\myvec{\stdrv}^T \mymat{A} \myvec{\stdrv}]]\\
	& = \sup_{\dvar\in\dvarset} \lexpt_\lprob[\trace[\mymat{A} \myvec{\stdrv} \myvec{\stdrv}^T]] = \sup_{\dvar\in\dvarset} \trace[\lexpt_\lprob[\mymat{A} \myvec{\stdrv} \myvec{\stdrv}^T]] \\
	& = \sup_{\dvar\in\dvarset} \trace[\mymat{A}\lexpt_\lprob[ \myvec{\stdrv} \myvec{\stdrv}^T]] = \sup_{\dvar\in\dvarset} \trace[\mymat{A}\dvar]. \qedhere
\end{align*}
The argument on how to extend the choice of $\varphi$ to $\fspace$ is is similar to the proof in \ref{prop:gen-conn-semiGN-GN}.
\end{proof}

Then it creates a path from univariate classical normal to a multivariate $G$-normal.
\ref{fig:connection-threenorm} shows the relations among linear, semi-$G$- and $G$-normal distributions. 
We can start from the \emph{univariate} objects (semi-$G$-normal distribution), and construct its multivariate version under semi-sequential independence, then approach the \emph{multivariate} $G$-normal distribution, which gives us a feasible way to start from univariate objects to approximately approach the multivariate distribution. 

\begin{figure}[h]
\centering
\includegraphics[width=0.95\linewidth]{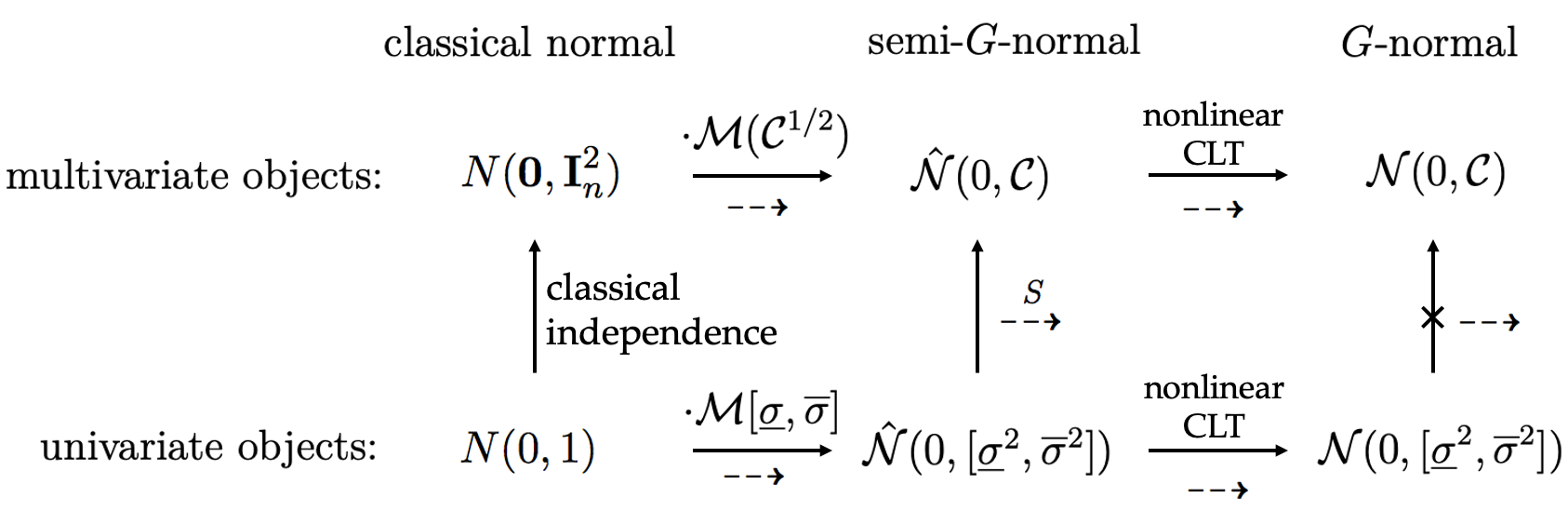}
	\caption{The relations among classical, semi-$G$-, and $G$-normal in univariate and multivariate cases}
	\label{fig:connection-threenorm}
\end{figure}

%
%
%
%
%
%
%
%

\subsection{A statistical interpretation of asymmetry in sequential independence}
\label{subsec:interpret-asymmetry-indep}

In this section, we will expand \ref{eg:indep-ex-org} (which is used to illustrate the asymmetry of independence in this framework) by studying its representation result to provide a more specific, statistical interpretation of this asymmetry. 
More interestingly, we will show that, for two semi-$G$-normally distributed random objects, each of them has \emph{certain} zero third moment (because its distributional uncertainty can be written as a family of classical normal with different variances). This property is preserved for the summation of them under semi-sequential independence. However, after we impose sequential independence onto them, their summation will exhibit \emph{third-moment uncertainty}. This phenomenon is closely related to the third moment uncertainty of the $G$-normal (as shown in \ref{subsec:connect-expt-CN-GN}) by applying the $G$-version central limit theorem (\ref{$G$-CLT}).





Next we expand \ref{eg:indep-ex-org} by considering $X$ and $Y$ as two semi-$G$-normal distributed random variables. 
\begin{eg}[Third moment uncertainty comes from asymmetry of independence]
\label{eg:two-semiGnorm-skewness}
Suppose $\gsd_1 \eqdistn \gsd_2 \sim \Maximal\sdInt$, and $\stdrv_1 \eqdistn \stdrv_2 \sim \GN(0,[1,1])$ (which is exactly the classical $\CN(0,1)$) imposed with sequential independence 
\[
\gsd_i \seqind \stdrv_i, i=1,2.
\]
Let $W_i \eqdistn \gsd_i \stdrv_i, i=1,2$, which turn out to be two identically distributed semi-$G$-normal random variables $W_1 \eqdistn W_2 \eqdistn \semiGN(0,\varI)$. Note that $(W_1,W_2)$ is a special case of $(X,Y)$ in \ref{eg:indep-ex-org}. We are going to show that, under different types of independence for $W_i$'s or different structures of sequential independence for $\gsd_i$'s and $\stdrv_i$'s, we will have different uncertainty for $W_1W_2^2$ and $(W_1+W_2)^3$ whose extreme scenarios can be described by their sublinear expectations. 

When $W_1 \semiseqind W_2$ or
\begin{equation}
	\gsd_1\seqind \gsd_2 \seqind \stdrv_1 \seqind \stdrv_2,
	\label{eq:indep-four-obj}
\end{equation}
since $W_1+W_2 \eqdistn \sqrt{2}W_1$, then
we have 
	\begin{equation}
	\label{skewness-zero}
	\expt[(W_1+W_2)^3]=\max_{v\in\sdInt} \lexpt_\lprob[(\sqrt{2}v\epsilon)^3]=0=-\expt[-(W_1+W_2)^3].
	\end{equation}
It shows that under semi-sequential independence, $W_1+W_2$ does not have third-moment uncertainty. Meanwhile, since $(W_1,W_2)$ follows bivariate semi-$G$-normal, we also have 
\[
\expt[W_1W_2^2]=\max_{(v_1,v_2)\in \sdInt^2}\lexpt_\lprob[(v_1v^2_2)\stdrv_1\stdrv_2^2]=0.
\]
Since we have shown that semi-sequential independence is symmetric (\ref{prop:semi-seqind-sym}), it means that things do not change when we consider $W_2\semiseqind W_2$ or $
\gsd_2\seqind \gsd_1 \seqind \stdrv_2 \seqind \stdrv_1.
$

However, if we only switch the order of independence between $\gsd_2$ and $\stdrv_1$ in \ref{eq:indep-four-obj} to get
\[
\gsd_1\seqind \stdrv_1 \seqind \gsd_2 \seqind \stdrv_2,
\]
which means $W_1 \fullseqind W_2$, implying $W_1\seqind W_2$. Note that $-W_i\eqdistn W_i,i=1,2$ and $-W_1\fullseqind -W_2$, so we still have 
\[
-(W_1+W_2)=(-W_1)+(-W_2)\eqdistn W_1+W_2.
\]
It indicates some ``symmetry'' in its ($G$-version) distribution. Although its second moment is uncertain, we still expect it to have some kind of ``zero skewness'' which indicates at least ``zero third moment''.
However, it turns out this is not the case:
	\begin{equation}
	\label{skewness-positive}
	\expt[(W_1+W_2)^3]=3\expt[W_1W_2^2]=3(\sdr^2-\sdl^2)\frac{\sdr}{\sqrt{2\pi}} >0 >-\expt[-(W_1+W_2)^3],
	\end{equation}
where we apply \ref{prop:linear-prop-expt} based on the facts that both $W_{1}$ and $W_{2}$ have certain
zero third moment as well as the results from \ref{eg:indep-ex-org}: \[
\expt[W_{1}W_{2}^{2}]>0\text{ while }\expt[W_{1}^{2}W_{2}]=0.
\]

\end{eg}

How can we understand the asymmetry of independence in \ref{eg:two-semiGnorm-skewness} from the representations of the sublinear expectations? This question is answered by the following \ref{prop:represent-two-semignorm}, which can be treated as a special case of \ref{thm:represent-n-seqind-semignorm}.

Let $C_{\text{s.poly}}$ denote a
basic family of bivariate polynomials:
\[
C_{\text{s.poly}}\coloneqq\{\varphi:\varphi(x_{1},x_{2})=(ax_{1}+bx_{2})^{n},\text{ or }cx_{1}^{p}x_{2}^{q},\text{ with }p,q,n\in\numset{N},a,b,c\in\numset{R}\}.
\]


\begin{prop}
(The joint distribution of two semi-$G$-normal random variables under various independence: for a small family of $\varphi$'s)
\label{prop:represent-two-semignorm}
 Consider $W_{i}\sim\semiGN(0,\varI),i=1,2$ and any $\varphi \in C_{\text{s.poly}}$,
\begin{itemize}
\item When $W_{1}\semiseqind W_{2}$, we have 
\[
\expt[\varphi(W_{1},W_{2})]=\max_{\myvec{\sigma}\in\myset{S}_2^0\sdInt}\lexpt_\lprob[\varphi(\sigma_{1}\stdrv_{1},\sigma_{2}\stdrv_{2})],
\]
where 
\begin{align*}
\myset{S}_2^0\sdInt & \coloneqq\bigl\{\myvec{\sigma}=(\sigma_{1},\sigma_{2}):(\sigma_{1},\sigma_{2})\in\{\sdl,\sdr\}^{2}\bigr\}.
\end{align*}
\item When $W_{1}\seqind W_{2}$ or $W_1\fullseqind W_2$, we have 
\[
\expt[\varphi(W_{1},W_{2})]=\max_{\myvec{\sigma}\in\myset{L}_2^0\sdInt}\lexpt_\lprob[\varphi(\sigma_{1}\stdrv_{1},\sigma_{2}\stdrv_{2})],
\]
where 
\begin{align*}
\myset{L}_2^0\sdInt & \coloneqq\bigl\{\myvec{\sigma}=(\sigma_{1},\sigma_{2}(\sigma_{1}\epsilon_{1})):\sigma_{1}\in\{\sdl,\sdr\},\\
 & \hphantom{\coloneqq\{\myvec{\sigma}=(\sigma_{1},\sigma_{2}(\sigma_{1}\epsilon_{1})):}\sigma_{2}(x)=\ind{x>0}(\sigma_{22}-\sigma_{21})+\sigma_{21},\\
 & \hphantom{\coloneqq\{\myvec{\sigma}=(\sigma_{1},\sigma_{2}(\sigma_{1}\epsilon_{1})):}(\sigma_{21},\sigma_{22})\in\{\sdl,\sdr\}^{2}\bigr\}.
\end{align*}
\end{itemize}

\end{prop}

\begin{rem}
	\ref{prop:represent-two-semignorm} provides us with the following intuitions: 
\begin{enumerate}
	\item we can directly see the difference between sequential and semi-sequential independence: under this basic setup, if $W_1\semiseqind W_2$, we can use the upper envelope of a \emph{four-element} set to represent (or compute) $\expt[\varphi(W_1,W_2)]$, while an \emph{eight-element} set is required when $W_1\seqind W_2$. Meanwhile, note that $\myset{S}_2^0\sdInt\subset \myset{L}_2^0\sdInt$: it indicates that sequential independence can cover a \emph{larger family} of models compared with the semi-sequential one. This statement is confirmed in general by \ref{thm:represent-n-seqind-semignorm};
	\item it reveals another a more intuitive insight on why $\expt[(W_1+W_2)^3]>0$ under sequential independence: the set difference $\myset{L}_2^0\sdInt\setminus\myset{S}_2^0\sdInt$ contains those elements where $\sigma_2$ actually depends on the previous $\sigma_1\epsilon_1$ (or simply the \emph{sign} of $\epsilon_1$). Although this kind of dependence does not create a shift in the mean part of $W_1+W_2$ (which is still zero), it will have strong effects on the skewness of the distribution of $W_1+W_2$. This phenomenon is related to the so-called \emph{leverage effect} in the context of financial time series analysis. In the companion of this paper,
	we will use a dual-volatility regime-switching data example to give a statistical illustration of this phenomenon and show the necessity of discussing $\myset{L}_2^0\sdInt$.
	\item we can get one specific \emph{interpretation of asymmetry} in the sequential independence from the format of $\myset{L}_2^0\sdInt$ - the role of $\sigma_1$ and $\sigma_2$ are not symmetric: when $W_1 \seqind W_2$, this sequential order means that $W_1$ is realized \emph{first} and
	the volatility part $\sigma_2\in \sdInt$ of $W_2$ may or may not depend on the value of $W_1$ so as to make the distributional uncertainty in $W_2$ unchanged. In short, when we aggregate the uncertainty set from time $1$ to $2$, due to the sequential order of the data, we can only have $\sigma_2$ is affected by a function of $(\sigma_1,\epsilon_1)$ but we almost \emph{never} have $\sigma_1$ is influenced by a function of $(\sigma_2,\epsilon_2)$ (due to the restriction from the order of time). 
As opposed to this asymmetry, the semi-sequential independence is \emph{symmetric} indicated by the form of $\myset{S}_2^0\sdInt$: the role of $\sigma_1$ and $\sigma_2$ are symmetric so we must have the same results for $\expt[W_1W_2^2]$ and $\expt[W_1^2W_2]$ under $W_1\semiseqind W_2$. 
	\item  More importantly, it also offers guidance on the possible \emph{simulation study} in this framework. When one intends to generate a data sequence that can go through the scenarios covered by a sequential independent random variables, a more cautious attitude and in-depth thought are required: different blocks of samples with \emph{separate} $\sigma_i\in\sdInt$ may only go through $\myset{S}_2^0\sdInt$ which can be at most treated as semi-sequential independence rather than sequential independence. In order to touch the latter one, one needs to generate those scenarios that allow possible classical \emph{dependence} between $\sigma_i$ and previous $(\sigma_k,\epsilon_k)$ with $k<i$. By borrowing the language of a state-space model, if we treat $\sigma_i$ as \emph{states} and $\sigma_i\epsilon_i$ as \emph{observations}, uncertainty in the \emph{dependence} between current states with previous observations needs to be considered in order to sufficiently discuss the uncertainty set covered by the \emph{sequential independence}. Otherwise, it is likely to be at most \emph{semi-sequential independence}. 
	We will further discuss this point in \ref{subsec:data-seq-G-normal} and further in the companion paper of this work.
\end{enumerate}

\end{rem}

\begin{eg}[The direction of independence comes from finer structure]
\label{eg:direct-indep-structure}
	In the $G$-framework, a symmetric (or mutual) independence between two non-constant random variables only arises if they belong to either of the following two categories: classical distribution or maximal distribution (\cite{hu2014independence}). 
	One interesting question is: how about the independence for combinations (such as products) of them? Logically speaking, if the combination does not fall into the two cases, they must have asymmetric independence, but where does this ``asymmetry'' come from? 
	To be specific, suppose we have $V\eqdistn\bar{V}\eqdistn\Maximal\sdInt$ and $\stdrv\eqdistn\bar{\stdrv}\eqdistn\GN(0,[1,1])$.
Meanwhile, assume independence $V\seqind\bar{V}$ and $\stdrv\seqind\bar{\stdrv}$.
By \ref{prop:one-indep-to-mutual-indep}, we also have $\bar{V}\seqind V$
and $\bar{\stdrv}\seqind\stdrv$. However, we do not have $W=V\epsilon$
and $\bar{W}=\bar{V}\bar{\stdrv}$ are mutually independent by \ref{thm:mutual-indep},
because $W$ and $\bar{W}$ are neither classically nor maximally
distributed. To further explain the interesting phenomenon here, we
have $V$ and $\bar{V}$ are \emph{mutually} independent, so are $\stdrv$ and
$\bar{\stdrv}$. It seems that the role of ``$V$ versus $\bar{V}$'' and
``$\stdrv$ versus $\bar{\stdrv}$'' should all be ``symmetric'' and
they do not appear any ``direction'' yet. Nonetheless, when we consider
the product $W$ and $\bar{W}$, if they are independent, we must
have the direction that either $W\seqind\bar{W}$ or $\bar{W}\seqind W$,
but there seems no middle stage where $W$ and $\bar{W}$ have some degree
of independence where their roles in this relation are symmetric. 
One question we may ask is, does there exist such kind of middle stage? 

In this example, we will give an affirmative answer to this question. 
It turns out the current conditions of independence are not enough,
the relation depends on the structure of independence among the \emph{four}
objects $V,\bar{V},\stdrv$ and $\bar{\stdrv}$. 

To be compatible with the assumed independence $V\seqind\bar{V}$ and $\stdrv\seqind\bar{\stdrv}$,
suppose we have additional sequential independence among the four objects.
There are essentially four cases: 
\begin{enumerate}
\item If $V\seqind\bar{V}\seqind\stdrv\seqind\bar{\stdrv}$, we have $W\semiseqind\bar{W},$
\item If $\bar{V}\seqind V\seqind\bar{\stdrv}\seqind\stdrv,$ we have $\bar{W}\semiseqind W,$
\item If $V\seqind\epsilon\seqind\bar{V}\seqind\bar{\stdrv}$, we have $W\seqind\bar{W},$
\item If $\bar{V}\seqind\bar{\stdrv}\seqind V\seqind\epsilon$, we have
$\bar{W}\seqind W.$ 
\end{enumerate}

Note that Case 1 and 2 are equivalent (by \ref{prop:semi-seqind-sym}). The semi-sequential independence (which is the $G$-version sequential independence) between $W$ and $\bar{W}$ is symmetric. 
From \ref{prop:one-indep-to-mutual-indep},
the first two cases are also equivalent to 
\[
\bar{V}\seqind V\seqind\stdrv\seqind\bar{\stdrv},
\]
or
\[
V\seqind\bar{V}\seqind\bar{\stdrv}\seqind\stdrv.
\]
However, Case 3 and 4 (sequential independence) are not equivalent (by \ref{eg:indep-ex-org}).

\end{eg}

\begin{eg}[The sequential independence is not an ``order'' with transitivity]
	Although sequential independence has its ``order'', it is not really an order relation with \emph{transitivity}. In other words, for three random variables $X,Y,Z\in \myset{H}$, the sequential independence $X\seqind Y$ and $Y\seqind Z$ do not necessarily imply $X\seqind Z$.
	 A trivial example is when we consider $Z_i,i=1,2$ both follow maximal distribution, if we have $Z_1\seqind Z_2$, then we have $Z_2\seqind Z_1$ (by \ref{thm:maximal-multi-relation}), but we never have $Z_1\seqind Z_1$. 
	 A non-trivial example (with three distinct random variables) comes from the fully-sequential independence structure \ref{full-seq-ind-n}. For two semi-$G$-normal objects $W=V\stdrv$ and $\bar{W}=\bar{V}\bar{\stdrv}$, $W\fullseqind \bar W$ means 
	 \[
	 V \seqind \stdrv \seqind \bar{V} \seqind \bar{\stdrv}.
	 \]
	 By \ref{ind-subseq}, we have $V\seqind \bar{V}$ and $\stdrv\seqind \bar{\stdrv}$. Then we further have the other direction of independence also holds $\bar{V}\seqind V$ and $\bar{\stdrv}\seqind \stdrv$ (by \ref{prop:one-indep-to-mutual-indep}). Then we have the following counter-example:
	 \[ \bar{\stdrv} \seqind \stdrv \text{ and } \stdrv \seqind \bar{V}\text{ but }\bar{\stdrv}\seqind \bar{V}\text{ does not hold},\]
	 because we already have $\bar{V} \seqind \bar{\stdrv}$ but the pair $(\bar{V},\bar{\stdrv})$ cannot have mutual independence by \ref{thm:mutual-indep}. 
	 Similarly, we have another example 
	 \[
	 \bar{V} \seqind V \text{ and } V\seqind \stdrv \text{ but }\bar{V}\seqind \stdrv \text{ does not hold.}
	 \]
\end{eg}

\subsection{What kinds of sequences are not $G$-normal?}
\label{subsec:data-seq-G-normal}





For examples: 
\begin{enumerate}
\item generate $Y_{i}\sim N(0,\sigma_i^{2})$ with $\sigma_i\sim\text{Unif}\sdInt$,
$i=1,2,\dotsc,n.$
This is essentially a sample from independent normal mixture (with scaling parameter following a uniform distribution). Note that this essence is not affected by the distribution of $\sigma$ (as long as $\sigma$ follows a fixed distribution). The whole data sequence $Y_i$ does not have any distributional uncertainty. 
\item first generate $\sigma_{i}\sim\text{Unif}\sdInt$, $i=1,2\dotsc,n$,
then generate $Y_{ij}\sim N(0,\sigma_{i}^{2})$ with $j=1,2,\dotsc,m$.
By introducing this blocking design, even though we pretend to treat
the switching rule of $\sigma_{i}$ as unknown here (although it may
not be so hard for data analyst to observe this pattern), if we look
at the uncertainty set considered in this generation scheme: ${\CN(0,\sigma^{2})\colon\sigma\in\sdInt}$,
it is actually at most a pseudo simulation of semi-$G$-normal distribution.
One typical feature of this sample is that it does not have skewness-uncertainty:
it has a certain zero skewness. 
\item consider equal-spaced grid 
\[
\sdl=\sigma_{1}<\sigma_{2}<\dotsc<\sigma_{m}=\sdr.
\]
For each $\sigma_{i}$ with $i=1,2,\dotsc,m$, generate $Y_{ij}\sim\CN(0,\sigma_{i}^{2})$,
$j=1,2,\dotsc,n$. Then treat 
\begin{equation}
\label{eq:max-mean-est-1}
	\max_{1\leq i\leq m}\frac{1}{n}\sum_{j=1}^{n}\varphi(Y_{ij}),
\end{equation}
as an approximation of 
$
\expt[\varphi(\GN(0,\varI)].
$
This kind of schemes has been used in some of the literature (such as \cite{deng2019stability,fei2019consistency}).
We may cautiously step back and ask ourselves: is this a valid approximation? Not really, it is actually an approximation of $\expt[\varphi(\semiGN(0,\varI)]$: 
\[
\max_{1\leq i\leq m}\frac{1}{n}\sum_{j=1}^{n}\varphi(Y_{ij}) \overset{n\to\infty}{\to} \max_{1\leq i\leq m} \lexpt[\varphi(\sigma_i\stdrv)] \overset{m\to\infty}{\to} \max_{\sigma \in \sdInt} \lexpt[\varphi(\sigma\stdrv)],
\]
where the first convergence can be treated as a classical almost sure convergence and the second one is a deterministic one due to the design of $\{\sigma_i\}_{i=1}^m$. 
This fact does not change even using some overlapping groups because each group can be at most treated as a sample from a normal mixture.
Again, the problem of the above-mentioned scheme is that it could be misleading
for general audience. It is actually going through the uncertainty
set of $\semiGN(0,\varI)$ in a semi-sequential independence rather than in sequential independence. For general $\varphi$, only in the later case, the normalized sum will be closer to the $G$-normal distribution. 
However, this issue can be fixed by consdering an extra step: if the function $\varphi$ considered in the question can be proved to be a convex or concave one, then in this practical sense, by \ref{cor:convex-case-same}, the semi-sequential independence and sequential independence can be treated as the same. 
For general fixed $\varphi$, we usually need to consider the $G$-EM procedure to do the approximation as discussed in \ref{subsec:connect-expt-CN-GN}, which is also closely related to Section 4 in \cite{fang2019limit}. Alternatively, we may consider a small family of $\varphi$'s so that we have a finite dimensional set of distributions to go through (such as the one in \ref{prop:represent-two-semignorm}). In this way, we can get a feasible approximation based on the idea of max-mean estimation by \cite{jin2021optimal} similar to the form \ref{eq:max-mean-est-1}. 
\end{enumerate}

The idea of this section will be further illustrated in the companion of this paper by using a series of data experiments.

\section{Conclusions and extensions}
\label{sec:conclude-and-ext}


For a researcher or practitioner from various backgrounds who may not be familiar with the notion of nonlinear expectation, but is comfortable the  classical probability theory and normal distribution, when 
they try to understand the $G$-normal from classical normal, it will be intuitive and beneficial if there exists an intermediate structure that can be directly transformed from the classical normal and also create a bridge towards the $G$-normal. 
Another thinking gap is from the classical independence (symmetric) to the $G$-version sequential independence (asymmetric). It will be useful if we have an intermediate stage of independence that is under distributional uncertainty but preverses the symmetry, so that it is associated with our common first impression on the relation beween two \emph{static} separate random objects both with distributional uncertainty, but no sequential order assumed. Once we talk about two objects with sequential order or in a dynamic way, it becomes possible to involve the sequential independence. 



This paper has rigorously set up and discussed the semi-$G$-normal distributions with its own types of independence, especially the semi-sequential independence. 
The hybrid roles of these new substructures, the semi-$G$-normal with its semi-sequential independence, can be summarized as follows: 
\begin{enumerate}
	\item The semi-$G$-normal $\semiGN(0,\varI)$ is closely related to the classical normal that it is simply a classical normal $\CN(0,1)$ scaled by a $G$-version constant $V\sim \Maximal\sdInt$ (with a typical independence). Then the semi-$G$-normal only exhibits the moment uncertainty with even order but its odd moments, especially the third moment (related to skewness) is preserved to be zero. Meanwhile, the semi-$G$-normal is also closely connected with the $G$-normal: they have the same sublinear expectation under a convex or concave transformation $\varphi$. For general $\varphi$, they are connected by the using the $G$-version central limit theorem. 
	\item The semi-$G$-normal $\semiGN(0,\varI)$ with semi-sequential independence also preserve the properties of classical normal in multivariate situation (\ref{thm:n-semi-seqind-sym}). 
	\item \ref{eg:direct-indep-structure} shows the hybrid and intermediate role of semi-sequential independence between classical and the sequential independence. The semi-sequential independence is related to the classical one in the sense that it is symmetric ($W_1\semiseqind W_2$ implies $W_2\semiseqind W_1$) and it is also related to the sequential independence under convex or concave $\varphi$ as illustrated in \ref{cor:convex-case-same}. 
\end{enumerate}

We can use a comparison table (\ref{tb:comparison}) to summarize the hybrid roles of this substructure: semi-$G$-normal with semi-sequential independence, creating a bridge connecting classical normal and $G$-normal. 


\begin{table}[ht]
\centering
\caption{Comparison among normal, semi-$G$-normal and $G$-normal}
\begin{tabular}{lcccc}
\toprule
 & Normal & Semi-$G$-normal & $G$-normal\tabularnewline

 & $N(0,\sigma^{2})$ & $\semiGN(0,\varI)$ & $\GN(0,\varI)$\tabularnewline

\midrule
Expectation & Linear & Sublinear & Sublinear \tabularnewline
 \hline
1st-moment & Certain (0) & Certain (0) & Certain (0)\tabularnewline

2nd-moment & Certain ($\sigma^{2}$) & Uncertain ($\varI$) & Uncertain ($\varI$)\tabularnewline

3rd-moment & Certain (0) & Certain (0) & Uncertain\tabularnewline
\hline 
Independence & $\begin{array}{c}
\text{Classical }(\independent)\\
\text{Symmetric}
\end{array}$ & $\begin{array}{c}
\text{Semi-sequential }(\semiseqind)\\
\text{Symmetric}
\end{array}$  & $\begin{array}{c}
\text{Sequential }(\seqind)\\
\text{Asymmetric}
\end{array}$ \tabularnewline
\hline 
(Setup) & $\begin{array}{c}
X\eqdistn\bar{X}\eqdistn\CN(0,\sigma^{2})\\
X\independent\bar{X}
\end{array}$ & $\begin{array}{c}
X\eqdistn\bar{X}\eqdistn\semiGN(0,\varI)\\
X\semiseqind\bar{X}
\end{array}$ & $\begin{array}{c}
X\eqdistn\bar{X}\eqdistn\GN(0,\varI)\\
X\seqind\bar{X}
\end{array}$\tabularnewline
Stability & $X+\bar{X}\eqdistn\sqrt{2}X$ 
& $X+\bar{X}\eqdistn\sqrt{2}X$ 
& $X+\bar{X}\eqdistn\sqrt{2}X$
\tabularnewline
Multivariate & $(X,\bar{X})\eqdistn\CN(\myvec{0},\sigma^{2}\idtymat_{2})$ 
& $(X,\bar{X})\eqdistn\semiGN(\myvec{0},\varI\idtymat_{2})$
& $(X,\bar{X})\overset{\text{d}}{\neq}\GN(\myvec{0},\varI\idtymat_{2})$ \tabularnewline
\bottomrule 
\end{tabular}
\label{tb:comparison}
\end{table}







Furthermore, we hope the substructures proposed in this paper will open up new extensions of discussions on the difference and connection between the $G$-expectation framework with the classical one, by providing more details on the operations on the distributions and independence in a ground where the researchers in both areas can have a proper overlapping intuition. 
In this way, we are able to have a finer discussion on the model uncertainty in the dynamic situation where the volatility part is ambiguous or cannot be precisely determined for data analysts. 

Next we will give several possible directions of extensions of this paper. 
Interestingly, the discussion in \ref{subsec:inference-ssvm} actually shows the vision that, after introducing the tool of sublinear expectation with representations in different situations, we are able to extend our horizon of statistical questions to include those that could be too complicated to handle under classical probability system. 


\subsection{The semi-$G$-family of distributions}
\label{subsec:semi-G-family}


We can extend our current notion of semi-$G$-normal distribution
to a broader class of semi-$G$-family of distributions or the class
of semi-$G$-version distributions. 

For simplicity, we only provide this notion in one-dimensional case. 

\begin{defn}
A random variable $W$ follows a semi-$G$-version distribution if
there exists a maximally distributed $\myvec{Z}\sim\Maximal(\Theta)$
(where $\Theta\subset\numset{R}^{k}$ is a bounded, closed and convex
set) and a classically distributed $\stdrv$, satisfying 
\[
Z\seqind\stdrv,
\]
such that 
\[
W=f(Z,\stdrv),
\]
for a Borel measurable function $f$ satisfying $f(Z,\stdrv)\in\myset{H}$. 
\end{defn}



\begin{rem}
The three types of independence in \ref{semi-GN-ind-n} can be carried over to members in
semi-$G$-family of distributions. For instance, with $W_{i}=f_i(Z_{i},\stdrv_{i}),i=1,2,\dotsc,n$,
they are called semi-sequentially independent if 
\[
Z_{1}\seqind Z_{2}\seqind\cdots\seqind Z_{n}\seqind\stdrv_{1}\seqind\stdrv_{2}\seqind\cdots\seqind\stdrv_{n}.
\]
\end{rem}


\begin{rem}
Most of \emph{classical }distributions should exist in this framework.
To give a quick validation, since there exists $\epsilon\sim\GN(0,[1,1])$
which follows standard normal distribution $\CN(0,1)$ with classical cumulative distribution function (cdf) $\Phi$
which is defined as 
\[
\Phi(x)\coloneqq\expt[\ind{\stdrv\leq x}]=\lexpt[\ind{\stdrv\leq x}].
\]
(Such cdf can be defined using the solution of the classical heat
equation.) Let 
\[
U\coloneqq\Phi(\epsilon).
\]
Note that $\Phi$ is a bounded and continuous function, so $U\in\myset{H}$.
Then we can check that $U$ follows
classical $\text{Unif}[0,1]$ distribution. Next we can use the classical inverse
cdf method. For any classical distribution with cdf $F$ (no matter
it is conitnuous or not), let 
\begin{equation}
	\label{eq:gen-inverse}
	F^{-1}(y)\coloneqq\inf\{x;F(x)\geq y\},
\end{equation}
denote the generalized inverse of $F$. Let $X\coloneqq F^{-1}(U).$ We only need
to add suitable conditions on $F$ so that $\expt[\abs{X}]<\infty$
to have $X\in\myset{H}$. Then we get a random object $X$ following
distribution with cdf $F$. 
\end{rem}

\begin{rem}
Note that we assume $\Theta$ to be a bounded, closed and convex set
for theoretical convenience. In practice, this condition can be weakened.
For instance, when we talk about the semi-$G$-normal random variable
$W=V\stdrv$ where $V\sim\Maximal\sdInt$, the interval $\sdInt$
can be changed to $\{\sdl,\sdr\}$ or $\{\sdl,(\sdl+\sdr)/2,\sdr\}$.
\end{rem}


\begin{eg}
	Here are several special examples of semi-$G$-version distributions. 
	\begin{enumerate}
		\item Consider $Z\sim\Maximal\sdInt$, $\stdrv\sim\CN(0,1)$ and $f(x,y)=xy$,
then $W$ follows the semi-$G$-normal distribution, whose distributional
uncertainty can be characterized by 
\[
\{\CN(0,\sigma^{2}),\sigma\in\sdInt\}.
\]

\item Consider $Z=(U,V)\sim\Maximal(\meanInt\times\sdInt)$, $\epsilon\sim\CN(0,1)$
and 
\[
W=U+V\stdrv.
\]
Then the distributional uncertainty of $W$ can be described as 
\[
\{\CN(\mu,\sigma^{2}),\mu\in\meanInt,\sigma\in\sdInt\}.
\]
We can also show that $\expt[\varphi(W)]$ can cover a family of normal
mixture models. 
\item (Semi-$G$-exponential) Let $Z\sim\Maximal[\mylower{\lambda},\myupper{\lambda}]$ and $\epsilon\sim\exp(1)$.
Consider 
\[
W=Z\stdrv,
\]
then we can check that the distributional uncertainty of $W$ can
be written as 
\[
\{\exp(\lambda),\lambda\in[\mylower{\lambda},\myupper{\lambda}]\},
\]
where each $\exp(\lambda)$ has pdf $f(x)=\frac{1}{\lambda}e^{-x/\lambda}\ind{x\geq 0}$. 
\item (Semi-$G$-Bernoulli) With $0\leq\mylower{p}\leq\myupper{p}\leq1$,
let $\stdrv\sim\text{Unif}[0,1]$, and 
\[
Z\sim\Maximal[\mylower{p},\myupper{p}].
\]
Consider 
\[
W=\ind{\epsilon-Z<0}.
\]
Then $W$ has distributional uncertainty
\[
\{\text{Bern}(p),p\in[\mylower{p},\myupper{p}]\}.
\]
\end{enumerate}
	
\end{eg}


\begin{eg}
	In general, 
	 we can take advantage of the idea of classical inverse cdf method to design the transformation $f$. Then we are able to consider any distributional
uncertainty in the form of 
\begin{equation}
\{F(x;\theta),\theta\in\Theta\},\label{eq:classical-distn-uncertainty}
\end{equation}
where $F(x;\theta)$ is the cdf of a classical distribution with parameter
$\theta$. Let $F^{-1}(y;\theta)$ denote the generalized inverse
of $F(x;\theta)$ as shown in \ref{eq:gen-inverse}. Consider $\myvec{Z}\sim\Maximal(\Theta)$ and $\stdrv\sim\text{Unif}[0,1]$,
and 
\[
W=F^{-1}(\stdrv,\myvec{Z}).
\]
After we add more conditions on $F$ such that $W\in \myset{H}$, 
we have $W$ has distributional uncertainty in the form \ref{eq:classical-distn-uncertainty},
because
\[
\expt[\varphi(W)]=\sup_{\myvec{z}\in\Theta}\lexpt[\varphi(F^{-1}(\stdrv,z))],
\]
where $F^{-1}(\stdrv,z)$ follows the distribution with cdf $F(x;z)$. 
	
\end{eg}


To further study the properties of semi-$G$-version distributions and the semi-sequential independence, 
let $\bar{\myset{H}}_{s}$ denote the semi-$G$-family of distributions: 
\[\bar{\myset{H}}_{s}\coloneqq\{X\in\myset{H}:X=f(V,\stdrv),V\sim\Maximal(\Theta),\stdrv\text{ classical},V\seqind\stdrv\}.\]
Note that $\bar{\myset{H}}_{s}$ satisfies: 
\begin{enumerate}
\item If $X\in\bar{\myset{H}}_{s}$, $aX\in\bar{\myset{H}}_{s}$ for any
$a\in\numset{R}$, 
\item If $X\in\bar{\myset{H}}_{s}$, $\abs{X}\in\bar{\myset{H}}_{s}$, 
\item For $X,Y\in\text{\ensuremath{\bar{\myset{H}}_{s}}}$, if $X\semiseqind Y$,
$X+Y\in\bar{\myset{H}}_{s}.$  
\end{enumerate}

For any $X,Y\in\bar{\myset{H}}_{s}$, we have $X\semiseqind Y$
is equivalent to $Y\semiseqind X$ (by the symmetry of semi-sequential
independence as illustrated by \ref{prop:semi-seqind-sym}), so we can omit the direction between $X$ and $Y$
and also call the mutual semi-sequential independence between them
as semi-$G$-independence. 

\begin{defn}
\label{def:semiGindep}(Semi-$G$-independence) For any $X,Y\in\bar{\myset{H}}_{s}$,
with $X=f(V_{x},\stdrv_{x})$ and $Y=g(V_{y},\stdrv_{y})$, $X$ and
$Y$ are semi-$G$-independent if 
\begin{enumerate}
\item $(V_{x},V_{y})\seqind(\stdrv_{x},\stdrv_{y})$, 
\item $V_{x}\seqind V_{y}$ (which is equivalent to $V_{y}\seqind V_{x}$), 
\item $\stdrv_{x},\stdrv_{y}$ are classically independent.
\end{enumerate}
\end{defn}

\begin{defn}
\label{def:semiGindep-n}(Semi-$G$-independence of sequence) For a sequence $\{X_{i}\}_{i=1}^{n}\subset\bar{\myset{H}}_{s}$
with $X_{i}=f_i(V_{i},\stdrv_{i})$, they are (mutually)
semi-$G$-independent if 
\begin{enumerate}
\item $(V_{1},V_{2},\dotsc,V_{n})\seqind(\stdrv_{1},\stdrv_{2},\dotsc,\stdrv_{n})$, 
\item $\{V_{i}\}_{i=1}^{n}$are $G$-version (sequentially) independent
(that is $\nseqind{V}{n}$), 
\item $\{\stdrv\}_{i=1}^{n}$ are classically independent.
\end{enumerate}
\end{defn}
A sequence $\{X_{i}\}_{i=1}^{n}$ are called \emph{semi-$G$-version
i.i.d.} (or semi-$G$-i.i.d.) if they are identically distributed
and semi-$G$-independent. 

In the following context, consider $X,Y\in\bar{\myset{H}}_{s}$, with $X=f(V_{x},\stdrv_{x})$
and $Y=g(V_{y},\stdrv_{y})$. 
\begin{prop}
\label{prop:joint-two}If $X$ and $Y$ are semi-$G$-independent,
for the joint vector $(X,Y)$, we have for any $\varphi\in\fspacedef$,
\[
\expt[\varphi(X,Y)]=\sup_{(v_{x},v_{y})\in[\sdl_{x},\sdr_{x}]\times[\sdl_{y},\sdr_{y}]}\lexpt[\varphi(f(v_{x},\stdrv_{x}),g(v_{y},\stdrv_{y}))].
\]
\end{prop}

\begin{proof}
This is direct consequence of the definition of the semi-$G$-independence. 
\end{proof}

For any $v_{x}\in[\sdl_{x},\sdr_{x}]$ and $v_{y}\in[\sdl_{y},\sdr_{y}]$,
let 
\[
h_{1}(v_{x})\coloneqq\lexpt[f(v_{x},\stdrv_{x})]\text{ and }h_{2}(v_{y})=\lexpt[g(v_{y},\stdrv_{y})].
\]
Then 
\[
\expt[X]=\expt[h_{1}(V_{x})]\text{ and }\expt[Y]=\expt[h_{1}(V_{y})].
\]

In the following context, for simplicty of discussion, we assume $h_1, h_2$ are continuous functions. Then we can take maximum on the rectangle $[\sdl_{x},\sdr_{x}]\times[\sdl_{y},\sdr_{y}]$
in \ref{prop:joint-two}. 
(This assumption can be relaxed whenever $\sup$ does not affect the derivation.) 
Readers may find out that, under the semi-$G$-independence, our manipulation of sublinear expectation of semi-$G$-version objects becomes quite intuitive and flexible. 

\begin{prop}
\label{prop:sum-semiseqind}
If $X$ and $Y$ are semi-$G$-independent, we have 
\[
\expt[X+Y]=\expt[X]+\expt[Y].
\]
\end{prop}

\begin{proof}
By \ref{prop:joint-two}, we have 
\begin{align*}
\expt[X+Y] & =\max_{(v_{x},v_{y})\in[\sdl_{x},\sdr_{x}]\times[\sdl_{y},\sdr_{y}]}\lexpt[f(v_{x},\stdrv_{x})+g(v_{y},\stdrv_{y})]\\
 & =\max_{(v_{x},v_{y})\in[\sdl_{x},\sdr_{x}]\times[\sdl_{y},\sdr_{y}]}[h_{1}(v_{x})+h_{2}(v_{y})]\\
 & =\max_{v_{x}}\max_{v_{y}}[h_{1}(v_{x})+h_{2}(v_{y})]\\
 & =\max_{v_{x}}[h_{1}(v_{x})+\max_{v_{y}}h_{2}(v_{y})]\\
 & =\max_{v_{x}}h_{1}(v_{x})+\max_{v_{y}}h_{2}(v_{y})=\expt[X]+\expt[Y].
\end{align*} \qedhere
\end{proof}

\begin{rem}
Compared with \ref{prop:linear-prop-expt}, for $X,Y\in\bar{\myset{H}}_{s}$, we have one more situation for $\expt[X+Y]=\expt[X]+\expt[Y]$ to hold: 
\begin{enumerate}
\item either $X$ or $Y$ has mean-certainty,  
\item $X\seqind Y$ or $X\seqind Y$, 
\item $X$ and $Y$ are semi-$G$-independent. 
\end{enumerate}
\end{rem}

\begin{prop}
\label{prop:prod-semiG}If $X$ and $Y$ are semi-$G$-independent
and either one of them has certain mean zero, we have
\[
\expt[XY]=-\expt[-XY]=0.
\]
\end{prop}

\begin{proof}
Since $X$ and $Y$ are semi-$G$-independent, by \ref{prop:joint-two},
\begin{align*}
\expt[XY] & =\max_{(v_{x},v_{y})\in[\sdl_{x},\sdr_{x}]\times[\sdl_{y},\sdr_{y}]}\lexpt[f(v_{x},\stdrv_{x})g(v_{y},\stdrv_{y})]\\
 & =\max_{(v_{x},v_{y})\in[\sdl_{x},\sdr_{x}]\times[\sdl_{y},\sdr_{y}]}h_{1}(v_{x})h_{2}(v_{y}).
\end{align*}
If we have either one of them has certain mean zero such as $\expt[X]=-\expt[-X]=0$,
we have 
\[
\max_{v_{x}\in[\sdl_{x},\sdr_{x}]}\lexpt[f(v_{x},\stdrv_{x})]=\min_{v_{x}\in[\sdl_{x},\sdr_{x}]}\lexpt[f(v_{x},\stdrv_{x})]=0,
\]
It means $h_{1}(v_{x})=0$ for any $v_{x}\in[\sdl_{x},\sdr_{x}]$.
Then we must have $\expt[XY]=0$ and similarly we have $-\expt[-XY]$
by changing $\max$ to $\min$.  
\end{proof}


\subsection{The semi-$G$-version of central limit theorem}
\label{subsec:semi-G-clt} 

After setting up the semi-$G$-family of distributions and semi-sequential independence, it turns out we can prove a semi-$G$-version of \emph{central limit theorem} in this context, which further brings a substructure connecting the classical central limit theorem with the $G$-version central limit theorem. 
It also shows the central role of semi-$G$-normal in a semi-$G$-version class of distributions.  
 
First we consider a subset of $\bar{\myset{H}}_s$:
\[
\myset{H}_{s}\coloneqq\{X\in\bar{\myset{H}}_{s}:X=V\stdrv,V\sim\Maximal\sdInt\text{ with }0\leq\sdl\leq\sdr\text{ and the classical }\stdrv\text{ is standardized}\},
\]
where we call a classical $\stdrv$ is \emph{standardized} if $\lexpt[\stdrv]=0$ and $\lexpt[\stdrv^{2}]=1$.
Here $\myset{H}_{s}$ can be treated as a class of semi-$G$-distributions with zero mean and variance uncertainty. 

Our current version of the semi-$G$-version of central limit theorem can be formulated as follows. 

\begin{thm}
\label{thm:semi-G-CLT} For any sequence $\{X_{i}\}_{i=1}^{n} = \{V_i\eta_i\}_{i=1}^{n}\subset\myset{H}_{s}$
that are semi-$G$-i.i.d. with certain zero mean and uncertain variance:
\[
\sdl^{2}=-\expt[-X_{1}^{2}]\leq\expt[X_{1}^{2}]=\sdr^{2}\text{ with }0\leq\sdl\leq\sdr,
\]
we have 
\[
\frac{1}{\sqrt{n}}\sum_{i=1}^{n}X_{i}\convergeto{\dist}W,
\]
where $W\sim\semiGN(0,\varI)$. To be specific, for any bounded and continuous $\varphi$, we have
\begin{equation}
\label{eq:clt-expt-converge}
	\lim_{n\to\infty}\expt[\varphi(\frac{1}{\sqrt{n}}\sum_{i=1}^{n}X_{i})]=\expt[\varphi(W)].
\end{equation}
\end{thm}

\begin{rem}
	Note that any $\varphi\in \fspacedef$ must be a bounded and continuous one, so the convergence in distribution (\ref{defn:G-distn}) must hold. 
\end{rem}

\begin{rem}
As a classical perspective of \ref{thm:semi-G-CLT}, by using the
representation of semi-$G$-normal under semi-$G$-independence, we
have 
\begin{equation}
\lim_{n\to\infty}\sup_{\myvec{\sigma}\in[\underline{\sigma},\overline{\sigma}]^{n}}E[\varphi(\frac{1}{\sqrt{n}}\sum_{i=1}^{n}\sigma_{i}\epsilon_{i})]=\sup_{v\in[\underline{\sigma},\overline{\sigma}]}E[\varphi(v\stdrv^{*})],\label{eq:semi-G-CLT-goal-1}
\end{equation}
where $\stdrv^{*}\sim\CN(0,1)$ and $\myvec{\sigma}=(\sigma_{1},\dotsc,\sigma_{n})$
is a scalar vector. This form is also equivalent to
\[
\lim_{n\to\infty}\sup_{\myvec{\sigma}\in S_{n}\sdInt}E[\varphi(\frac{1}{\sqrt{n}}\sum_{i=1}^{n}\sigma_{i}\epsilon_{i})]=\sup_{v\in[\underline{\sigma},\overline{\sigma}]}E[\varphi(v\stdrv^{*})],
\]
where $\myvec{\sigma}$ could be any hidden processes valuing in $\sdInt$
that is independent from $(\stdrv_{1},\stdrv_{2},\dotsc,\stdrv_{n})$. 
When the unknown variance form is taken in this way, the uncertainty on the behavior the normalized summation can be asymptotically characterized by the semi-$G$-normal. 

If $\myvec{\sigma}$ is chosen from a larger family that may involve
dependence between $\sigma_{i}$ with previous $(\stdrv_{j},j<i)$,
then it will be related with the $G$-version central limit theorem
(under sequential independence, rather than the semi-$G$-version
independence): $X_{i}$ are \emph{sequentially independent,}
\[
\lim_{n\to\infty}\expt[\varphi(\frac{1}{\sqrt{n}}\sum_{i=1}^{n}X_{i})]=\expt[\varphi(W^{G})],
\]
which gives us
\[
\lim_{n\to\infty}\sup_{\myvec{\sigma}\in \mysetrv{L}_{n}^{*}\sdInt}E[\varphi(\frac{1}{\sqrt{n}}\sum_{i=1}^{n}\sigma_{i}\epsilon_{i})]=\expt[\varphi(W^{G})].
\]
To summarize, the semi-$G$-normal distribution can be treated as
the attractor for normalized summations of semi-$G$-i.i.d. random variables
and the $G$-normal is the attractor for summations of $G$-version i.i.d.
random variables. 
\end{rem}

In the proof of \ref{thm:semi-G-CLT}, we adapt the idea of Lindeberg method in a ``leave-one-out'' manner (\cite{breiman1992}) to
the sublinear context. One of the reason that we are able to do such adaptation is the symmetry in semi-$G$-independence: $X_i$ is semi-$G$-independent from $\{X_j,j\neq i\}$ (Note that we cannot do such adaptation under sequential independence due to its asymmetry). 
More details of the proof can be found in \ref{ap:pf-semi-G-clt}.

Since we only have finite second moment assumption so far in \ref{thm:semi-G-CLT}, by adding stronger moment conditions on $X_n$, the function space of $\varphi$ can be taken to be $\fspace$ to include those unbounded ones. This statement is based on \ref{prop:extend-function-space}. 

As a basic example, given a stronger condition $\expt[\abs{X_1}^3]<\infty$, we can check that the convergence \ref{eq:clt-expt-converge} holds for $\varphi(x)=x^3$ by direct computation (\ref{eg:converge-varphi-x3}). 

\begin{eg}[Check $\varphi(x)=x^{3}$]
\label{eg:converge-varphi-x3}
In the convergence \ref{eq:clt-expt-converge},
since $\expt[W^{3}]=0$, we only need to show: 
\[
\lim_{n\to\infty}\expt[(\frac{1}{\sqrt{n}}\sum_{i=1}^{n}X_{i})^{3}]=0.
\]
In fact, 
\begin{align*}
\expt[(\frac{1}{\sqrt{n}}\sum_{i=1}^{n}X_{i})^{3}] & =n^{-3/2}\expt[(\sum_{i=1}^{n}X_{i})^{3}]\\
 & =n^{-3/2}\expt[\sum_{i=1}^{n}X_{i}^{3}+\sum_{i\neq j\text{ or }j\neq k}X_{i}X_{j}X_{k}].
\end{align*}
(Note that, if $i=j$ and $j=k$, the second term becomes the first
one.) 
For the summand $X_{i}X_{j}X_{k}$ of the second term, without loss
of generality, we assume that $i\leq j\leq k$. Then we have three
cases: 
\begin{enumerate}
	\item $i<j=k$, 
    \item $i=j<k$, 
  \item $i<j<k$.
\end{enumerate}


In Case 1, since $X_{i}$ and $X_{j}$ are semi-$G$-independent,
we have: 
\begin{align*}
\expt[X_{i}X_{j}^{2}] & =\max_{(v_{i},v_{j})}\lexpt[v_{i}v_{j}^{2}\eta_{i}\eta_{j}^{2}]\\
 & =\max_{(v_{i},v_{j})}v_{i}v_{j}^{2}\lexpt[\eta_{i}]\lexpt[\eta_{j}^{2}]=0.
\end{align*}
(Note that $\expt[X_{i}X_{j}^{2}]=0$ does not hold under sequential independence $X_i \seqind X_j$ by \ref{eg:indep-ex-org}.) Meanwhile, we can obtain $-\expt[-X_{i}X_{j}^{2}]=0$ so $X_{i}X_{j}^{2}$ has certain mean zero.

We can similarly prove the result in Case 2, that is, $X_{i}^2X_{k}$ has certain mean zero. 
For Case 3, since $X_{i},X_{j},X_{k}$ are semi-$G$-independent,
we have 
\[
\expt[X_{i}X_{j}X_{k}]=\max_{(v_{i},v_{j},v_{k})}v_{i}v_{j}v_{k}\lexpt[\eta_{i}]\lexpt[\eta_{j}]\lexpt[\eta_{k}]=0.
\] We further have $-\expt[-X_{i}X_{j}X_{k}]=0$ using the same logic. 
Therefore, 
\[
\expt[(\frac{1}{\sqrt{n}}\sum_{i=1}^{n}X_{i})^{3}]=n^{-3/2}\expt[\sum_{i=1}^{n}X_{i}^{3}]=n^{-1/2}\expt[X_{1}^{3}]\to 0,
\]
where we use the condition that $\expt[\abs{X_{1}}^{3}]<\infty$ and \ref{prop:sum-semiseqind}.
\end{eg}

\subsection{Fine structures of independence and the associated family of state-space volatility models}
\label{subsec:fine-structure}

In \ref{eg:direct-indep-structure}, we have mainly
discussed the independence between two semi-$G$-distributed objects.
Here we consider three of them as a starting point to discuss much
finer structure of independence. 

Consider $W_{i}=V_{i}\stdrv_{i}\eqdistn\semiGN(0,\varI),i=1,2,3$.
The independence structure among them is essentially related to the
$G$-version independence among $V_{i}$ and $\stdrv_{i}$, $i=1,2,3$.
For instance, 
\begin{itemize}
\item [(a)] $V_{1}\seqind V_{2}\seqind V_{3}\seqind\stdrv_{1}\seqind\stdrv_{2}\seqind\stdrv_{3}$,
\item [(b)] $V_{1}\seqind\stdrv_{1}\seqind V_{2}\seqind\stdrv_{2}\seqind V_{3}\seqind\stdrv_{3}$. 
\end{itemize}
Note that a) is equivalent to $W_{1}\semiseqind W_{2}\semiseqind W_{3}$
and b) means $W_{1}\fullseqind W_{2}\fullseqind W_{3}$ which implies
$W_{1}\seqind W_{2}\seqind W_{3}$. 

Then we can see that there are several middle stages between (a) and (b). 
%
In order to present these intermediate stages, let us play a simple game: switch
two components each time and change the independence structure from (a)
to (b). During this game, the following rules are required: 
\begin{itemize}
\item[R1] we must keep the independence $V_{i}\seqind\stdrv_{i}$
due to the definition of semi-$G$-normal, 
\item[R2] we must keep the order as $(V_{1},V_{2},V_{3})$ and $(\stdrv_{1},\stdrv_{2},\stdrv_{3}$),
because the independence order of elements within each vector is usually equivalent. Otherwise, if we break this order, we need an unnecessary extra step to retrieve the index order $(1,2,3)$ to be consistent with (b). 
\end{itemize}
Here we can get two approaches: 
\begin{enumerate}
\item Since we do not want to break the order within $(V_{1},V_{2},V_{3})$
or $(\stdrv_{1},\stdrv_{2},\stdrv_{3}$), the first step has to be
switching some $V_{i}$ with $\stdrv_{j}$ with $i,j=1,2,3$. For the $\stdrv$ part,
we can only move $\stdrv_{1}$ due to R1, and similarly for $V$ part
we can only move $V_{3}$. Hence, the first step is to exchange $V_{3}$
and $\stdrv_{1}$ in (a) to get 
\begin{equation}
V_{1}\seqind V_{2}\seqind\stdrv_{1}\seqind V_{3}\seqind\stdrv_{2}\seqind\stdrv_{3}.\label{eq:indep-mid-form-2}
\end{equation}
Then we have two equivalent ways to move on. 
\item One way is to exchange $V_{2}$ and $\stdrv_{1}$ in \ref{eq:indep-mid-form-2}
to get 
\begin{equation}
V_{1}\seqind\stdrv_{1}\seqind V_{2}\seqind V_{3}\seqind\stdrv_{2}\seqind\stdrv_{3}.\label{eq:indep-mid-form-3}
\end{equation}
Then we can exchange $V_{3}$ and $\stdrv_{2}$ to get (b).
\item Another way is to exchange $V_{3}$ and $\stdrv_{2}$ to get 
\begin{equation}
V_{1}\seqind V_{2}\seqind\stdrv_{1}\seqind\stdrv_{2}\seqind V_{3}\seqind\stdrv_{3}.\label{eq:indep-mid-form-3-2}
\end{equation}
Then we can exchange $V_{2}$ and $\stdrv_{1}$ to get (b). 
\end{enumerate}

Note that \ref{eq:indep-mid-form-3} implies the following relation:
\[
W_{1}\seqind(W_{2},W_{3})\text{ and }W_{2}\semiseqind W_{3}.
\]
We can show that the family of models associated with the representation of $\expt[\varphi(W_1,W_2,W_3)]$ under \ref{eq:indep-mid-form-3} can be illustrated by \ref{fig:diagram-L-subset-3}.
Similarly, \ref{eq:indep-mid-form-3-2} implies 
\[
(W_{1},W_{2})\seqind W_{3}\text{ and }W_{1}\semiseqind W_{2}.
\]
The family of models associated with \ref{eq:indep-mid-form-3-2} can be described by \ref{fig:diagram-L-subset-3-2}. The family of models for \ref{eq:indep-mid-form-2} can be shown by \ref{fig:diagram-L-subset-2}.

The intuition here is: if all $V_{j}$ is before $\stdrv_{i}$, since
$\stdrv_{i}$ has distributional certainty, in the directed graph,
$\sigma_{j}$ does not have effects on $\stdrv_{i}$. As long as we
have $\stdrv_{i}$ is \emph{before} $V_{j}$ in the order of the $G$-version independence, we must have the additional edge from
$\stdrv_{i}$ to $\sigma_{j}$ in the directed graph of the
family of models to represent the sublinear expectation of the joint vector. 

We can see that by changing the independence structure, the sublinear expectation of a joint vector of semi-$G$-version of distributions can be represented by classes of state-space models with different graphical structures. 

One question to be explored is whether there is an independence structure 
that is associated with the familiy shown in \ref{fig:diagram-L-subset-markov}. Our conjecture is as follows: at least we need the following conditions, 
\begin{itemize}
	\item [{1)}] $V_{1}\seqind\stdrv_{1}\seqind V_{2}\seqind\stdrv_{2}$ which
means $W_{1}\seqind W_{2}$, 
\item [{2)}] $V_{2}\seqind\stdrv_{2}\seqind V_{3}\seqind\stdrv_{3}$ which
means $W_{2}\seqind W_{3}$, 
\item [{3)}] $V_{1}\seqind V_{3}\seqind\stdrv_{1}\seqind\stdrv_{3}$ which
means $W_{1}\semiseqind W_{3}$. 
\end{itemize}

\begin{figure}[h]
	\centering
\begin{tikzpicture}[
            > = stealth, 
            shorten > = 1pt, 
            auto,
            node distance = 2cm, 
            semithick 
        ]

        \tikzstyle{every state}=[
            draw = black,
            thick,
            fill = white,
            minimum size = 4mm
        ]
        
        \node[state](s1){$\sigma_1$};
        \node[state](o1)[below of =s1]{$Y_1$};
        \path[->] (s1) edge node {$\cdot \epsilon_1$} (o1);
        \node[state](s2)[right of =s1]{$\sigma_2$};
        \node[state](o2)[below of =s2]{$Y_2$};
        \path[->] (s2) edge node {$\cdot \epsilon_2$} (o2);
        \path[->] (s1) edge (s2);

        \node[state](s3)[right of =s2]{$\sigma_3$};
        \node[state](o3)[below of =s3]{$Y_3$};
        \path[->] (s3) edge node {$\cdot \epsilon_3$} (o3);
        \path[->] (s2) edge (s3);
        \path[dashed,->](o1) edge (s3);

   \end{tikzpicture}
   \caption{Diagram for \ref{eq:indep-mid-form-2} }
   \label{fig:diagram-L-subset-2}
   \end{figure}
   
   \begin{figure}[h]
	\centering
\begin{tikzpicture}[
            > = stealth, 
            shorten > = 1pt, 
            auto,
            node distance = 2cm, 
            semithick 
        ]

        \tikzstyle{every state}=[
            draw = black,
            thick,
            fill = white,
            minimum size = 4mm
        ]
        
        \node[state](s1){$\sigma_1$};
        \node[state](o1)[below of =s1]{$Y_1$};
        \path[->] (s1) edge node {$\cdot \epsilon_1$} (o1);
        \node[state](s2)[right of =s1]{$\sigma_2$};
        \node[state](o2)[below of =s2]{$Y_2$};
        \path[->] (s2) edge node {$\cdot \epsilon_2$} (o2);
        \path[->] (s1) edge (s2);

        \node[state](s3)[right of =s2]{$\sigma_3$};
        \node[state](o3)[below of =s3]{$Y_3$};
        \path[->] (s3) edge node {$\cdot \epsilon_3$} (o3);
        \path[->] (s2) edge (s3);
        \path[dashed,->](o1) edge (s2);
        \path[dashed,->](o1) edge (s3);

   \end{tikzpicture}
   \caption{Diagram for \ref{eq:indep-mid-form-3} }
   \label{fig:diagram-L-subset-3}
   \end{figure}
   
   \begin{figure}[h]
	\centering
\begin{tikzpicture}[
            > = stealth, 
            shorten > = 1pt, 
            auto,
            node distance = 2cm, 
            semithick 
        ]

        \tikzstyle{every state}=[
            draw = black,
            thick,
            fill = white,
            minimum size = 4mm
        ]
        
        \node[state](s1){$\sigma_1$};
        \node[state](o1)[below of =s1]{$Y_1$};
        \path[->] (s1) edge node {$\cdot \epsilon_1$} (o1);
        \node[state](s2)[right of =s1]{$\sigma_2$};
        \node[state](o2)[below of =s2]{$Y_2$};
        \path[->] (s2) edge node {$\cdot \epsilon_2$} (o2);
        \path[->] (s1) edge (s2);

        \node[state](s3)[right of =s2]{$\sigma_3$};
        \node[state](o3)[below of =s3]{$Y_3$};
        \path[->] (s3) edge node {$\cdot \epsilon_3$} (o3);
        \path[->] (s2) edge (s3);
        \path[dashed,->](o2) edge (s3);
        \path[dashed,->](o1) edge (s3);

   \end{tikzpicture}
   \caption{Diagram for \ref{eq:indep-mid-form-3-2} }
   \label{fig:diagram-L-subset-3-2}
   \end{figure}

    \begin{figure}[h]
	\centering
\begin{tikzpicture}[
            > = stealth, 
            shorten > = 1pt, 
            auto,
            node distance = 2cm, 
            semithick 
        ]

        \tikzstyle{every state}=[
            draw = black,
            thick,
            fill = white,
            minimum size = 4mm
        ]
        
        \node[state](s1){$\sigma_1$};
        \node[state](o1)[below of =s1]{$Y_1$};
        \path[->] (s1) edge node {$\cdot \epsilon_1$} (o1);
        \node[state](s2)[right of =s1]{$\sigma_2$};
        \node[state](o2)[below of =s2]{$Y_2$};
        \path[->] (s2) edge node {$\cdot \epsilon_2$} (o2);
        \path[->] (s1) edge (s2);

        \node[state](s3)[right of =s2]{$\sigma_3$};
        \node[state](o3)[below of =s3]{$Y_3$};
        \path[->] (s3) edge node {$\cdot \epsilon_3$} (o3);
        \path[->] (s2) edge (s3);
        \path[dashed,->](o1) edge (s2);
        \path[dashed,->](o2) edge (s3);

   \end{tikzpicture}
   \caption{Diagram for the common structure of classical first-order hidden Markov models with feedback}
   \label{fig:diagram-L-subset-markov}
   \end{figure}
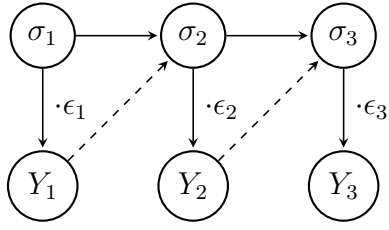

\subsection{A robust confidence interval for regression under heteorskedastic noise with unknown variance structure}
\label{subsec:robust-CI}


Let $\{W_i\}_{i=1}^\infty$ denote a sequence of nonlinearly i.i.d. semi-$G$-normally distributed random variables with $W_1 \sim \semiGN(0,\varI)$.
In \ref{subsec:connect-expt-CN-GN}. 
we have studied the $G$-EM procedure which is aimed at the following expression: 
\begin{equation}
\label{eq:iterate-weight-sum-2}
	\expt[\varphi(\sum_{i=1}^n a_i W_i)].
\end{equation}
This section will provide a basic example in the context of regression to show why we need to think about \ref{eq:iterate-weight-sum-2} in statistical practice. 

Consider a simple linear regression problem in the context of sequential data $(x_i,Y_i)$(where the order of the data \emph{matters}):
\begin{equation}
\label{eq:SLR-cl}
	Y_{i}=\beta_{0}+\beta_{1}x_{i}+\xi_i,i=1,2,\dotsc,n,
\end{equation}
where $x_i$ is treated as known and $\xi_i=\sigma_i\stdrv_i$ with $\sigma_i:\Omega\to \sdInt$ and $\stdrv_i\sim \CN(0,1)$ for each $i=1,2,\dotsc,n$. We can see that the noise $\xi_i$ part is heteorskedastic (although $\sigma_i$ is not observable). 
However, if the variance structure of the noise part $\xi_i$ is complicated due to measurement errors or the data is collected from different subpopulations with different variances, we need to have some precaution on the properties of the least square estimator $\hat{\beta}_1$, especially when we are lack of prior knowledge on the dynamic of $\sigma_i$.
If we worry that $\sigma_i$ may depend on the preivous $\epsilon_k$ with $k<i$, rather than assuming a single probabilistic model for $\sigma_i$ then perform the regression, in an early stage of data analysis,
 we can first assume $\myvec{\sigma}=(\sigma_i)_{i=1}^n$ could belong to any elements in $\myset{L}_n\sdInt$ defined in \ref{subsec:represent-semignorm-indep}. 
 Note that the distributional uncertainty of each $\xi_i$ can be described by $W_i\sim \semiGN(0,\varI)$. Then the distributional uncertainty of \ref{eq:SLR-cl} can be translated into a $G$-version format: 
 \begin{equation}
 \label{eq:SLR-R}
 	Y^G_{i}=\beta_{0}+\beta_{1}x_{i}+W_{i},i=1,2,\dotsc,n.
 \end{equation}
 
%
Let 
\[
a_{i}\coloneqq\frac{x_{i}-\bar{x}_{n}}{\sum(x_{i}-\bar{x}_{n})^{2}}.
\]
Then the least-square estimator can be written as 
\begin{equation}
\hat{\beta}_{1}=\frac{\sum(x_{i}-\bar{x}_{n})Y^G_{i}}{\sum(x_{i}-\bar{x}_{n})^{2}}=\sum a_{i}Y^G_{i}=\beta_{0}\sum a_{i}+\sum a_{i}x_{i}\beta_{1}+\sum a_{i}W_{i}=\beta_{1}+\sum a_{i}W_{i}.\label{eq:LS-formula-1}
\end{equation}
Then we have 
$
\hat{\beta}_{1}-\beta_{1}=\sum a_{i}W_{i}.
$
Note that $\expt[\hat{\beta}_1]=-\expt[-\hat{\beta}_1]=\beta_1$. 

Then we are able to study the properties of $\hat{\beta}_{1}$ by assigning different forms of $\varphi$ in \ref{eq:iterate-weight-sum-2}: 
\begin{enumerate}
\item With $\varphi(x)=x^{k}$ and $k\in \numset{N}_+$, we have the centred moments of $\hat{\beta}_{1}$
\[
\expt[\varphi(\sum a_{i}W_{i})]=\expt[(\hat{\beta}_{1}-\beta_{1})^{k}].
\]
\item With $\varphi(x)=\ind{\abs{x}>c}$, we get
the object that is useful to derive a confidence interval in this
context: 
\begin{equation}
\label{eq:CI-represent}
\expt[\varphi(\sum a_{i}W_{i})]=\upprob(\abs{\hat{\beta}_{1}-\beta_{1}}>c).
\end{equation} 
\end{enumerate}

Interestingly, from \ref{thm:represent-n-seqind-semignorm} and \ref{thm:represent-capacity}, \ref{eq:CI-represent} further leads us to a \emph{robust} confidence
interval by solving the following equation:
\[
\upprob(\abs{\hat{\beta}_{1}-\beta_{1}}>c_{\alpha/2})=\sup_{\myvec{\sigma}\in\myset{L}_{n}\sdInt}\lprob(\abs{\sum a_{i}\sigma_{i}\stdrv_{i}}>c_{\alpha/2}) = \alpha,
\]
or 
\[
\inf_{\myvec{\sigma}\in\myset{L}_{n}\sdInt}\lprob(\abs{\sum a_{i}\sigma_{i}\stdrv_{i}}\leq c_{\alpha/2}) = 1-\alpha.
\]
The resulting confidence interval is robust in the sense that its coverage rate will be at least $1-\alpha$ regardless of
the unknown variance structure of the noise part $\sigma_{i}\stdrv_{i}$ in the regression. 
If we have more information that shows $\sigma_k$ does not depend on the previous $\stdrv_i$ with $i<k$, we can consider a smaller family of sets $\myset{S}_n\sdInt$. Alternatively, it also provides a way to perform a \emph{sensitivity analysis} on the performance of a regression estimator (such as $\hat{\beta}_1$ here) under heteroscedastic noise with unknown variance structure that could belong to different family of models. 


Then this discussion leads us to another interesting question. 
In an early stage of data analysis, should we choose $\myvec{\sigma}\in \myset{S}_n\sdInt$ or $\myvec{\sigma}\in \myset{L}_n\sdInt$? This question will be explored in \ref{subsec:inference-ssvm}. 



\subsection{Inference on the general model structure of a state-space volatility model}
\label{subsec:inference-ssvm}

Recall the setup in \ref{subsec:robust-CI}.
In practice, if lacking knowledge on the underlying dynamic of the datasets, whether we should choose $\myvec{\sigma}\in \myset{S}_n\sdInt$ or $\myvec{\sigma}\in \myset{L}_n\sdInt$ is a difficult problem in classical statistical methodology (in model specification) because both families involve a infinitely-dimensional family of elements. However, it turns out it can be essentially transformed into a $G$-version question: it has a feasible solution once we introduce the $G$-expectation of semi-$G$-family of distributions. This becomes a hypothesis test to distinguish between semi-sequential independence and sequential independence. 
To be specific, we are able to consider a test:
\[
H_0:\myvec{\sigma} \in \myset{S}_n\sdInt\textbf{ vs } H_a:\myvec{\sigma} \in \myset{L}_n\sdInt\setminus \myset{S}_n\sdInt.
\]
A good interpretation of this test is, since the class of hidden Markov models (with volatility as the switching regimes) belong to $\myset{S}$. If we reject the null hypothesis, it means the underlying $\myvec{\sigma}$ process cannot be treated as a switching-regime in the hidden Markov setup (or in any other kinds of normal mixture model), but we need to re-investigate the dataset and consider the $\myvec{\sigma}$ process outside of the family of the normal mixture model (for instance, we may need to introduce other dependency, like the one between the previous observation $Y_{<t}$ with current $\sigma_t$, such as a feedback design). 
Throughout this discussion, we did not make any parametric assumption on the model of $\myvec{\sigma}$, and we are still able to give a rigorous test on this distinction. 
The idea of this test will take advantage of \ref{thm:represent-n-seqind-semignorm} to transform the distinction between two family of classical models to a task of distinguishing two different types of independence ($\semiseqind$ versus $\seqind$) for semi-$G$-normal vector $(W_1,W_2,\dotsc,W_n)$. There are plenty of test functions $\varphi$ (neither convex nor concave ones) to reveal the difference between $\semiseqind$ and $\seqind$ such that 
\[
\expt^{S}[\varphi(W_{1},W_{2},\dotsc,W_{n})]<\expt^{L}[\varphi(W_{1},W_{2},\dotsc,W_{n})].
\]
For instance, we can choose 
\[
\varphi(x_{i},i=1,2,\dotsc,n)=(\frac{1}{\sqrt{n}}\sum_{i=1}^{n}x_{i})^{3}.
\]
Under this $\varphi$, the expectation under $\semiseqind$ is a certain zero but the one under $\seqind$ is greater than zero. Then we should be able to construct a test statistic based on the form of $\varphi$ and obtain a rejection region by studying its tail probability under $\upprob$ which can be transformed back into the sublinear expectation of $(W_{1},W_{2},\dotsc,W_{n})$. 

How to choose the test function will have significant effect on the performance of this hypothesis test. 
Moreover, the current interpretation of $n$ is the length of the whole data sequence and $\myvec{\sigma}$ is the unknown volatility dynamic of the full sequence. We can also interpret $n$ as the group size after grouping the dataset in a either non-overlapping or overlapping manner, then we can consider $\myvec{\sigma}$ for each group to test whether there is a case falling into the class of $H_a$, because the sublinear expectation can give a control on the extremes of the group statistics as indicated by \cite{jin2021optimal} and Section 2.2 in \cite{fang2019limit}.

\section*{Acknowledgements and the story behind the semi-$G$-normal}

We have received many useful suggestions and feedback from the community in the past four years which are beneficial to the formation of this paper (so this paper can be treated as a report to the community). 

The authors would like to first express their sincere thanks to Prof.~Shige Peng who visited our department in May 2017 (invited by Prof.~Reg Kulperger) and our discussion at that time motivated us to study a distributional and probabilistic structure that has a direct connection with the intuition behind the existing max-mean estimation proposed by \cite{jin2021optimal}.
Later on during the Fields-China Industrial Problem Solving Workshop in Finance, and a short course on the $G$-expectation framework given by Prof.~Peng at Fields Institute, Toronto, Canada, we had several interesting discussions on the data experiments in this context, which can be treated as the starting point of the companion paper of the current one. In our regular discussion notes in that period, there was a prototype of the current semi-$G$-normal distribution and also a question on independence between semi-$G$-normal was raised which is currently included and answered by \ref{eg:direct-indep-structure}. 



Although the design of semi-$G$-normal is mainly for distributional purpose, this concept was first proposed in \cite{li2018iterative}, which was applied to design an iterative approximation towards the $G$-normal distribution by starting from the linear expectations of classical normal, as discussed in \ref{subsec:connect-expt-CN-GN}.  
During the 2018 Young Researcher Meeting on BSDEs, Nonlinear Expectations and Mathematical Finance in Shanghai Jiao Tong University, 
we have received beneficial feedback on this iterative method from participants in the conference. Specially we would like to thank to Prof.~Yiqing Lin on providing more references that have potential theoretical connections. \ref{rem:improve-iterate} is benefited from the comments by Prof.~Shuzhen Yang and Prof.~Xinpeng Li.

The first author would also like to express his gratitude to Prof.~Huaxiong Huang at Fields Institute and his Ph.D.~student Nathan Gold for their support and suggestions during a separate long-term and challenging joint project (regularly discussing with Prof.~Peng) during summer 2017 on a stochastic method of the $G$-heat equation in high-dimensional case and its theoretical and numerical convergence. In this project, the first author has learned the intuition about the methods in how to appropriately design a switching rule in the stochastic volatility to approximate the solution to a fully nonlinear PDE, which is related to the methods based on BSDEs and the second order BSDEs, and also the intuition behind nonlinear expectations in this context. Although the methods are different, this experience creates another motivation for \cite{li2018iterative}.  




The authors are grateful for the valuable discussions with the community during the conference of Probability, Uncertainty and Quantitative Risk in July 2019. 
One of the motivations of \ref{eg:direct-indep-structure} is from the comments by Prof.~Mingshang Hu on the independence property of maximal distribution. The writing of \ref{subsec:data-seq-G-normal} is motivated by the discussions with Prof.~Peng during the conference. 
\ref{subsec:data-seq-G-normal} and further the data experiments in the  companion paper are also benefited from the discussions on the meaning of sequential independence under a set of measures with Prof.~Jianfeng Zhang. 



We have also benefited from the feedback from participants coming from various backgrounds in the Annual Meetings of SSC (Statistical Society of Canada) in 2018 and 2019 to understand the impression from general audience on the $G$-expectation framework. 
During the poster session of the Annual Meeting of SSC at McGill University in 2018, we have received several positive comments about designing a substructure connecting the $G$-expectation framework (which is a highly technical one for general audience) with the objects in the classical system. These comments further motivate us to write this paper for general readers. 
In the Annual Meeting of SSC at Calgary University in 2019, there is a comment from the audience on the property of $\myset{H}$ and the choice of function space ($\myset{H}$ could be quite small if we choose a large function space for $\varphi$). It has motivated us to improve the preliminary setup (\ref{sec:G-frame-setup}) and put more attention on the design of $\myset{H}$.

During the improvement of this manuscript from the first version (April 2021) to the third version (October 2021), 
the authors are grateful to Prof.~Defei Zhang who gives many beneficial comments (such as the comment on the product space and an improvement of \ref{fig:connection-threenorm}) and Prof.~Xinpeng Li whose suggestion motivates us to develop the research in \ref{subsec:semi-G-clt}. 

\newpage

\section{Proofs}
\label{sec:proofs}
%

%
%

\subsection{Proofs in \ref{subsec:property-maximal}}
\label{pf:subsec-property-maximal}

\begin{proof}[Proof of \ref{prop:rep-uni-maximal}]
The finiteness of $\lexpt[\abs{\varphi(V)}]$ is obvious due to the continuity of $\varphi$ and the compactness of $\sdInt$.
First of all, note that \ref{eq:represent-maximal-1} 
is a direct result of \ref{def:maximal}. It is also not hard to see \ref{eq:represent-maximal-3}, since for any $\sigma\in\mysetrv{D}\sdInt$, it satisfies
$\lprob_{\sigma}(\sdInt)=1$, then
\begin{align*}
\lexpt[\varphi(\sigma)] & =\int_{\sdl}^{\sdr}\varphi(x)\lprob_{\sigma}(\diff x)\leq\int_{\sdl}^{\sdr}\max_{x\in\sdInt}\varphi(x)\lprob_{\sigma}(\diff x)\\
 & =\max_{x\in\sdInt}\varphi(x)\lprob_{\sigma}(\sdInt)=\max_{x\in\sdInt}\varphi(x),
\end{align*}
which implies
\[
\max_{\sigma \in \mysetrv{D}\sdInt}\lexpt[\varphi(\sigma)]\leq\max_{\sigma\in \sdInt}\lexpt[\varphi(\sigma)].
\]
Since $\sdInt \subset\mysetrv{D}\sdInt$, we also have the other direction
of inequality holds. Similarly, we can show \ref{eq:represent-maximal-disc}.


To validate \ref{eq:represent-maximal-2}, 
we need to show that for any $\alpha>0$, there exist a random variable $\sigma_{\alpha}\in\mysetrvcont\sdInt$,
such that 
\begin{equation}
\label{eq:ineq-proof-3}
\lexpt[\varphi(\sigma_{\alpha})]>\expt[\varphi(V)]-\alpha.
\end{equation}
Let $v^{*}=\argmax_{v\in\sdInt}\varphi(v)$. Then we have $\expt[\varphi(V)]=\varphi(v_{0})$.
Since $\varphi$ is a continuous function on $\sdInt$, there exists
$v_{0}\in(\sdl,\sdr)$ such that $\varphi(v_{0})>\varphi(v^{*})-\alpha/2$.
In a classical probability space $\lexptSpace$, 
consider a series of random variables 
$\xi_{n}\coloneqq v_{0}+e/\sqrt{n}$
where $e\sim\CN(0,1)$ and $n\in\numset{N}_{+}$. 
In short, $\xi_{n}\sim\CN(v_{0},1/n)$ with diminishing variance. 
Then we must have $\xi_{n}\convergeto{\dist}v_{0}$.
Then transform $\xi_{n}$ into
its truncation on $\sdInt$: $\xi_{n}^{*}\coloneqq \xi_{n}I_{n}$ with $I_{n}\coloneqq\ind{\xi_{n}\in\sdInt}$. 
We can easily
show that $I_{n}\convergeto{\lprob}1$ since, for any $a>0$,
$
\lprob(\abs{I_{n}-1} >a)=\lprob(I_{n}=0)=1-\lprob(\xi_{n}\in\sdInt)\to0.
$
By classical Slustky's theorem,
$\xi_{n}^{*}=\xi_{n}I_{n}\convergeto{\dist}v_{0}.$ Therefore, for any
$\varphi\in\fspace(\numset{R})$,
\[
\lexpt[\varphi(\xi_{n}^{*})]\to\varphi(v_{0}).
\]
For any $\alpha>0$, there exists $n_{\alpha}$ such that 
$
\lexpt[\varphi(\xi_{n_{\alpha}}^{*})]>\varphi(v_{0})-\alpha/2.
$
Let $\sigma_\alpha\coloneqq \xi_{n_{\alpha}}$ which belongs to $\mysetrvcont\sdInt$. It the required object satisfying \ref{eq:ineq-proof-3}, because
\[
\lexpt[\varphi(\sigma_\alpha))]>\varphi(v_{0})-\alpha/2>\varphi(v_{0})-\alpha=\expt[\varphi(V)]-\alpha.\qedhere
\]
\end{proof}

\begin{proof}[Proof of \ref{thm:maximal-multi-relation}]
\label{pf:thm:maximal-multi-relation}
We can prove it by mathematical induction. For $d=1$, it obviously
holds. Suppose the results holds for $d=k$ with $k\in\numset{N}_{+}$,
namely, 
\[
\myvec{\gsd}_{(k)}\coloneqq(\gsd_{1},\gsd_{2},\dotsc,\gsd_{k})\sim\Maximal(\prod_{i=1}^{k}[\sdl_{i},\sdr_{i}]),
\]
then we only need to show it holds for $d=k+1$. In fact, consider any
locally Lipschitz function 
\[
\varphi:(\numset{R}^{k+1},\norm{\cdot})\to(\numset{R},\abs{\cdot}),
\]
satisfying, there exists $C_{\varphi}>0$, $m\in\numset{R}_{+}$,
\[
\abs{\varphi(x)-\varphi(y)}\leq C_{\varphi}(1+\norm{x}^{m}+\norm{y}^{m})\norm{x-y}.
\]
Since $\gsd_{k+1}$ is independent from $\myvec{\gsd}_{(k)}$,
\[
\expt[\varphi(\gsd_{1},\gsd_{2},\dotsc,\gsd_{k})] =\expt[\varphi(\myvec{\gsd}_{(k)},\gsd_{k+1})] = \expt\left[\expt[\varphi(\myvec{\sigma}_{(k)},\gsd_{k+1})]_{\myvec{\sigma}_{(k)}=\myvec{\gsd}_{(k)}}\right].
\]
Let 
$
\varphi_{k}(x)\coloneqq\varphi(\myvec{\sigma}_{(k)},x),
$
and 
\[
\psi_{k+1}(\myvec{\sigma}_{(k)})\coloneqq\max_{\sigma_{k+1}\in[\sdl_{k+1},\sdr_{k+1}]}\varphi(\myvec{\sigma}_{(k)},\sigma_{k+1})=\max_{\sigma_{k+1}\in [\sdl_{k+1},\sdr_{k+1}]}\varphi_{k}(\sigma_{k+1}).
\]
For notational convenience, we sometimes omit the domain $[\sdl_{k+1},\sdr_{k+1}]$ of the maximization here in our later discussions if it is clear from the context.
\begin{claim}
	\label{claim:maximal-varphi-k}
	We have $\varphi_{k}\in\fspace(\numset{R})$ and $\psi_{k+1}\in\fspace(\numset{R}^{k})$.
\end{claim}
Then we are able to apply the representation of maximal distribution $\myvec{V}_{(k)}$ (allowed by \ref{claim:maximal-varphi-k}) to have
\begin{align*}
\expt[\varphi(\gsd_{1},\gsd_{2},\dotsc,\gsd_{k+1})] & =\expt[\varphi(\myvec{\gsd}_{(k)},\gsd_{k+1})]\\
 & =\expt\Bigl[\expt[\underbrace{\varphi(\myvec{\sigma}_{(k)},\gsd_{k+1})}_{\varphi_{k}(\sigma_{k+1})}]_{\myvec{\sigma}_{(k)}=\myvec{\gsd}_{(k)}}\Bigr]\\
 & =\expt\Bigl[[\max_{\sigma_{k+1}}\varphi_{k}(\sigma_{k+1})]_{\myvec{\sigma}_{(k)}=\myvec{\gsd}_{(k)}}\Bigr]\\
 & =\expt[\psi_{k+1}(\myvec{\gsd}_{(k)})]\\
 & =\max_{\myvec{\sigma}_{(k)}}\psi_{k+1}(\myvec{\sigma}_{(k)})\\
 & =\max_{(\sigma_1,\sigma_2,\dotsc,\sigma_k)}\max_{\sigma_{k+1}}\varphi(\sigma_{1},\dotsc,\sigma_{k},\sigma_{k+1})\\
 & =\max_{(\sigma_1,\sigma_2,\dotsc,\sigma_{k+1})}\varphi(\sigma_{1},\dotsc,\sigma_{k},\sigma_{k+1}).
\end{align*}
Therefore, 
\[
(\gsd_{1},\gsd_{2},\dotsc,\gsd_{k+1})\sim\Maximal(\prod_{i=1}^{k+1}[\sdl_{i},\sdr_{i}]).
\]
The conclusion can be achieved by induction.

The remaining task is to prove \ref{claim:maximal-varphi-k}. 


To show $\varphi_{k}\in\fspace(\numset{R})$, we write
\begin{align*}
\abs{\varphi_{k}(x)-\varphi_{k}(y)} & =\abs{\varphi(\myvec{\sigma}_{(k)},x)-\varphi(\myvec{\sigma}_{(k)},y)}\\
 & \leq C_{\varphi}(1+\norm{(\myvec{\sigma}_{(k)},x)}^{m}+\norm{(\myvec{\sigma}_{(k)},y)}^{m})\norm{x-y},
\end{align*}
where we adapt $\norm{\cdot}$ to lower dimension in the sense that
$\norm{x}\coloneqq\norm{(\myvec{0}_{(k)},x)}$. Notice
\[
\norm{(\myvec{\sigma}_{(k)},x)}  =\norm{(\myvec{\sigma}_{(k)},0)+(\myvec{0}_{(k)},x)} \leq\norm{\myvec{\sigma}_{(k)}}+\norm{x}.
\]
Meanwhile, there exists $K\geq0$ (actually $K=\max\{1,2^{m-1}\}$),
such that 
\[
\norm{(\myvec{\sigma}_{(k)},x)}^{m}  \leq(\norm{\myvec{\sigma}_{(k)}}+\norm{x})^{m} \leq K(\norm{\myvec{\sigma}_{(k)}}^{m}+\norm{x}^{m}).
\]
Then we have 
\[
\abs{\varphi_{k}(x)-\varphi_{k}(y)}\leq C_{1}(1+\norm{x}^{m}+\norm{y}^{m})\norm{x-y},
\]
where $C_{1}=C_{\varphi}\max\{1+2K\norm{\myvec{\sigma}_{(k)}}^{m},K\}.$

Next we check $\psi_{k+1}\in\fspace(\numset{R}^{k})$. For any $\myvec{a}_{(k)},\myvec{b}_{(k)}\in\numset{R}^{k}$, 
\begin{align*}
 & |\psi_{k+1}(\myvec{a}_{(k)})-\psi_{k+1}(\myvec{b}_{(k)})|\\
= & |\max_{\sigma_{k+1}\in\sdInt}\varphi(\myvec{a}_{(k)},\sigma_{k+1})-\max_{\sigma_{k+1}\in\sdInt}\varphi(\myvec{b}_{(k)},\sigma_{k+1})|\\
\leq & \max_{\sigma_{k+1}\in\sdInt}|\varphi(\myvec{a}_{(k)},\sigma_{k+1})-\varphi(\myvec{b}_{(k)},\sigma_{k+1})|\\
\leq & \max_{\sigma_{k+1}\in\sdInt}C_{\varphi}(1+\norm{(\myvec{a}_{(k)},\sigma_{k+1})}^{m}+\norm{(\myvec{b}_{(k)},\sigma_{k+1})}^{m})\norm{\myvec{a}_{(k)}-\myvec{b}_{(k)}}\\
\leq & C_{2}(1+\norm{\myvec{a}_{(k)}}+\norm{\myvec{b}_{(k)}})\norm{\myvec{a}_{(k)}-\myvec{b}_{(k)}},
\end{align*}
where $C_{2}=C_{\varphi}\max\{1+2K\sdr^{m},K\}$. 
\end{proof}

\begin{proof}[Proof of \ref{cor:Properties-of-Maximal}]
\label{pf:cor:Properties-of-Maximal}
The first statement can be proved by studying the range of $\psi(\myvec{\gsd})$.
First, we need to show that $\varphi\circ\psi(x)\coloneqq\varphi(\psi(x))$
is also a locally Lipschitz function for any $\varphi\in\fspace(\numset{R}^{d})$.
Suppose $\psi$ satisfies, 
\begin{equation}
\norm{\psi(\myvec{x})-\psi(\myvec{y})}\leq C_{\psi}(1+\norm{\myvec{x}}^{p}+\norm{\myvec{y}}^{p})\norm{\myvec{x}-\myvec{y}}.\label{eq:psi-property}
\end{equation}
We first can write
\begin{align}
|\varphi\circ\psi(\myvec{x})-\varphi\circ\psi(\myvec{y})| & =|\varphi(\psi(\myvec{x}))-\varphi(\psi(\myvec{y}))|\nonumber \\
 & \leq C_{\varphi}(1+\norm{\psi(\myvec{x})}^{m}+\norm{\psi(\myvec{y})}^{m})\norm{\psi(\myvec{x})-\psi(\myvec{y})}.\label{eq:prodfunction-local-lip}
\end{align}
As preparations for next step, we are going to frequently use tha
basic fact that the lower-degree polynomials can be dominated by higher-degree
ones in the sense that, 
\begin{equation}
\label{eq:ineq-poly-1}
	\norm{\myvec{x}}^{k}\leq\max\{1,\norm{\myvec{x}}^{l}\}\leq1+\norm{\myvec{x}}^{l}\text{ with }k\leq l,
\end{equation}
and for any $k,l\in\numset{N}_{+}$,
\begin{equation}
\label{eq:ineq-poly-2}
	\norm{\myvec{x}}^{k}\norm{\myvec{y}}^{l}\leq\frac{1}{2}(\norm{\myvec{x}}^{2k}+\norm{\myvec{x}}^{2l}).
\end{equation}
In \ref{eq:prodfunction-local-lip}, we can directly use \ref{eq:psi-property}
to dominate $\norm{\psi(\myvec{x})-\psi(\myvec{y})}$. For the parts
like $\norm{\psi(\myvec{x})}^{m}$, \ref{eq:psi-property} implies,
\[
\norm{\psi(\myvec{x})} \leq|\psi(\myvec{x})-\psi(\myvec{0})|+|\psi(\myvec{0})|
  \leq C_{\psi}(1+\norm{\myvec{x}}^{p})\norm{\myvec{x}}+C_{0}
  \leq C_{\psi}'(1+\norm{\myvec{x}}^{p+1}),
\]
then there exists $C_{\psi}''>0$ such that,
\[
\norm{\psi(\myvec{x})}^{m}  \leq[C_{\psi}'(1+\norm{\myvec{x}}^{p+1})]^{m}
 \leq C_{\psi}''(1+\norm{\myvec{x}}^{(p+1)m}).
\]
Hence, we can get $\varphi\circ\psi\in\fspace(\numset{R}^{d})$ by
the inequality as follows, 
\begin{align*}
|\varphi\circ\psi(\myvec{x})-\varphi\circ\psi(\myvec{y})| & \leq K_{1}(1+\norm{\myvec{x}}^{(p+1)m}+\norm{\myvec{y}}^{(p+1)m})(1+\norm{\myvec{x}}^{p}+\norm{\myvec{y}}^{p})\norm{\myvec{x}-\myvec{y}}\\
 & \leq K_{2}(1+\norm{\myvec{x}}^{2(p+1)pm}+\norm{\myvec{y}}^{2(p+1)pm})\norm{\myvec{x}-\myvec{y}}.
\end{align*}
Finally, we have $\psi(\myvec{\sigma})\sim\Maximal(\mysetrv{S})$ from its
representation:
\begin{align*}
\expt[\varphi(\myvec{S})]=\expt[\varphi(\psi(\myvec{\gsd}))] & =\expt[\varphi\circ\psi(\myvec{\gsd})]\\
 & =\max_{\sigma_{i}\in[\sdl_{i},\sdr_{i}],i=1,2,\dotsc,d}\varphi\circ\psi(\sigma_{1},\sigma_{2},\dotsc,\sigma_{d})\\
 & =\max_{\myvec{s}\in\mysetrv{S}}\varphi(\myvec{s}).
\end{align*}

The second statement essentially comes from the basic property of the maximum of a continuous function on a rectangle: in this ideal setup, the  order of taking marginal maximum does not affect the final value. To show the basic idea, start from a simple case $d=2$: 
if $\gsd_{1}\seqind \gsd_{2}$, for any $\varphi\in\fspace(\numset{R}^{2})$,
we can work on $(\gsd_{2},\gsd_{1})$ to show the other direction of independence, 
\begin{align*}
	\expt[\varphi(\gsd_{2},\gsd_{1})] & =\expt[\expt[\varphi(\sigma_{2},\gsd_{1})]_{\sigma_{1}=\gsd_{1}}]\\
	& = \max_{\sigma_{1}\in[\sdl_{1},\sdr_{1}]}\max_{\sigma_{2}\in[\sdl_{2},\sdr_{2}]}\varphi(\sigma_{2},\sigma_{1}) \\
	& =\max_{(\sigma_{1},\sigma_{2})\in\prod_{i=1}^{2}[\sdl_{i},\sdr_{i}]}\varphi(\sigma_{2},\sigma_{1})\\
	& = \max_{\sigma_{2}\in[\sdl_{2},\sdr_{2}]}\max_{\sigma_{1}\in[\sdl_{1},\sdr_{1}]}\varphi(\sigma_{2},\sigma_{1}) \\
	& = \expt[\expt[\varphi(\sigma_{2},\gsd_{1})]_{\sigma_{2}=\gsd_{2}}],
\end{align*}
where we have used the fact that $
\varphi_{x}(y)\coloneqq\max_{x\in\sdInt}\varphi(x,y)\in\fspace(\numset{R})
$ if $\varphi\in\fspace(\numset{R}^{2})$, which can be validated by  \ref{claim:maximal-varphi-k}.
Hence, we have $\gsd_{2}\seqind \gsd_{1}$. 
  
In general, for any permutation $(i_{1},i_{2},\dotsc,i_{d})$ of $(1,2,\dotsc,d)$, our objective is
to prove for any $j=2,\dotsc d$, 
\[
(\gsd_{i_{1}},\gsd_{i_{2}},\dotsc \gsd_{i_{j-1}})\seqind \gsd_{i_{j}}.
\]
From the first statement, $(\gsd_{i_{1}},\gsd_{i_{2}},\dotsc,\gsd_{i_{j}})$,
as a function of $(\gsd_{1},\gsd_{2},\dotsc,\gsd_{d})$, must also follow a
maximal distribution, characterized by $\Maximal(\myset{V}_{j})$
with 
\[
\myset{V}_{j}\coloneqq\prod_{k=1}^{j}[\sdl_{i_{k}},\sdr_{i_{k}}].
\]
Then we can mimic the derivation for $d=2$ to check the independence,
\begin{align*}
\expt[\varphi(\gsd_{i_{1}},\gsd_{i_{2}},\dotsc,\gsd_{i_{j}})] & =\max_{(\sigma_{i_{1}},\sigma_{i_{2}},\dots,\sigma_{i_{j}})\in\myset{V}_{j}}\varphi(\sigma_{i_{1}},\sigma_{i_{2}},\dots,\sigma_{i_{j}})\\
 & =\max_{(\sigma_{i_{1}},\sigma_{i_{2}},\dotsc,\sigma_{i_{j-1}})}\max_{\sigma_{i_{j}}}\varphi(\sigma_{i_{1}},\sigma_{i_{2}},\dots,\sigma_{i_{j}})\\
 & =\expt[[\max_{\sigma_{i_{j}}}\varphi(\sigma_{i_{1}},\sigma_{i_{2}},\dots,\sigma_{i_{j-1}},\gsd_{i_{j}})]_{\sigma_{i_{k}}=\gsd_{i_{k}},k=1,2\dotsc,j-1}]\\
 & =\expt[\expt[\varphi(\sigma_{i_{1}},\sigma_{i_{2}},\dotsc,\sigma_{i_{j-1}},\gsd_{i_{j}})]_{\sigma_{i_{k}}=\gsd_{i_{k}},k=1,2\dotsc,j-1}].
\end{align*}
Since it holds for all possible $j$, it is equivalent to say
\[
\gsd_{i_{1}}\seqind \gsd_{i_{2}}\seqind\cdots\seqind \gsd_{i_{d}}.\qedhere
\]

\end{proof}

\subsection{Proofs in \ref{subsec:semi-G-normal-intro} (improved)}
\label{pf:subsec:semi-G-normal-intro}

%


%


In order to show the uniqueness of decomposition (\ref{prop:unique-decompose}), we first prepare several lemmas. 
\begin{lem}
\label{lem:g-ky-classical}For any $g(K,\eta)\in\bar{\myset{H}}_{s}$
where $K\sim\Maximal(\Theta)$ and $\eta$ is classical, if $g(K,\eta)\eqdistn\stdrv$
where $\stdrv$ is classical, we must have, for any fixed $k\in\Theta$,
\[
g(k,\eta)\eqdistn\stdrv.
\]
\end{lem}

\begin{proof}
Since for any function $\psi$,
\[
\max_{k\in\Theta}\lexpt[\psi(g(k,\eta))]=\expt[\psi(g(K,\eta))]=\lexpt[\psi(\stdrv)],
\]
by replacing $\psi$ with $-\psi$, we have 
\begin{align*}
\min_{k\in\Theta}\lexpt[\psi(g(k,\eta))] & =-\expt[-\psi(g(K,\eta))]\\
 & =-\lexpt[-\psi(\stdrv)]=\lexpt[\psi(\stdrv)].
\end{align*}
It means for any $k\in\Theta$, we have 
\[
\lexpt[\psi(g(k,\eta))]\equiv\lexpt[\psi(\stdrv)].
\]
Therefore, we have $g(k,\eta)\eqdistn\stdrv.$
\end{proof}

\begin{lem}
\label{lem:bound-maximal}For a maximally distributed $V\sim\Maximal\sdInt$,
we have $\lowprob(V\in\sdInt)=1$. 
\end{lem}

\begin{proof}
Let 
\[
\varphi_{n}(x)=\begin{cases}
1 & x\in\sdInt\\
n(x-\sdl)+1 & x\in[\sdl-\frac{1}{n},\sdl)\\
-n(x-\sdr)+1 & x\in(\sdr,\sdr+\frac{1}{n}]\\
0 & \text{otherwise}
\end{cases}.
\]
Then we have $\varphi_{n}\in\fspace(\numset{R})$ and $\varphi_{n}(x)\downarrow\indicator_{\sdInt}(x)$
or $-\varphi_{n}(x)\uparrow-\indicator_{\sdInt}(x)$. (Since each
$\varphi_{n}(V)\in\myset{H}$, we have $\varphi(V)=\lim_{n\to\infty}\varphi_{n}(V)\in\myset{H}$
by the completeness of $\myset{H}$.) Note that 
\[
\expt[-\varphi_{n}(V)]=\max_{x\in\sdInt}(-\varphi_{n}(x))=-\min_{x\in\sdInt}\varphi_{n}(x)=-1.
\]
It implies that 
\[
\expt[-\indicator_{\sdInt}(V)]=\lim_{n\to\infty}\expt[-\varphi_{n}(V)]=-1,
\]
then 
\[
\lowprob(V\in\sdInt)=-\expt[-\indicator_{\sdInt}(V)]=1.
\]
\end{proof}

\begin{lem}
\label{prop:hxy-hx}Consider $K\sim\Maximal(\Theta)$ where $\Theta$
is a compact and convex set and $\eta$ follows a non-degenerate classical
distribution $P_{\eta}$ with $K\seqind\eta$. For any $h\in\fspace$,
if $h(K,\eta)\sim\Maximal\sdInt$, there exists $B\in\myset{B}(\numset{R})$
with $P_{\eta}(B)=1$ such that $h(x,y)$ does not depend on $y$
or simply $h(x,y)=h(x)$ when $y\in B$. 
\end{lem}

\begin{proof}
For any $\varphi\in\fspace$ with $\varphi(x)>0$ on $x\in\sdInt$,
let $\sigma^{*}\coloneqq\argmax_{\sigma\in\sdInt}\varphi(\sigma)$.
Then we have 
\begin{align*}
\varphi(\sigma^{*})=\max_{\sigma\in\sdInt}\varphi(\sigma) & =\expt[\varphi(h(K,\eta))]\\
 & =\max_{k\in\Theta}\lexpt[\varphi(h(k,\eta))]\\
 & =\max_{k\in\Theta}\int\varphi(h(k,y))P_{\eta}(\diff y)\\
 & \leq\int\max_{k\in\Theta}\varphi(h(k,y))P_{\eta}(\diff y).
\end{align*}
Meanwhile, note that $h(K,\eta)$ is bounded by $\sdInt$ and $K$
is bounded by $\Theta$ in quasi-surely sense or
\[
\lowprob(h(K,\eta)\in\sdInt)=1,\;\lowprob(K\in\Theta)=1.
\]
Then, for any $\lprob\in\myset{P}$, we have
\[
\lprob(\{\omega:h(K(\omega),\eta(\omega))\in\sdInt\})=1,\;\lprob(\{\omega:K(\omega)\in\sdInt\})=1,
\]
then the intersection of two events has probability $1$,
\[
\lprob(\{\omega:h(K,\eta)\in\sdInt,K\in\Theta\})=1.
\]
Hence, with $A\coloneqq\{y:h(k,y)\in\sdInt,k\in\Theta\},$ we must
have 
\[
P_{\eta}(A)=\lprob(\{\omega:\eta(\omega)\in A\})\geq\lprob(\{\omega:h(K,\eta)\in\sdInt,K\in\Theta\})=1.
\]
 (The measurability of $A$ comes from the continuity of $h$. Under
any $\lprob\in\myset{P}$, the distribution of $\eta$ is always $P_{\eta}$
due to \ref{rem:prob-eta}.) Then we have 
\begin{align*}
\varphi(\sigma^{*})\leq\int\max_{k\in\Theta}\varphi(h(k,y))P_{\eta}(\diff y) & =\int_{A}\max_{k\in\Theta}\varphi(h(k,y))P_{\eta}(\diff y)\\
 & \leq\int_{A}\max_{\sigma\in\sdInt}\varphi(\sigma)P_{\eta}(\diff y)\\
 & =\max_{\sigma\in\sdInt}\varphi(\sigma)\int_{A}P_{\eta}(\diff y)=\varphi(\sigma^{*}).
\end{align*}
Therefore, 
\[
\int_{A}\max_{k\in\Theta}\varphi(h(k,y))P_{\eta}(\diff y)=\max_{\sigma\in\sdInt}\varphi(\sigma).
\]
For any $y\in A$, since $h(k,y)\in\sdInt$, 
\[
0<\max_{k\in\Theta}\varphi(h(k,y))\leq\max_{\sigma\in\sdInt}\varphi(\sigma).
\]
Then there must exist $B\subset A$ with $P_{\eta}(B)=1$ such that
for $y\in B$, 
\[
\max_{k\in\Theta}\varphi(h(k,y))=\max_{\sigma\in\sdInt}\varphi(\sigma),
\]
or 
\[
\expt[\varphi(h(K,y))]=\max_{\sigma\in\sdInt}\varphi(\sigma).
\]
For any $f\in\fspace$, let $\varphi=f(x)+C$ with $C=-\min_{x\in\sdInt}f(x)+1$,
then $\varphi>0$ on $x\in\sdInt$. We have 
\[
\expt[f(h(K,y))]=\expt[\varphi(h(K,y))]-C=\max_{\sigma\in\sdInt}\varphi(\sigma)-C=\max_{\sigma\in\sdInt}f(\sigma).
\]
Therefore, for $y\in B$, 
\begin{equation}
h(K,y)\sim\Maximal\sdInt.\label{eq:maximal-h-ky}
\end{equation}
If there exist two distinct $y_{1},y_{2}\in B$, 
\[
\delta\coloneqq h(K,y_{1})-h(K,y_{2})\neq0
\]
we must have 
\[
h(K,y_{1})=h(K,y_{2})+\delta\sim\Maximal[\sdl+\delta,\sdr+\delta].
\]
This is a contradiction against \ref{eq:maximal-h-ky}. Then we have,
for any $y\in B$,
\[
h(V_{1},y)\equiv h(V_{1},c)\eqqcolon h(V_{1}),
\]
where $c$ is any constant chosen from $B$.
\end{proof}

\begin{proof}[Proof of \ref{prop:unique-decompose}]
Since $W\in\bar{\mysetrv{H}}_{s}$, we have $W=f(K,\eta)$ where $K\sim\Maximal(\Theta)$
and $\eta$ is classical satisfying $K\seqind\eta$. Suppose there
exist $V_{i}\in\bar{\myset{H}}_{s}$ and $\stdrv_{i}\in\bar{\myset{H}}_{s}$
such that $W=V_{1}\stdrv_{1}=V_{2}\stdrv_{2}$. To be specific, without
loss of generality, we can assume
\begin{align*}
V_{i} & =h_{i}(K,\eta),\\
\stdrv_{i} & =g_{i}(K,\eta),
\end{align*}
such that $f(K,\eta)=h_{1}(K,\eta)g_{1}(K,\eta)=h_{2}(K,\eta)g_{2}(K,\eta)$.
Then we have 
\[
g_{2}(K,\eta)=\frac{h_{1}(K,\eta)}{h_{2}(K,\eta)}g_{1}(K,\eta).
\]
Note that $h_{i}(K,\eta)\sim\Maximal\sdInt$, then by \ref{prop:hxy-hx},
there exist $B_{i}$ with $P_{\eta}(B_{i})=1$ such that $h_{i}(x,y)=h_{i}(x)$
when $y\in B_{i}$. Let $B=B_{1}\cap B_{2}$ then we still have $P_{\eta}(B)=1$.
Then we have, for any $\varphi$,
\begin{align*}
\expt[\varphi(\frac{h_{1}(K,\eta)}{h_{2}(K,\eta)}g_{1}(K,\eta))] & =\sup_{k\in\Theta}\lexpt[\varphi(\frac{h_{1}(k,\eta)}{h_{2}(k,\eta)}g_{1}(k,\eta))]\\
 & =\sup_{k\in\Theta}\int\varphi(\frac{h_{1}(k,y)}{h_{2}(k,y)}g_{1}(k,y))P_{\eta}(\diff y)\\
 & =\sup_{k\in\Theta}\int_{B}\varphi(\frac{h_{1}(k,y)}{h_{2}(k,y)}g_{1}(k,y))P_{\eta}(\diff y)\\
 & =\sup_{k\in\Theta}\int_{B}\varphi(\frac{h_{1}(k)}{h_{2}(k)}g_{1}(k,y))P_{\eta}(\diff y)\\
 & =\sup_{k\in\Theta}\int\varphi(\frac{h_{1}(k)}{h_{2}(k)}g_{1}(k,y))P_{\eta}(\diff y)\\
 & =\sup_{k\in\Theta}\lexpt[\varphi(\frac{h_{1}(k)}{h_{2}(k)}g_{1}(k,\eta))]=\expt[\varphi(\frac{h_{1}(K)}{h_{2}(K)}g_{1}(K,\eta))].
\end{align*}
Similarly, we can also show 
\[
\frac{h_{1}(K,\eta)}{h_{2}(K,\eta)}\eqdistn\frac{h_{1}(K)}{h_{2}(K)}.
\]
Note that 
\[
R(K)\coloneqq h_{1}(K)/h_{2}(K)\sim\Maximal(S),
\]
where $S=\{h_{1}(k)/h_{2}(k),k\in\Theta\}$. By \ref{lem:g-ky-classical},
letting $Z\eqdistn\CN(0,1)$, the fact that $g_{1}(K,\eta)\eqdistn Z$
implies, for $k\in\Theta$,
\[
g_{1}(k,\eta)\eqdistn Z.
\]
Then we have, with $\psi_{k}(x)\coloneqq\varphi(R(k)x)$, 
\begin{align*}
\expt[\varphi(R(K)g_{1}(K,\eta))] & =\sup_{k\in\Theta}\lexpt[\varphi(R(k)g_{1}(k,\eta))]\\
 & =\sup_{k\in\Theta}\lexpt[\psi_{k}(g_{1}(k,\eta))]=\sup_{k\in\Theta}\lexpt[\psi_{k}(Z)]\\
 & =\sup_{k\in\Theta}\lexpt[\varphi(R(k)Z)]=\sup_{s\in S}\lexpt[\varphi(sZ)].
\end{align*}
Meanwhile, 
\[
\expt[\varphi(R(K)g_{1}(K,\eta))]=\expt[\varphi(g_{2}(K,\eta))]=\lexpt[\varphi(Z)].
\]
Then we have the set $S$ has to be a singleton $\{1\}$. It means
that $R(K)\sim\Maximal(\{1\})$ or $R(K)=1$ (in a quasi-surely sense).
Then we also have 
\[
\frac{h_{1}(K,\eta)}{h_{2}(K,\eta)}\eqdistn\frac{h_{1}(K)}{h_{2}(K)}\eqdistn\Maximal(\{1\}).
\]
It means that $h_{1}(K,\eta)=h_{2}(K,\eta)$ then $g_{1}(K,\eta)=g_{2}(K,\eta)$.
The uniqueness has been proved. 
\end{proof}

\begin{proof}[Proof of \ref{thm:conn-G}]
This is a direct consequence of \ref{prop:convex-concave-case}. Let $\stdrv\sim \CN(0,1)$.
On the one hand, for any $\varphi\in \fspace$, as discussed in \ref{rem:comp-semiGN-GN},
we have 
\[
\expt[\varphi(W^G)]\geq \sup_{\sigma\in\sdInt} \lexpt[\varphi(\sigma\stdrv)]=\expt[\varphi(W)].
\]
On the other hand, when $\varphi$ is convex or concave, by \ref{prop:convex-concave-case}, we have 
\[
\expt[\varphi(W)] \geq \max_{\sigma\in\{\sdl,\sdr\}} \lexpt[\varphi(\sigma \stdrv)] \geq \expt[\varphi(W^G)]. 
\]
Hence, we have $\expt[\varphi(W^G)]=\expt[\varphi(W)]$ under convexity (or concavity) of $\varphi$.

For readers' convenience, we include an explicit proof on why we have such results for semi-$G$-normal distribution. For techinical convenience, we assume $\varphi$ is second order differentiable. 
%
%
From the representation of the semi-$G$-normal distribution (\ref{prop:convex-concave-case}), with $G(v)\coloneqq \lexpt_\lprob[\varphi(v\stdrv)](v\in\sdInt)$, our goal is to show
\[
\expt[\varphi(W)]=\max_{v\in\sdInt}G(v)= \begin{cases}
G(\sdr) & \varphi\text{ is convex}\\
G(\sdl) & \varphi\text{ is concave}
\end{cases}.
  \]
First of all, by Taylor expansion 
$
\varphi(x)=\varphi(0)+\varphi^{(1)}(0)x+\varphi^{(2)}(\xi_{x})\frac{x^{2}}{2}
$
with $\xi_x \in (0,x)$, we have,
\[
G(v) = \lexpt_\lprob[\varphi(0)+\varphi^{(1)}(0)v\stdrv+\varphi^{(2)}(\xi_{v\stdrv})\frac{v^{2}}{2}\stdrv^{2}] = \varphi(0)+\frac{1}{2}\lexpt_\lprob[\varphi^{(2)}(\xi_{v\stdrv})(v\stdrv)^{2}],
\]
where $\xi_{v\stdrv}\in(0,v\stdrv)$ is a random variable depending on $\stdrv$. Let $M\coloneqq v\stdrv\sim N(0,v^{2})$,
then 
\[
K(v)\coloneqq \lexpt_\lprob[\varphi^{(2)}(\xi_{v\stdrv})(v\stdrv)^{2}]=\lexpt_\lprob[\varphi^{(2)}(\xi_{M})M^{2}]=\int \phi(\frac{m}{v})\varphi^{(2)}(\xi_{m})m^{2}\diff m,
  \]
 where $\phi(x)$ is the density of $\CN(0,1)$.
When $\varphi$ is convex, we can use the fact $\varphi^{(2)}\geq 0$
to show
the monotonicity of $K(v)$:\begin{align*}
K^{\prime}(v) 
& =\int\phi^{\prime}(\frac{m}{v})(-\frac{m}{v^2})\varphi^{(2)}(\xi_{m})m^{2}\diff m\\
 & =\int\frac{1}{\sqrt{2\pi}}\underbrace{\left[\frac{m^{2}}{v^{3}}\exp\left(-\frac{m^{2}}{2v^{2}}\right)\right]}_{\geq0\text{ for }v\in\sdInt}\underbrace{\vphantom{\left[\frac{1}{2v}\exp\left(-\frac{m^{2}}{2v^{2}}\right)\right]}\varphi^{(2)}(\xi_{m})\,m^{2}}_{\geq0}\diff m \geq 0.
\end{align*}
Its tells us $K(v)$ is increasing with respect to $v\in \sdInt$, then
$K(v)$ reaches its maximum at $v=\sdr$. Hence,
\[
\expt[\varphi(W)] = \max_{v\in\sdInt}G(v) = \max_{v\in\sdInt}(\varphi(0)+\frac{K(v)}{2})=G(\sdr).
  \]
When $\varphi$ is concave, $-\varphi$ is convex. Replace 
$\varphi$ above with $-\varphi$ and repeat the same procedure, we are able to show $-G(v)$ is increasing with respect to $v$, that is, $G(v)$ is decreasing and reaches its maximum at $\sdl$.
\end{proof}

\subsection{Proofs in \ref{subsec:indep-semignorm}}
\label{pf:subsec:indep-semignorm}

The proofs in this section are mainly based on the results in \ref{subsec:basic-re-indep-seq} which provides fruitful tools to deal with the independence of sequence in this framework.

\begin{lem}
\label{lem:ind-Y-seq-norm}
In sublinear expectation space, 
	for a sequence of i.i.d. random variables $\{\stdrv_i\}_{i=1}^n\sim\GN(0,[1,1])$ (namely, $\stdrv_1\seqind \stdrv_2\seqind \dotsc \seqind \stdrv_n$), we have 
	\[
	(\stdrv_1,\stdrv_2,\dotsc,\stdrv_n)^T\sim N(\myvec 0,\idtymat_n^2),
	 \]
	where $\idtymat_n$ is the $n\times n$ identity matrix. 
\end{lem}

\begin{proof}
	Since the distribution of $\stdrv_i$ can be treated as the classical $\CN(0,1)$, the sequential independence can be treated the classical independence (\ref{rem:seq-indep-and-cl-indep}). Then we can get the required results by applying the classical logic. 
\end{proof}

\begin{rem}
Since the independence of $\{\stdrv_i\}_{i=1}^n$ is classical, the order of independence can also be arbitrarily switched so we can easily obtain a result similar to \ref{thm:maximal-multi-relation}.
\end{rem}

\begin{prop}
\label{prop:stdrv-uni-multi}
	For a sequence of i.i.d. random variables $\{\stdrv_i\}_{i=1}^n\sim\GN(0,[1,1])$, the following three statements are equivalent: 
	\begin{itemize}
		\item[(1)] $\nseqind{\stdrv}{n}$,
		\item[(2)] $\stdrv_{k_{1}}\seqind \stdrv_{k_{2}}\seqind\cdots\seqind \stdrv_{k_{n}}$ for any permutation $\{k_j\}_{j=1}^n$ of $\{1,2,\dotsc,n\}$,
		\item[(3)] $(\stdrv_{1},\stdrv_{2},\dotsc,\stdrv_{n})\sim\CN(\myvec{0},\varI\mymat{I}_{n})$.
	\end{itemize}
\end{prop}

\begin{proof}[Proof of \ref{thm:equiv-fullseqind}]
Since the fully-sequential independence implies (F1) and (F2) by \ref{ind-subseq-2} and \ref{prop:indep-sub-seq}, we
only need to show the other direction. When $n=2$, this result is
a consequence of \ref{prop:four-obj-seqind-gen}. For
$i\leq j$, let 
\[
(V,\stdrv)_{i}^{j}\coloneqq(V_{i},\stdrv_{i},V_{i+1},\stdrv_{i+1},\dotsc,V_{j},\stdrv_{j}).
\]
Next we proceed by math induction. Suppose the result holds for $n=k$ with $k\geq 2$. For $n=k+1$, we only need to
show: given the conditions
\begin{enumerate}
\item $(V_{1},\stdrv_{1})\seqind(V_{2},\stdrv_{2})\seqind\cdots\seqind(V_{k+1},\stdrv_{k+1})$, 
\item $V_{i}\seqind\stdrv_{i}$ for $i=1,2,\dotsc,k+1$, 
\end{enumerate}
we have the fully-sequential independence: 
\begin{equation}
	\label{eq:pf-full-seq}
	V_{1}\seqind\stdrv_{1}\seqind\cdots\seqind V_{k}\seqind\stdrv_{k}\seqind V_{k+1}\seqind\stdrv_{k+1}.
\end{equation}
Since all the independence relation in \ref{eq:pf-full-seq} until the term $\stdrv_{k}$ can
be guaranteed by the presumed result with $n=k$, we only need to
show the additional independence: 
\begin{enumerate}
\item $(V,\stdrv)_{1}^{k}\seqind V_{k+1}$,
\item $((V,\stdrv)_{1}^{k},V_{k+1})\seqind\stdrv_{k+1}$. 
\end{enumerate}
The first one comes from (F1) whose definition implies $(V,\stdrv)_{1}^{k}\seqind(V_{k+1},\stdrv_{k+1})$.
The second one comes from \ref{lem:three-obj-seqind-2} given the following statements: 
\begin{enumerate}
\item $(V,\stdrv)_{1}^{k}\seqind(V_{k+1},\stdrv_{k+1})$ by (F1); 
\item $(V,\stdrv)_{1}^{k}\seqind V_{k+1}$ by the first one.
\end{enumerate}
Then we have the required result for $n=k+1$. The proof is finished by math induction.
\end{proof}

\begin{proof}[Proof of \ref{thm:equiv-semiseqind}]
	First the definition \ref{semi-seq-ind-n} of semi-sequential independence implies (S1) to (S3) by 
	\ref{ind-subseq-2} and \ref{prop:indep-sub-seq}. We only need to check the other direction. For $i\leq j$, let $V_{i}^{j}\coloneqq(V_{i},V_{i+1},\dotsc,V_{j})$
and similarly define the notation $\stdrv_{i}^{j}$. Our goal can
be expanded as by \ref{defn:indep-seq}: 
\begin{enumerate}
\item $V_{1}^{l}\seqind V_{l+1}$ for any $l=1,2,\dotsc,n-1$, 
\item $(V_{1}^{n-1},V_{n},\stdrv_{1}^{l})\seqind\stdrv_{l+1}$ for any $l=1,2,\dotsc,n-1$. 
\end{enumerate}
The first one comes from (S2). For the second one, note that we have 
\begin{enumerate}
\item $(V_{1}^{n-1},V_{n})\seqind(\stdrv_{1}^{l},\stdrv_{l+1})$ by (S1), 
\item $\stdrv_{1}^{l}\seqind\stdrv_{l+1}$ by (S3),
\item $V_{1}^{n-1}\seqind V_{n}$ by (S2),
\end{enumerate}
then by \ref{prop:four-obj-seqind-gen}, we have proved the second relation.
\end{proof}

\begin{proof}[Proof of \ref{thm:n-semi-seqind-sym}]
The idea main the equivalent definition of semi-sequential independence given by \ref{thm:equiv-semiseqind}, which shows the symmetry within $V$ part and $\stdrv$ part. 
The equivalence of the three statemenst will be proved in this logic: 
\[
(3)\iff(1) \iff (2) .
\]
Let $\pi:(x_{1},x_{2},\dotsc,x_{n})\to(x_{k_{1}},x_{k_{2}},\dotsc,x_{k_{n}})$ denote a permutation function. 

$(1) \iff (2)$. It is a direct translation of \ref{thm:equiv-semiseqind} by considering the equivalence in each part: 
\begin{enumerate}
	\item The equivalence in (S1) can be seen by by treating each vector as a function of each other under the permutation $\pi$ (or $\pi^{-1}$.)
	\item The equivalence in (S2) comes from \ref{thm:maximal-multi-relation}. 
	\item The equivalence in (S3) comes from \ref{prop:stdrv-uni-multi}. 
\end{enumerate} 

$(3) \iff (1)$. Let $\myvec W\coloneqq (W_1,W_2,\dotsc,W_n)$. Then 
	\[
	\myvec W = (V_1\stdrv_1,V_2\stdrv_2,\dotsc,V_n\stdrv_n) = \gdsd \myvec{\stdrv}
	 .\]
	 Then we can decompose (3) into three conditions each of which is equivalent to the condition in (1) under the context of \ref{thm:equiv-semiseqind}:
	 \begin{enumerate}
\item Since $\mymat{V}=(V_{1},\dotsc,V_{n})\diag(1,\dotsc,1)$, we have
$\mymat{V}\seqind\myvec{\stdrv}$ if and only if $(V_{1},\dotsc,V_{n})\seqind\myvec{\stdrv}$
which is (S1) in \ref{thm:equiv-semiseqind}.
\item Note that $\mymat{V}\sim\Maximal(\sdInt\idtymat_{n})$ is equivalent
to $\myvec{V}\sim\Maximal(\sdInt^{n})$ which is further equivalent
to (S2): $\{V_{i}\}_{i=1}^{n}$ are sequentially independent.
\item By \ref{prop:stdrv-uni-multi}, we have $\myvec{\stdrv}\sim\CN(\myvec{0},\idtymat_{n})$ is equivalent
to (S3): $\{\stdrv\}_{i=1}^{n}$ are sequentially independent.\qedhere 
\end{enumerate}
\end{proof}

\subsection{Proofs in \ref{subsec:represent-semignorm-indep}}
\label{pf:subsec:represent-semignorm-indep}

\begin{proof}[Proof of \ref{thm:represent-n-seqind-semignorm}] 
\textbf{(A note on the finiteness of sublinear expectations)}
For any $\varphi\in\fspace(\numset{R}^{k+1})$, it means there exists
$m\in\numset{N}_{+}$ and $C_{0}>0$ such that for $\myvec{x},\myvec{y}\in\numset{R}^{k+1},$
\[
\abs{\varphi(\myvec{x})-\varphi(\myvec{y})}\leq C_{0}(1+\norm{\myvec{x}}^{m}+\norm{\myvec{y}}^{m})\norm{\myvec{x}-\myvec{y}}.
\]
Without loss of generality, we can assume $\varphi(\myvec{0})=0$,
then we have $\abs{\varphi(x)}\leq C_{0}(1+\norm{x}^{m})\norm{x}$.
It implies
\[
\expt[\abs{\varphi(\myvec{W})}]\leq C_{0}(\expt[\norm{\myvec{W}}]+\expt[\norm{\myvec{W}}^{m}]).
\]
To validate $\expt[\abs{\varphi(\myvec{W})}]<\infty$ under each case,
it will be sufficient to confirm 
the finiteness of this sublinear expectation:
for any $q\in\numset{N}_{+}$, 
\begin{equation}
\label{eq:finite-moment}
	\expt[\norm{\myvec{W}}^{q}]<\infty.
\end{equation}

\textbf{(Semi-sequential independence case)}
Under the independence specified by \ref{eq:cond-semiseqind}, from \ref{thm:represent-multi-semignorm}, we have 
\[
\myvec{W}\sim\semiGN(\myvec{0},\myset{V}),
\]
where $\myset{V}=\{\diag(\sigma_{1}^2,\sigma_{2}^2,\dotsc,\sigma_{n}^2):\sigma_{i}\in\sdInt\}$.
Therefore, 
\[
\expt[\varphi(\myvec{W})]=\max_{\mymat{V}\in\myset{V}}\lexpt_\lprob[\varphi(\mymat{V}^{1/2}\myvec{\epsilon})]=\max_{\myvec{\sigma}\in\myset{C}_n\sdInt}\expt[\varphi(\myvec{\sigma}*\myvec{\epsilon})],
\]
where $\mymat{V}^{1/2}$ is the symmetric square root of $\mymat{V}$ and
$
\myset{C}_n\sdInt\coloneqq \{\myvec{\sigma}:(\sigma_1,\sigma_2,\dotsc,\sigma_n) \in \sdInt^n\}.
$
At the same time, we can validate the finiteness $\expt[\abs{\varphi(\myvec{W})}]<\infty$, because $\mymat{V}^{1/2}\myvec{\epsilon}$ follows a classical multivariate normal distribution, then for any $q\in\numset{N}_{+}$, $\lexpt_{\lprob}[\norm{\mymat{V}^{1/2}\myvec{\epsilon}}^{q}]<\infty$, which implies \ref{eq:finite-moment}.

Next, since we have $\mysetrv{C}_n\sdInt\subset\mysetrv{S}_n\sdInt$,
we only need to show for any $\myvec{\sigma}\in\mysetrv{S}_n\sdInt$,
\begin{equation}
\label{ineq:semi-seq-case-1}
	\lexpt_\lprob[\varphi(\myvec{\sigma}*\myvec{\epsilon})]\leq\expt[\varphi(\myvec{W})].
\end{equation}
Note that $\myvec{\sigma}\independent\myvec{\epsilon}$ for $\myvec{\sigma}\in\mysetrv{S}_n\sdInt$
and the random vector $\myvec{\sigma}$ must follow a joint distribution
supporting on a subset of $\sdInt^{n}$, then we can apply the representation
of multivariate semi-$G$-normal distribution (\ref{thm:represent-multi-semignorm}) to get the inequality \ref{ineq:semi-seq-case-1}. 


\textbf{(Sequential independence case)} We proceed by mathematical induction. For $n=1$, the result \ref{eq:represent-n-seqind-1,eq:represent-n-seqind-2} as well as the finiteness \ref{eq:finite-moment} hold by applying \ref{thm:represent-uni-semignorm}. 
Suppose they also hold for $n=k$ with $k\in \numset{N}_+$.
Our objective is to prove them when $n=k+1$ by using the result with $n=k$. We decompose this goal into three inequalities: 
\begin{equation}
\label{eq:pf-ineq-1}
	\expt[\varphi(\myvec{W}_{(k+1)})]\leq \sup_{\myvec{\sigma}\in\mysetrv{L}_{k+1}\sdInt}\lexpt_\lprob[\varphi(\myvec{\sigma}*\myvec{\stdrv})],
\end{equation}
\begin{equation}
	\label{eq:pf-ineq-0}
	\sup_{\myvec{\sigma}\in\mysetrv{L}_{k+1}\sdInt}\lexpt_\lprob[\varphi(\myvec{\sigma}*\myvec{\stdrv})]\leq \sup_{\myvec{\sigma}\in\mysetrv{L}^*_{k+1}\sdInt}\lexpt_\lprob[\varphi(\myvec{\sigma}*\myvec{\stdrv})],
\end{equation} 
and 
\begin{equation}
\label{eq:pf-ineq-2}
	\sup_{\myvec{\sigma}\in\mysetrv{L}^*_{k+1}\sdInt}\lexpt_\lprob[\varphi(\myvec{\sigma}*\myvec{\stdrv})]\leq \expt[\varphi(\myvec{W}_{(k+1)})].
\end{equation}
After we check the three inequalities above, $\sup$ can be changed to $\max$ since we will show the sublinear expectation can be reached by some $\myvec{\sigma}\in \mysetrv{L}_{k+1}\sdInt$ in the proof of \ref{eq:pf-ineq-1}.

First of all, \ref{eq:pf-ineq-0} is straightforward due to the fact that 
$\mysetrv{L}_{k+1}\sdInt\subset \mysetrv{L}^*_{k+1}\sdInt$. 

Second, to validate \ref{eq:pf-ineq-1}, it is sufficient show the sublinear expectation $\expt[\varphi(\myvec{W})]$ can be reached by choosing some $\myvec{\sigma}\in\mysetrv{L}_{k+1}\sdInt$.
In fact, we can directly select it by the iterative procedure (similar to the idea of \ref{iterate-G-1}). 
\begin{align*}
\expt[\varphi(W_{1},W_{2},\dotsc,W_{k+1})] & =\expt[\varphi(\myvec{W}_{(k)},W_{k+1})]\\
 & =\expt[\expt[\varphi(\myvec{w}_{(k)},W_{k+1})]_{\myvec{w}_{(k)}=\myvec{W}_{(k)}}]\\
 & =\expt[(\max_{\sigma_{k+1}\in\sdInt}\lexpt_{\lprob}[\varphi(\myvec{w}_{(k)},\sigma_{k+1}\epsilon_{k+1})])_{\myvec{w}_{(k)}=\myvec{W}_{(k)}}]\\
 & =\expt[(\lexpt_{\lprob}[\varphi(\myvec{w}_{(k)},v_{k+1}(\myvec{w}_{(k)})\epsilon_{k+1})]],
\end{align*}
where $v_{k+1}(\cdot)$ is the maximizer depending on the value of $\myvec{w}_{(k)}$.
\begin{claim}
	\label{claim:varphi-k}
	For any $\varphi\in \fspace(\numset{R}^{k+1})$, let 
	\[
	\varphi_{k}(x) \coloneqq \max_{\sigma_{k+1}\in\sdInt}\lexpt_{\lprob}[\varphi(x,\sigma_{k+1}\epsilon_{k+1})].
	\]
	Then we have $\varphi_{k}\in\fspace(\numset{R}^{k})$.
\end{claim}
To apply the result when $n=k$, we first confirm that $\varphi_{k}\in\fspace(\numset{R}^{k})$ (due to \ref{claim:varphi-k})
Then we have
\begin{align*}
\expt[\varphi(\myvec{W}_{(k)},W_{k+1})] & =\expt[\varphi_{k}(\myvec{W}_{(k)})],\\
 & =\max_{\myvec{\sigma}_{(k)}\in\mysetrv{L}_{k}{\sdInt}}\lexpt_\lprob[\varphi_{k}(\myvec{\sigma}_{(k)}*\myvec{\epsilon}_{(k)})]=\lexpt_\lprob[\varphi_{k}(\myvec{v}_{(k)}*\myvec{\epsilon}_{(k)})]
\end{align*}
where $\myvec{v}_{(k)}\in\mysetrv{L}_{k}{\sdInt}$ is
the maximizer. From this procedure, we can choose 
\[
\myvec{v}_{(k+1)}\coloneqq(\myvec{v}_{(k)},v_{k+1}(\myvec{v}_{(k)}*\myvec{\epsilon}_{(k)})),
\]
which is corresponding to an element in $\mysetrv{L}_{k+1}\sdInt$.
Then it is easy to confirm that $\expt[\varphi(\myvec{W}_{(k+1)})]=\lexpt_\lprob[\varphi(\myvec{v}_{(k+1)}*\myvec{\epsilon}_{(k+1)})]$
by repeating the procedure above.
Meanwhile, the finiteness \ref{eq:finite-moment} is also guaranteed since, for any $q\in\numset{N}_{+}$,
choose $\varphi(\cdot)=\norm{\cdot}^{q}\in\fspace$, we have 
\[
\expt[\norm{\myvec{W}_{(k+1)}}^{q}]=\expt[\varphi(\myvec{W}_{(k+1)})]=\expt[\varphi_{k}(\myvec{W}_{(k)})]<\infty,
\]
due to the confirmed fact that $\varphi_{k}\in\fspace$ and the assumed \ref{eq:finite-moment} for $n=k$.  

Third, as an equivalent way of viewing \ref{eq:pf-ineq-2}, we need to prove for any $\myvec{\sigma}_{(k+1)}\in\mysetrv{L}^*_{k+1}\sdInt$, the corresponding linear expectation is dominated by $\expt[\varphi(\myvec{W}_{(k+1)})${]}. 
Actually, we can write the classical expectation as
\begin{equation}
\lexpt_{\lprob}[\varphi(\myvec{\sigma}_{(k)}\myvec{\epsilon}_{(k)},\sigma_{k+1}\epsilon_{k+1})]=\lexpt_{\lprob}[\lexpt_{\lprob}[\varphi(\myvec{\sigma}_{(k)}\myvec{\epsilon}_{(k)},\sigma_{k+1}\epsilon_{k+1})|\sigmafield{F}_{k}]].\label{eq:expt-n-sigep}
\end{equation}
Recall the notation we used in the proof of \ref{eq:pf-ineq-1}, 
\[
\varphi_{k}(\myvec{w}_{(k)})\coloneqq\max_{\sigma_{k+1}\in\sdInt}\lexpt_{\lprob}[\varphi(\myvec{w}_{(k)},\sigma_{k+1}\epsilon_{k+1})]=\expt[\varphi(\myvec{w}_{(k)},W_{k+1})].
\]
For the conditional expectation part in \ref{eq:expt-n-sigep}, since the information of $(\myvec{\sigma}_{(k)},\myvec{\epsilon}_{(k)})$
is given and $\sigma_{k+1}\independent\epsilon_{k+1}|\sigmafield{F}_{k}$,
from the representation of univariate semi-$G$-normal (\ref{thm:represent-uni-semignorm}),
it must satisfy: 
\begin{align*}
\lexpt_{\lprob}[\varphi(\myvec{\sigma}_{(k)}\myvec{\epsilon}_{(k)},\sigma_{k+1}\epsilon_{k+1})|\sigmafield{F}_{k}] & \leq\max_{\sigma_{k+1}\in\sdInt}\lexpt_{\lprob}[\varphi(\myvec{\sigma}_{(k)}\myvec{\epsilon}_{(k)},\sigma_{k+1}\epsilon_{k+1})|\sigmafield{F}_{k}]=\varphi_{k}(\myvec{\sigma}_{(k)}\myvec{\epsilon}_{(k)}).
\end{align*}
Hence, by taking expectations on both sides and applying the presumed result for $n=k$,
we have
\begin{align*}
\lexpt_{\lprob}[\varphi(\myvec{\sigma}_{(k)}\myvec{\epsilon}_{(k)},\sigma_{k+1}\epsilon_{k+1})] & \leq\lexpt_{\lprob}[\varphi_{k}(\myvec{\sigma}_{(k)}\myvec{\epsilon}_{(k)})]\leq\max_{\myvec{\sigma}_{(k)}\in \mysetrv{L}^*_{k}\sdInt }\lexpt_{\lprob}[\varphi_{k}(\myvec{\sigma}_{(k)}\myvec{\epsilon}_{(k)})]\\
 & =\expt[\varphi_{k}(\myvec{W}_{(k)})]=\expt[\expt[\varphi(\myvec{w}_{(k)},W_{k+1})]_{\myvec{w}_{(k)}=\myvec{W}_{(k)}}]\\
 & =\expt[\varphi(\myvec{W}_{(k)},W_{k+1})].
\end{align*}
Therefore, we have shown \ref{eq:pf-ineq-2}. The proof is completed by induction. 

\textbf{(fully-sequential independence case)} Note that fully-sequential independence implies the sequential independence and we have shown an explicit representation of $\expt[\varphi(\myvec{W})]$ for the latter situation. Hence, the representation here is the same as \ref{eq:represent-n-seqind-1,eq:represent-n-seqind-2}.

To prove \ref{claim:varphi-k}, 
first recall the definition of $\varphi\in\fspace(\numset{R}^{k+1})$, which means there exists
$m\in\numset{N}_{+}$ and $C_{0}>0$ such that for $\myvec{x},\myvec{y}\in\numset{R}^{k+1}$
\[
\abs{\varphi(\myvec{x})-\varphi(\myvec{y})}\leq C_{0}(1+\norm{\myvec{x}}^{m}+\norm{\myvec{y}}^{m})\norm{\myvec{x}-\myvec{y}}.
\]
Note that
\[
\varphi_{k}(\myvec{x})=\max_{v\in\sdInt}\lexpt[\varphi(\myvec{x},v\epsilon)].
\]
Then we write 
\begin{align*}
|\varphi_{k}(\myvec{x})-\varphi_{k}(\myvec{y})| & \leq\bigl|\max_{v\in\sdInt}(\lexpt[\varphi(\myvec{x},v\epsilon)-\varphi(\myvec{y},v\epsilon)])\bigr|\\
 & \leq\max_{v\in\sdInt}\lexpt[\abs{\varphi(\myvec{x},v\epsilon)-\varphi(\myvec{y},y\epsilon)}]\\
 & \leq\max_{v\in\sdInt}\lexpt[C_{0}(1+\norm{(\myvec{x},v\epsilon)}^{m}+\norm{(\myvec{y},v\epsilon)}^{m})\norm{\myvec{x}-\myvec{y}}],
\end{align*}
where we adapt the norm to lower dimension in the sense that $\norm{\myvec{a}_{(k+1)}}\coloneqq\norm{(\myvec{a}_{(k)},0)}$.
By triangle inequality, for any $v\in\sdInt$,
\[
\norm{(\myvec{x},v\epsilon)}\leq\norm{\myvec{x}}+\abs{v\epsilon}\leq\norm{\myvec{x}}+\sdr\abs{\epsilon},
\]
then 
\[
\norm{(x,v\epsilon)}^{m}\leq(\norm{x}+\sdr\abs{\epsilon})^{m}\leq\max\{1,2^{m-1}\}(\norm{x}^{m}+\sdr^{m}\abs{\epsilon}^{m}).
\]
Hence, with $C_{1}=C_{0}\max\{1,2^{m-1}\}$ and $C_{2}=C_{1}\max\{1,2\sdr^{m}\lexpt[\abs{\epsilon}^{m}]\}$,
\begin{align*}
|\varphi_{k}(\myvec{x})-\varphi_{k}(\myvec{y})| & \leq C_{0}\max_{v\in\sdInt}(1+\expt[\norm{(x,v\epsilon)}^{m}]+\expt[\norm{(y,v\epsilon)}^{m}])\norm{x-y}\\
 & \leq C_{1}(1+\norm{x}^{m}+\norm{y}^{m}+2\sdr^{m}\lexpt[\abs{\epsilon}^{m}])\norm{x-y}\\
 & \leq C_{2}(1+\norm{x}^{m}+\norm{y}^{m})\norm{x-y}.\qedhere
\end{align*}
\end{proof}



\begin{proof}[Proof of \ref{cor:convex-case-same}]
Under semi-sequential independence, note that 
\[
(W_{1},W_{2},\dotsc,W_{n})\sim\semiGN(\myvec{0},\myset{C}),
\]
with $\myset{C}=\{\diag(\sigma^2_{1},\sigma^2_{2},\dotsc,\sigma^2_{n}):\sigma_{i}\in\sdInt\}$.
It has the representation (\ref{cor:Properties-of-Maximal}) that, 
\[
\expt[\varphi(\myvec{W}_{(n)})]=\max_{\sigma_{i}\in\sdInt}\lexpt[\varphi(\diag(\sigma_{1},\sigma_{2},\dotsc,\sigma_{n})\myvec{\stdrv}_{(n)})],
\]
where $\myvec{\stdrv}_{(n)}$ follows a standard multivariate normal.
When $\varphi$ is convex, by simply repeating the idea of \ref{thm:conn-G}
in multivariate case, we can show 
\[
\expt[\varphi(\myvec{W}_{(n)})]=\lexpt[\varphi(\diag(\sdr,\sdr,\dotsc,\sdr)\myvec{\stdrv}_{(n)}).
\]
Accordingly, when $\varphi$ is concave, we can get similar result
with $\sdr$ replaced by $\sdl$. 

Under sequential independence, based on the idea of showing 
\[
\expt[\varphi(\myvec{W})]=\max_{\myvec{\sigma}\in\mysetrv{L}_{n}\sdInt}\lexpt[\varphi(\myvec{\sigma}*\myvec{\stdrv})],
\]
in \ref{thm:represent-n-seqind-semignorm}. The maximizer can be obtained by implementing
the iterative algorithm: with $\varphi_{0}\coloneqq\varphi$, $i=1,2,\dotsc,n$,
\begin{equation}
\label{eq:pf-iterate-varphi}
	\varphi_{i}(\myvec{x}_{(n-i)})=\max_{\sigma_{n-i+1}\in\sdInt}\lexpt[\varphi_{i-1}(\myvec{x}_{(n-i)},\sigma_{n-i+1}\epsilon_{n-i+1})].
\end{equation}
Then we only need to record the optimizer $\sigma_{n-i+1}(\cdot)$
which is a function of $\myvec{x}_{(n-i)}$ to get the maximizer $\myvec{\sigma}^{*}\in\mysetrv{L}_{n}\sdInt$.
First we can show that, for $i=1,2,\dotsc,n$,
\begin{equation}
\label{rel:varphi-convex}
	\varphi_{i-1}\text{ is convex (concave)}\implies\varphi_{i}\text{ is convex (concave)},
\end{equation}
namely, the convexity (or concavity) of $\varphi_{i-1}$ can be carried
over to $\varphi_{i}$. Actually, if $\varphi_{i-1}$ is convex (in
$\numset{R}^{n-i+1})$, it must be convex with respect to each subvector of arguments. Then by applying \ref{thm:conn-G}, we have 
\begin{equation}
\label{eq:pf-iterate-varphi-2}
	\varphi_{i}(\myvec{x}_{(n-i)})=\lexpt[\varphi_{i-1}(\myvec{x}_{(n-i)},\sdr\epsilon_{n-i+1})],
\end{equation}
which also gives us the choice of $\sigma_{n-i+1}$. Then we can validate
the convexity of $\varphi_{i}$ by definition: with $\lambda\in[0,1]$,
$e\coloneqq\sdr\epsilon_{n-i+1}$,
\begin{align*}
\varphi_{i}(\lambda\myvec{x}_{(n-i)}+(1-\lambda)\myvec{y}_{(n-i)}) & =\lexpt[\varphi_{i-1}(\lambda\myvec{x}_{(n-i)}+(1-\lambda)\myvec{y}_{(n-i)},e)]\\
 & =\lexpt[\varphi_{i-1}(\lambda(\myvec{x}_{(n-i)},e)+(1-\lambda)(\myvec{y}_{(n-i)},e))]\\
 & \leq\lambda\lexpt[\varphi_{i-1}(\myvec{x}_{(n-i)},e)]+(1-\lambda)\lexpt[\varphi_{i-1}(\myvec{y}_{(n-i)},e)]\\
 & =\lambda\varphi_{i}(\myvec{x}_{(n-i)})+(1-\lambda)\varphi_{i}(\myvec{y}_{(n-i)}).
\end{align*}
We can follow the same arguments to show the concave case. 
Finally, we can
start from the convexity (concavity, respectively) of $\varphi_{0}$ to show the convexity
of all $\varphi_{i}$ and along the way, we get each of the optimal $\sigma_{n-i+1}$
is equal to $\sdr$ ($\sdl$, respectively). 
\end{proof}

\begin{proof}[Proof of \ref{cor:convex-case-same-2}]
	The idea is similar to the proof of \ref{cor:convex-case-same} by replacing \ref{eq:pf-iterate-varphi} by 
	\[
	\varphi_{i}(\myvec{x}_{(n-i)})=\expt[\varphi_{i-1}(\myvec{x}_{(n-i)},W^G_{n-i+1})].
	\]
	In order to check the statement \ref{rel:varphi-convex}, due to the convexity of $\varphi$, we can use \ref{prop:convex-concave-case} to show \ref{eq:pf-iterate-varphi-2}. The remaining part is the same as the proof of \ref{cor:convex-case-same}.
\end{proof}

\subsection{Proofs in \ref{subsec:interpret-asymmetry-indep}}

The goal of this section is to prove \ref{prop:represent-two-semignorm}, which is a simple representation for $\expt[\varphi(W_{1},W_{2})]$
for $\varphi\in C_{\text{s.poly}}$. Throughout this section, without further notice, we consider $W_{1}\eqdistn W_{2}\eqdistn\semiGN(0,\varI)$ imposed with sequential independce $W_{1}\seqind W_{2}$. We also have the expression $W_{i}\coloneqq V_{i}\stdrv_{i}$ with $V_{i}\sim\Maximal\meanInt$,
$\stdrv_{i}\sim\CN(0,1)$ and $V_{i}\seqind\stdrv_{i}$. 

\begin{lem}
\label{lem:poly-omit}
For $p,q\in\numset{N}_{+}$,
if $q$ is odd, 
\[
\expt[W_{1}^{p}W_{2}^{q}]=-\expt[-W_{1}^{p}W_{2}^{q}]=0,
\]
that is, it has certain zero mean. 
\end{lem}

\begin{proof}
Directly work on the sublinear expectation by imposing the sequential
independence. Let $K(x)\coloneqq\expt[xW_{2}^{q}],$ then we have
\[
K(x)=x^{+}\expt[W_{2}^{q}]+x^{-}\expt[-W_{2}^{q}]=0,
\]
because we are essentially working on the odd-moment of $\sigma\epsilon$
with $\sigma\in\sdInt$. Then we have 
\begin{align*}
\expt[W_{1}^{p}W_{2}^{q}] & =\expt[\expt[w_{1}^{p}W_{2}^{q}]_{w_{1}=W_{1}}]\\
 & =\expt[K(W_{1}^{p})]=0.
\end{align*}
Similarly, we have $-\expt[-W_{1}^{p}W_{2}^{q}]=0$. 
\end{proof}

\begin{lem}
\label{lem:even-moment-two-semignorm}
For $\varphi(x_{1},x_{2})=(x_{1}+x_{2})^{n}$ with \textbf{even} positive
integer $n$, 
\[
\expt[\varphi(W_{1},W_{2})]=\lexpt[\varphi(\sdr\stdrv_{1},\sdr\stdrv_{2})].
\]
Furthermore, we have the even moments of $(W_{1}+W_{2})$:
\[
\expt[(W_{1}+W_{2})^{n}]=(n-1)!!\sdr^{n}2^{n/2}.
\]
\end{lem}

\begin{proof}
This result directly comes from due to the convexity of $\varphi$, which can be validated by considering its Hessian matrix.
Then we can check that 
\[
\expt[(W_{1}+W_{2})^{n}]=\lexpt[\sdr^{n}(\epsilon_{1}+\epsilon_{2})^{n}]=\lexpt[\sdr^{n}(\sqrt{2}\epsilon_{1})^{n}]=(n-1)!!\sdr^{n}2^{n/2}.
\]
\end{proof}

\begin{lem}
\label{lem:odd-moment-two-semignorm}
For $\varphi(x_{1},x_{2})=(x_{1}+x_{2})^{n}$ with \textbf{odd} positive
integer $n$, 
\begin{equation}
\expt[\varphi(W_{1},W_{2})]=\lexpt_{\lprob}[\varphi(\sigma_{1}\stdrv_{1},\sigma_{2}\stdrv_{2})],\label{eq:odd-moment-represent}
\end{equation}
where $(\sigma_{1},\sigma_{2})$ satisfies: 
\begin{align}
\sigma_{1} & =\sdr,\nonumber \\
\sigma_{2} & =\sigma_{2}(\sigma_{1}\epsilon_{1})\nonumber \\
 & =\sdr(\sigma_{1}\epsilon_{1})^{+}-\sdl(\sigma_{1}\stdrv_{1})^{-}.\label{eq:choice-2-sigma}
\end{align}
Furthermore, for odd $n\geq3$, we have the moments of $(W_{1}+W_{2})$:
\[
\expt[(W_{1}+W_{2})^{n}]=\sqrt{\frac{2}{\pi}}\bigl(\sum_{k=0}^{(n-3)/2}C_{k}\sdr^{2k+1}2^{k-1}k!\bigr),
\]
where $C_{k}=\binom{n}{2k+1}(n-2k-1)!!(\sdr^{n-2k-1}-\sdl^{n-2k-1}).$
\end{lem}

\begin{proof}
We can directly check the sublinear expectation 
\begin{align*}
\expt[(W_{1}+W_{2})^{n}] & =\expt[\sum_{i=0}^{n}\binom{n}{i}W_{1}^{i}W_{2}^{n-i}].
\end{align*}
Since the terms in the form of $W_{1}^{n}$, $W_{2}^{n}$ or $W_{1}^{i}W_{2}^{n-i}$
with even $i$ (then $n-i$ is odd) all have zero mean (with no ambiguity),
they can be omitted in the computation by \ref{lem:poly-omit}. Hence,
\begin{align*}
\expt[(W_{1}+W_{2})^{n}] & =\expt[\sum_{i\text{ odd}}\binom{n}{i}W_{1}^{i}W_{2}^{n-i}]\\
 & =\expt[\underbrace{\expt[\sum_{i\text{ odd}}\binom{n}{i}w_{1}^{i}W_{2}^{n-i}}_{\eqqcolon\varphi_{1}(w_{1})}]_{w_{1}=W_{1}}].
\end{align*}
The inner part can be expressed as 
\begin{align*}
\varphi_{1}(w_{1}) & =\expt[\sum_{i\text{ odd}}\binom{n}{i}w_{1}^{i}W_{2}^{n-i}]\\
 & =\max_{\sigma_{2}\in\sdInt}\lexpt[\sum_{i\text{ odd}}\binom{n}{i}w_{1}^{i}(\sigma_{2}\stdrv_{2})^{n-i}]\\
 & =\max_{\sigma_{2}\in\sdInt}\sum_{i\text{ odd}}\binom{n}{i}w_{1}^{i}\sigma_{2}^{n-i}\lexpt[\stdrv_{2}^{n-i}]\\
 & =\max_{\sigma_{2}\in\sdInt}\sum_{k=0}^{(n-3)/2}\binom{n}{2k+1}w_{1}^{2k+1}\sigma_{2}^{n-2k-1}(n-2k-1)!!\eqqcolon\max_{\sigma_{2}\in\sdInt}H(\sigma_{2};w_{1}).
\end{align*}
Notice that the monotonicity of $H(\sigma_{2};w_{1})$ with respect
to $\sigma_{2}$ depends on the sign of $w_{1}$. Hence 
\begin{align*}
\varphi_{1}(w_{1}) & =\begin{cases}
H(\sdr;w_{1}) & \text{ if }w_{1}\geq0\\
H(\sdl;w_{1}) & \text{ if }w_{1}<0
\end{cases}\\
 & =\ind{w_{1}\geq0}(H(\sdr;w_{1})-H(\sdl;w_{1}))+H(\sdl;w_{1})
\end{align*}
Then we have 
\begin{align*}
\expt[(W_{1}+W_{2})^{n}] & =\expt[\varphi_{1}(W_{1})]\\
 & =\expt[\ind{W_{1}\geq0}(H(\sdr;W_{1})-H(\sdl;W_{1}))+H(\sdl;W_{1})].
\end{align*}
Here we have 
\[
\expt[H(\sdl;W_{1})]=\expt[\sum_{k=0}^{(n-3)/2}\binom{n}{2k+1}\sigma_{2}^{n-2k-1}(n-2k-1)!!W_{1}^{2k+1}],
\]
with each $W_{1}^{2k+1}$ has certain mean zero, so $\expt[H(\sdl;W_{1})]=-\expt[-H(\sdl;W_{1})]=0$.
Therefore, 
\begin{align*}
\expt[(W_{1}+W_{2})^{n}] & =\expt[\ind{W_{1}\geq0}(H(\sdr;W_{1})-H(\sdl;W_{1}))]\\
 & =\expt[\ind{W_{1}\geq0}\sum_{k=0}^{(n-3)/2}C_{k}W_{1}^{2k+1}]\eqqcolon\expt[K(W_{1})]
\end{align*}
with $C_{k}=\binom{n}{2k+1}(n-2k-1)!!(\sdr^{n-2k-1}-\sdl^{n-2k-1})\geq0$.
Since $K(x)$ is a convex function by noting that it stays at
$0$ on $x<0$ and increases when $x\geq0$, 
we have 
\[
\expt[(W_{1}+W_{2})^{n}]=\lexpt[K(\sdr \stdrv_{1})].
\]
Therefore, we obtain the optimal of $(\sigma_{1},\sigma_{2})$ in
the form \ref{eq:choice-2-sigma}, which can be doubly checked by
plugging it back to the right handside of \ref{eq:odd-moment-represent}
to show the equality. We can further get the exact value of $\expt[(W_{1}+W_{2})^{n}]$
by continuing the procedure above, 
\begin{align*}
\expt[\ind{W_{1}\geq0}\sum_{k=0}^{(n-3)/2}C_{k}W_{1}^{2k+1}] & =\max_{\sigma_{1}\in\sdInt}\text{\lexpt}[\ind{\sigma_{1}\stdrv_{1}\geq0}\sum_{k=0}^{(n-3)/2}C_{k}(\sigma_{1}\stdrv_{1})^{2k+1}]\\
 & =\text{\lexpt}[\ind{\stdrv_{1}\geq0}\sum_{k=0}^{(n-3)/2}C_{k}(\sdr \stdrv_{1})^{2k+1}]\\
 & =\sum_{k=0}^{(n-3)/2}C_{k}\sdr^{2k+1}\lexpt[\ind{\stdrv_{1}\geq0}\stdrv_{1}^{2k+1}].
\end{align*}
Here we need to use the property of the classical half-normal distribution:
\begin{align*}
\lexpt[|\stdrv|^{2k+1}] & =\lexpt[|\stdrv^{2k+1}|]\\
 & =\lexpt[\ind{\stdrv_{1}\geq0}\stdrv_{1}^{2k+1}]+\lexpt[\ind{\stdrv_{1}<0}(-\stdrv_{1})^{2k+1}].
\end{align*}
Since $\stdrv$ and $-\stdrv$ have the same distribution, 
\begin{align*}
\lexpt[\ind{\stdrv_{1}<0}(-\stdrv_{1})^{2k+1}] & =\lexpt[\ind{\stdrv_{1}\leq0}(-\stdrv_{1})^{2k+1}]\\
 & =\lexpt[\ind{-\stdrv_{1}\geq0}(-\stdrv_{1})^{2k+1}]\\
 & =\lexpt[\ind{\stdrv_{1}\geq0}\stdrv_{1}^{2k+1}],
\end{align*}
then $\lexpt[|\stdrv|^{2k+1}]=2\lexpt[\ind{\stdrv_{1}\geq0}\stdrv_{1}^{2k+1}].$
Hence,
\[
\lexpt[\ind{\stdrv_{1}\geq0}\stdrv_{1}^{2k+1}]=\frac{1}{2}\lexpt[|\stdrv|^{2k+1}].
\]
Also notice that $\stdrv\sim\CN(0,1)$, then $|\stdrv|$ follows a half-normal
distribution or $|\stdrv|=|\stdrv|/1$ follows a $\chi_{1}$-distribution with
raw moments: 
\[
\lexpt[|\stdrv|^{n}]=2^{n/2}\frac{\Gamma((n+1)/2)}{\Gamma(1/2)},
\]
Then for $n=2k+1$, with $k\in\numset{N}$
\[
\lexpt[\abs{\stdrv}^{2k+1}]=2^{k}\sqrt{\frac{2}{\pi}}k!.
\]
Therefore, 
\begin{align*}
\expt[(W_{1}+W_{2})^{n}] & =\lexpt[K(\sdr \stdrv_{1})]\\
 & =\sum_{k=0}^{(n-3)/2}C_{k}\sdr^{2k+1}\frac{1}{2}\lexpt[|\stdrv|^{2k+1}]\\
 & =\sqrt{\frac{2}{\pi}}\bigl(\sum_{k=0}^{(n-3)/2}C_{k}\sdr^{2k+1}2^{k-1}k!\bigr).\qedhere
\end{align*}
\end{proof}

\begin{proof}[Proof of \ref{prop:represent-two-semignorm}]
	The representation under semi-sequential independence can be directly
checked based on \ref{thm:represent-n-seqind-semignorm} and \ref{rem:vision-represent}. In the following context, we only consider the case of sequential independence $W_1\seqind W_2$, because $W_1\fullseqind W_2$ will induce the same result by a similar logic to the proof of \ref{thm:represent-n-seqind-semignorm}. 
The basic idea is we need to show that $\expt[\varphi(W_{1},W_{2})]$
can be reached by the linear expectation on the right handside for
some $\myvec{\sigma}\in\myset{L}_{2}^{0}\sdInt$. Then we have 
\[
\expt[\varphi(W_{1},W_{2})]\leq\max_{\myvec{\sigma}\in\myset{L}_{2}^{0}\sdInt}\lexpt_{\lprob}[\varphi(\sigma_{1}\stdrv_{1},\sigma_{2}\stdrv_{2})].
\]
The reverse direction of inequality comes from the fact that $\myset{L}_{2}^{0}\sdInt\subset\myset{L}_{2}\sdInt$
and \ref{thm:represent-n-seqind-semignorm}. 
The logic here is similar to the proof of \ref{thm:represent-n-seqind-semignorm} in sequential-independence case, that is, we only need to record the optimal choice of $(\sigma_1,\sigma_2)$ when evaluating the sublinear expectation in an iterative way. 

For instance, 
when $\varphi(x_{1},x_{2})=(x_1+x_2)^n$, the sublinear expectation can be reached by some $\myvec{\sigma}\in \myset{L}_{2}^{0}\sdInt$ as illustrated in \ref{lem:even-moment-two-semignorm} and \ref{lem:odd-moment-two-semignorm}.
For $\varphi(x_{1},x_{2})=x_{1}^{p}x_{2}^{q}$ with $p,q\in \numset{N}_+$,
\begin{equation}
\label{eq:expt-W1W2-pq}
	\expt[W_{1}^{p}W_{2}^{q}]=\expt[(\max_{\sigma_{2}\in\sdInt}w_{1}^{p}\lexpt[(\sigma_{2}\stdrv_{2})^{q}])].
\end{equation}

Meanwhile, for any $(\sigma_{1},\sigma_{2})\in\myset{L}_{2}^{0}\sdInt$,
let $Y_{i}=\sigma_{i}\stdrv_{i},i=1,2$ and $\lexpt=\lexpt_{\lprob}$
denote the linear expectation. Note that $\sigma_1\in\{\sdl,\sdr\}$,
\[
\sigma_{2}=\sigma_{2}(Y_{1})=(\sigma_{22}-\sigma_{21})\ind{Y_{1}>0}+\sigma_{21},
\]
with $(\sigma_{21},\sigma_{22})\in\{\sdl,\sdr\}^2$.
We also have $Y_{1}\independent\epsilon_{2}$ due to the setup which
is the same as \ref{subsec:represent-semignorm-indep}. Then 
\begin{align*}
\lexpt[Y_{1}^{p}Y_{2}^{q}] & =\lexpt\bigl[Y_{1}^{p}\lexpt[Y_{2}^{q}|Y_{1}]\bigr]\\
 & =\lexpt\bigl[Y_{1}^{p}\lexpt[(\sigma_{2}(Y_{1}))^{q}\stdrv_{2}^{q}|Y_{1}]\bigr]\\
 & =\lexpt[(\sigma_{1}\stdrv_{1})^{p}(\sigma_{2}(Y_{1}))^{q}]\lexpt[\stdrv_{2}^{q}]\\
 & =\sigma_{1}^{p}\lexpt[(\sigma_{22}-\sigma_{21})\ind{Y_{1}>0}+\sigma_{21})^{q}\stdrv_{1}^{p}]\lexpt[\stdrv_{2}^{q}]\\
 & =\sigma_{1}^{p}\lexpt[(\sigma_{22}^{q}-\sigma_{21}^{q})\ind{\stdrv_{1}>0}+\sigma_{21}^{q})\stdrv_{1}^{p}]\lexpt[\stdrv_{2}^{q}].
\end{align*}
Hence, 
\begin{equation}
\lexpt[Y_{1}^{p}Y_{2}^{q}]=\sigma_{1}^{p}\bigl((\sigma_{22}^{q}-\sigma_{21}^{q})\lexpt[\ind{\stdrv_{1}>0}\stdrv_{1}^{p}]+\sigma_{21}^{q}\lexpt[\stdrv_{1}^{p}]\bigr)\lexpt[\stdrv_{2}^{q}].\label{eq:lexpt-w1w2-pq}
\end{equation}

Then we divide our discussions into three cases: a) $q$ is odd, b)
$q$ is even and $p$ is even, c) $q$ is even and $p$ is odd.
When $q$ is odd, the expectation in \ref{eq:expt-W1W2-pq} is equal to $0$ by \ref{lem:poly-omit}.
It can be obviously reached the linear expectation on the right handside
by choosing any $\myvec{\sigma}\in\myset{L}_{2}^{0}\sdInt$ by \ref{eq:lexpt-w1w2-pq}. When
$q$ is even, we can see that the choice of $\sigma_{2}$ depends
on the sign of $w_{1}^{p}$ which further depends on the sign of $w_{1}$
if $p$ is odd (otherwise it is always non-negative). To be specific,
when both $q$ and $p$ are even, 
\begin{align*}
\expt[W_{1}^{p}W_{2}^{q}] & =\lexpt[\stdrv_{2}^{q}]\sdr^{q}\expt[W_{1}^{P}]=\sdr^{p}\sdr^{q}\lexpt[\stdrv_{1}^{p}]\lexpt[\stdrv_{2}^{q}],
\end{align*}
which can be reached by choosing $\sigma_{1}=\sigma_{2}=\sdr$, namely,
$(\sigma_{21},\sigma_{22})=(\sdr,\sdr)$. When $q$ is even and $p$
is odd, we have 
\begin{align*}
\expt[W_{1}^{p}W_{2}^{q}] & =\lexpt[\stdrv_{2}^{q}]\expt[\sdr^{q}(W_{1}^{p})^{+}-\sdl^{q}(W_{1}^{p})^{-}]\\
 & =\lexpt[\stdrv_{2}^{q}]\expt[(\sdr^{p}-\sdl^{q})(W_{1}^{p})^{+}+\sdl^{q}W_{1}^{p}]\\
 & =\lexpt[\stdrv_{2}^{q}]\expt[(\sdr^{p}-\sdl^{q})(W_{1}^{p})^{+}].
\end{align*}
Hence, by \ref{thm:conn-G}, we have 
\[
\expt[W_{1}^{p}W_{2}^{q}]=\begin{cases}
\lexpt[\stdrv_{2}^{q}](\sdr^{p}-\sdl^{q})\lexpt[(\sdr^{p}\stdrv_{1}^{p}))^{+}] & \text{if  }\sdr^{p}\geq\sdl^{q}\\
\lexpt[\stdrv_{2}^{q}](\sdl^{p}-\sdr^{q})\lexpt[-(\sdl^{p}\stdrv_{1}^{p}))^{+}] & \text{if }\sdr^{p}<\sdl^{q}
\end{cases}
\]
It can be reached by choosing $\sigma_{1}$ and $(\sigma_{21},\sigma_{22})$
accordingly in \ref{eq:lexpt-w1w2-pq}. 
Similar logic can be applied to $\varphi(x_{1},x_{2})=cx_{1}^{p}x_{2}^{q}$ 
and also $\varphi(x_{1},x_{2})=(ax_1+bx_2)^n$ where the scaling does not have effects on the form of the optimal choice of $\myvec{\sigma}$.
\end{proof}

\subsection{Proofs in \ref{subsec:semi-G-clt}}
\label{ap:pf-semi-G-clt}


To prepare for the proof, we consider the following function space: 
\begin{itemize}
    \item $C^{k}(\numset{R})$: the space of $k$-times continuously differentiable
functions on $\numset{R}$
	\item $C_{b}(\numset{R})$: the space of bounded and continuous functions on $\numset{R}$, 
	\item $C^{*}(\numset{R})=\{\varphi\in C^{2}(\numset{R}):\varphi''\text{ is bounded and uniformly continuous}\}$.
\end{itemize}
For any $\varphi\in C^{*}(\numset{R})$, since $\varphi''$ is bounded, we have
\[
M\coloneqq\sup_{x\in\numset{R}}\abs{\varphi''(x)}<\infty.
\]

The following \ref{lem:clt-fspace-to-check} shows that we only need to check those $\varphi \in C^{*}(\numset{R})$ to prove \ref{thm:semi-G-CLT}.
 
\begin{lem}
\label{lem:clt-fspace-to-check}
Assume $\expt[\abs{Z_{n}}]<\infty$ and $\expt[\abs{Z}]<\infty$.
Suppose the convergence 
\begin{equation}
\expt[\varphi(Z_{n})]\to\expt[\varphi(Z)],\label{eq:converge-dist}
\end{equation}
holds for any $\varphi\in C^{*}(\numset{R})$. It also holds for $\varphi\in C_{b}(\numset{R})$.
\end{lem}

\begin{proof}[Proof of \ref{lem:clt-fspace-to-check}]
We first consider $\varphi\in C_{b}(\numset{R})$ with a compact support
$S$. 
By the uniform approximation provided by \cite{pursell1967uniform},
for any $a>0$, there exists $\varphi_{a}\in C^{3}(\numset{R})$ with support $S$ such that 
\[
\sup_{x\in \numset{R}} \abs{\varphi(x)-\varphi_{a}(x)}<\frac{a}{2}.
\]
For $k=1,2,3$, since $\varphi_{a}^{(k)}$ is continuous and it is
on a compact support, it must be bounded by $M_{k}$. By mean-value
theorem, for $\delta>0$ and some $\beta\in[0,1]$, we have 
\[
\abs{\varphi_{a}^{(2)}(x)-\varphi_{a}^{(2)}(x+\delta)}\leq\abs{\varphi^{(3)}(x+\beta\delta)}\delta\leq M_{3}\delta.
\]
Thus $\varphi_{a}^{(2)}$ is uniformly continuous and bounded, implying
$\varphi_{a}\in C^{*}(\numset{R})$. In this way, we have 
\begin{align*}
\abs{\expt[\varphi(Z_n)]-\expt[\varphi(Z)]} & \leq\abs{\expt[\varphi(Z_n)]-\expt[\varphi_a(Z_n)]}+\abs{\expt[\varphi(Z)]-\expt[\varphi_a(Z)]}\\
 & +\abs{\expt[\varphi_a(Z_n)]-\expt[\varphi_a(Z)]}\\
 & \leq a+\abs{\expt[\varphi_a(Z_n)]-\expt[\varphi_a(Z)]}.
\end{align*}
Hence, $\limsup_{n\to\infty}\abs{\expt[\varphi(Z_n)]-\expt[\varphi(Z)]}\leq a$. It means that 
\[
0\leq \liminf_{n\to\infty}\abs{\expt[\varphi(Z_n)]-\expt[\varphi(Z)]}\leq 
\limsup_{n\to\infty}\abs{\expt[\varphi(Z_n)]-\expt[\varphi(Z)]}\leq a.
\]
Since $a$ can be arbitrarily small, we have the convergence \ref{eq:converge-dist}
holds. 

Next consider any $\varphi\in C_{b}(\numset{R})$ which is bounded
by $B$. For any $K>0$, it can be decomposed into $\varphi=\varphi_{1}+\varphi_{2}$
where $\varphi_{1}$ has compact support $[-K,K]$ and $\varphi_{2}$
satisfies $\varphi_{2}(x)=0$ if $\abs{x}\leq K$ and for $\abs{x}>K$,
\[
\abs{\varphi_{2}(x)}\leq B\leq\frac{B\abs{x}}{K},
\]
Then we have 
\[
\abs{\expt[\varphi(Z_n)]-\expt[\varphi(Z)]}\leq\abs{\expt[\varphi_{1}(Z_n)]-\expt[\varphi_{1}(Z)]}+\abs{\expt[\varphi_{2}(Z_n)]-\expt[\varphi_{2}(Z)]},
\]
where the first term must converge by our previous argument. Then we
only need to work on the second term that satisfies: 
\[
\abs{\expt[\varphi_{2}(Z_n)]-\expt[\varphi_{2}(Z)]}\leq\frac{B}{K}(\expt[\abs{Z_{n}}]+\expt[\abs{Z}]).
\]
Note that $L\coloneqq \expt[\abs{Z_{n}}]+\expt[\abs{Z}]<\infty$. Then we have $\limsup_{n\to\infty}\abs{\expt[\varphi(Z_n)]-\expt[\varphi(Z)]}\leq \frac{BL}{K}$. 
Since $K$ can be arbitrarily large, we obtain the convergence \ref{eq:converge-dist}. 
\end{proof}

\begin{lem}
\label{lem:delta}
For any $\varphi\in C^*(\numset{R})$, the function $\delta:\numset{R}_{+}\to\numset{R}_{+}$,
defined as
\[
\delta(a)\coloneqq\sup_{\abs{x-y}\leq a}\abs{\varphi''(x)-\varphi''(y)},
\]
must be a bounded and increasing one. It also satisfies $\lim_{a\downarrow0}\delta(a)=0$.
\end{lem}

\begin{proof}[Proofs of \ref{lem:delta}]
The boundedness (and the limit property) can be directly derive from
the boundedness (and uniform continuity) of $\varphi''$. For the
monotone property, for any $0<a\leq b$, since $\{(x,y):\abs{x-y}\leq a\}\subset\{(x,y):\abs{x-y}\leq b\}$,
we must have $\delta(a)\leq\delta(b)$.
\end{proof}


\begin{proof}[Proof of \ref{thm:semi-G-CLT}]
We adapt the the idea of the Linderberg method in a ``leave-one-out''
manner to
the sublinear context. One of the reason that we are able to do such adaptation is the symmetry in semi-$G$-independence: $X_i$ is semi-$G$-independent from $\{X_j,j\neq i\}$. 

Note that $X_{i}=V_{i}\eta_{i}$ with the semi-$G$-independence
then we have 
\[
(V_{1},\dots,V_{n})\seqind(\eta_{1},\dotsc,\eta_{n}).
\]
Then we consider a sequence of classically i.i.d. $\{\stdrv_{i}\}_{i=1}^{n}$
satisfying $\stdrv_{1}\sim\CN(0,1)$ and 
\[
(V_{1},\dots,V_{n})\seqind(\eta_{1},\dotsc,\eta_{n})\seqind(\stdrv_{1},\dotsc,\stdrv_{n}).
\]
For each $n$, consider a triangle array, 
\[
e_{i,n}=\frac{X_{i}}{\sqrt{n}},
\]
and 
\[
S_{n}\coloneqq e_{1,n}+\cdots+e_{n,n}.
\]
For this $n$, consider another triangle array $\{W_{i,n}\}_{i=1}^{n}\coloneqq\{(V_{i}\stdrv_{i})/\sqrt{n}\}_{i=1}^{n}$
which are semi-$G$-version i.i.d. following semi-$G$-normal and
satisfy 
\[
W_{i,n}\eqdistn W_{1,n}\eqdistn\frac{W}{\sqrt{n}}.
\]
Note that here we use the same $V_{i}$ sequence in $W_{i}$. This
setup is important for our proof to overcome the difficulty brought
by the sublinear property of $\expt$ (it also gives some insight about
the role of $\sigma^2$ in the classical central limit theorem compared with $V^2$ in sublinear context). Let
\[
W_{n}\coloneqq W_{1,n}+\cdots+W_{n,n},
\]
then we must have $W_{n}\sim\semiGN(0,\varI)$ (by the stability of
semi-$G$-normal as shown in \ref{semi-GN-eqdistn}); 

Our goal is to show the difference, for any $\varphi \in C^*(\numset{R})$ (recall \ref{lem:clt-fspace-to-check}), as $n\to\infty$,
\begin{equation}
\abs{\expt[\varphi(S_{n})]-\expt[\varphi(W)]}=\abs{\expt[\varphi(S_{n})]-\expt[\varphi(W_{n})]}\to0.\label{eq:diff-goal}
\end{equation}
Consider the following summations: 
\begin{equation}
M_{i,n}=\sum_{j=1}^{i}e_{j,n}+\sum_{j=i+1}^{n}W_{j,n},\label{eq:sum-M}
\end{equation}
and 
\begin{equation}
U_{i,n}=\sum_{j=1}^{i-1}e_{j,n}+\sum_{j=i+1}^{n}W_{j,n},\label{eq:sum-U}
\end{equation}
with the common convention that an empty sum is defined as zero. Note
that $M_{0,n}=W_{n}$ and $M_{n,n}=S_{n}$, then we can transform
the difference in \ref{eq:diff-goal} to the telescoping sum
\begin{align}
\expt[\varphi(S_{n})]-\expt[\varphi(W_{n})] & \leq\expt[\varphi(S_{n})-\varphi(W_{n})]\nonumber \\
 & =\expt\sum_{i=1}^{n}(\varphi(M_{i,n})-\varphi(M_{i-1,n}))\nonumber \\
 & \leq\sum_{i=1}^{n}\expt[\varphi(M_{i,n})-\varphi(M_{i-1,n})].\label{eq:telescope-1}
\end{align}
and 
\begin{equation}
\expt[\varphi(W_{n})]-\expt[\varphi(S_{n})]\leq\sum_{j=1}^{n}\expt[\varphi(M_{n-j,n})-\varphi(M_{n-j+1,n})].\label{eq:telescope-2}
\end{equation}
Then we only need to work on the summand $\expt[\varphi(M_{i,n})-\varphi(M_{i-1,n})]$.
By a Taylor expansion, 
\begin{align*}
\varphi(M_{i,n})-\varphi(M_{i-1,n}) & =\varphi(U_{i,n}+e_{i,n})-\varphi(U_{i,n}+W_{i,n})\\
 & =(e_{i,n}-W_{i,n})\varphi'(U_{i,n})\\
 & +[\frac{1}{2}e_{i,n}^{2}\varphi''(U_{i,n}+\alpha e_{i,n})-\frac{1}{2}W_{i,n}^{2}\varphi''(U_{i,n}+\beta W_{i,n})],\\
 & \eqqcolon(a)+(b)
\end{align*}
for some $\alpha,\beta\in[0,1]$. 

For the first term $(a)$, its sublinear expectation must exist because
the growth of $\varphi'$ is at most linear due to the boundedness
of $\varphi''$. Note that $U_{i,n}$ is the inner product of 
\[
V_{u}=(V_{1},\dotsc,V_{i-1},V_{i+1},\dotsc,V_{n}),
\]
and 
\[
\xi_{u}=(\eta_{1},\dotsc,\eta_{i-1},\stdrv_{i+1},\dotsc,\stdrv_{n}),
\]
with the independence $V_{u}\seqind\xi_{u}$, so we have $e_{i,n}-W_{i,n}(=n^{-1/2}V_{i}(\eta_{i}-\stdrv_{i}))$
and $U_{i,n}$ are semi-$G$-independent. Then we can compute
\begin{align*}
\expt[(e_{i,n}-W_{i,n})\varphi'(U_{i,n})] & =\max_{(v_{i},v_{u})}\lexpt[n^{-1/2}v_{i}(\eta_{i}-\stdrv_{i})\varphi'(v_{u}^{T}\xi_{u})]\\
(\text{classical indep.}) & =\max_{(v_{i},v_{u})}n^{-1/2}v_{i}\underbrace{\lexpt[\eta_{i}-\stdrv_{i}]}_{=0}\lexpt[\varphi'(v_{u}^{T}\xi_{u})]=0.
\end{align*}
Similarly, we have $-\expt[-(e_{i,n}-W_{i,n})\varphi'(U_{i,n})]=0$.
Hence, $(a)$ has certain mean zero. Then we have 
\[
\expt[\varphi(M_{i,n})-\varphi(M_{i-1,n})]=\expt[(b)].
\]

For the second term $(b)$, note that 
\begin{align*}
 & 2\times(b)\\
= & e_{i,n}^{2}[\varphi''(U_{i,n}+\alpha e_{i,n})-\varphi''(U_{i,n})]-W_{i,n}^{2}[\varphi''(U_{i,n}+\beta W_{i,n})-\varphi''(U_{i,n})]+(e_{i,n}^{2}-W_{i,n}^{2})\varphi''(U_{i,n})\\
\eqqcolon & (b)_{1}-(b)_{2}+(b)_{3}.
\end{align*}
For $(b)_{1}$, since $\abs{\alpha e_{i,n}}\leq\abs{e_{i,n}}$, by
recalling the property of $\delta(\cdot)$ (\ref{lem:delta}), we have 
\[
\expt[\abs{(b)_{1}}]\leq\expt[e_{i,n}^{2}\delta(\abs{e_{i,n}})]=\frac{1}{n}\expt[X_{1}^{2}\delta(n^{-1/2}\abs{X_{1}})],
\]
where we use the setup $e_{i,n}=\frac{X_{i}}{\sqrt{n}}$ and $X_{i}\eqdistn X_{1}$.
Similarly, we have 
\[
\expt[\abs{(b)_{2}}]\leq\expt[W_{i,n}^{2}\delta(\abs{W_{i,n}})]=\frac{1}{n}\expt[W^{2}\delta(n^{-1/2}\abs{W})],
\]
where we use the setup $W_{i,n}\eqdistn\frac{W}{\sqrt{n}}$. For $(b)_{3}$,
since $(e_{i,n},W_{i,n})$ and $U_{i,n}$ are semi-$G$-independent, (noting that $e_{i,n}$ and $W_{i,n}$ depend on the same $V_i$,)
we have 
\begin{align*}
\expt[(b)_{3}] & =\max_{(v_{i},v_{u})}\lexpt[n^{-1} v_{i}^{2}(\eta_{i}^{2}-\stdrv_{i}^{2})\varphi''(v_{u}^{T}\xi_{u})]\\
(\text{classical indep.}) & =\max_{(v_{i},v_{u})}n^{-1} v_{i}^{2}\underbrace{\lexpt[\eta_{i}^{2}-\stdrv_{i}^{2}]}_{=0}\lexpt[\varphi''(v_{u}^{T}\xi_{u})]=0,
\end{align*}
where we use the fact that $\lexpt[\eta_{i}^{2}]=\lexpt[\stdrv_{i}^{2}]=1$.
Similarly we have $-\expt[-(b)_{3}]=0$ so $(b)_{3}$ has certain
mean zero. Therefore, we have 
\begin{align*}
\expt[\varphi(M_{i,n})-\varphi(M_{i-1,n})] & =\frac{1}{2}\expt[(b)_{1}-(b)_{2}]\\
 & \leq\frac{1}{2}(\expt[\abs{b}_{1}]+\expt[\abs{b}_{2}])\\
 & =\frac{1}{2n}(\expt[X_{1}^{2}\delta(n^{-1/2}\abs{X_{1}})]+\expt[W^{2}\delta(n^{-1/2}\abs{W})]).
\end{align*}
Meanwhile, if we reverse the role of $\varphi(M_{i,n})$ and $\varphi(M_{i-1,n})$
and let $i=n-j+1$ with $j=1,2,\dotsc,n$, we get
\begin{align*}
\expt[\varphi(M_{n-j,n})-\varphi(M_{n-j+1,n})] & =\expt[\varphi(M_{i-1,n})-\varphi(M_{i,n})]\\
 & =\frac{1}{2}\expt[(b)_{2}-(b)_{1}]\\
 & \leq\frac{1}{2}(\expt[\abs{b}_{2}]+\expt[\abs{b}_{1}])\\
 & =\frac{1}{2n}(\expt[X_{1}^{2}\delta(n^{-1/2}\abs{X_{1}})]+\expt[W^{2}\delta(n^{-1/2}\abs{W})]).
\end{align*}
Hence, by \ref{eq:telescope-1} and \ref{eq:telescope-2}, we
have 
\begin{align*}
\abs{\expt[\varphi(S_{n})]-\expt[\varphi(W)]} & =\abs{\expt[\varphi(S_{n})]-\expt[\varphi(W_{n})]}\\
 & =\max\{\expt[\varphi(S_{n})]-\expt[\varphi(W_{n})],\expt[\varphi(W_{n})]-\expt[\varphi(S_{n})]\}\\
 & \leq\max\{\sum_{i=1}^{n}\expt[\varphi(M_{i,n})-\varphi(M_{i-1,n})],\sum_{j=1}^{n}\expt[\varphi(M_{n-j,n})-\varphi(M_{n-j+1,n})].\}\\
 & \leq\sum_{i=1}^{n}\frac{1}{2n}(\expt[X_{1}^{2}\delta(n^{-1/2}\abs{X_{1}})]+\expt[W^{2}\delta(n^{-1/2}\abs{W})])\\
 & =\frac{1}{2}(\expt[X_{1}^{2}\delta(n^{-1/2}\abs{X_{1}})]+\expt[W^{2}\delta(n^{-1/2}\abs{W})]).
\end{align*}
Note that for any $v_{1}\in\sdInt$, $\abs{v_{1}\eta_{1}}\leq\sdr\abs{\eta}_{1}$,
we have
\begin{align*}
\expt[X_{1}^{2}\delta(n^{-1/2}\abs{X_{1}})] & =\max_{v_{1}\in\sdInt}\lexpt[v_{1}^{2}\eta_{1}^{2}\delta(n^{-1/2}\abs{v_{1}\eta_{1}})]\\
(\text{monotonicity of }\delta) & \leq\max_{v_{1}\in\sdInt}v_{1}^{2}\lexpt[\eta_{1}^{2}\delta(n^{-1/2}\sdr\abs{\eta_{1}})]\\
 & =\sdr^2\lexpt[\eta_{1}^2\delta(n^{-1/2}\sdr\abs{\eta_{1}})].
\end{align*}
By \ref{lem:delta}, we have $\delta(a)\leq2M$ for all $a\in\numset{R}^{+}$
so $\eta_{1}^{2}\delta(n^{-1/2}\sdr\abs{\eta_{1}})\leq2M\eta_{1}^{2}$. 
Meanwhile, we have $\eta_{1}^2\delta(n^{-1/2}\sdr\abs{\eta_{1}})\to0$
(classically) almost surely as $n\to \infty$, then by classical dominance convergence theorem, we have $\lexpt[\eta_{1}^2\delta(n^{-1/2}\sdr\abs{\eta_{1}})]\to0$
implying $\expt[X_{1}^{2}\delta(n^{-1/2}\abs{X_{1}})]\to0$. Similarly
we can show $\expt[W^{2}\delta(n^{-1/2}\abs{W})]\to0$. Finally,
we have 
\[
\abs{\expt[\varphi(S_{n})]-\expt[\varphi(W)]}\to0,
\]
or 
\[
\lim_{n\to\infty}\expt[\varphi(\frac{1}{\sqrt{n}}\sum_{i=1}^{n}X_{i})]=\expt[\varphi(W)].\qedhere
\]
\end{proof}



\bibliographystyle{apalike}
\bibliography{Refs}

\end{document}